\documentclass[12pt]{amsart}
\usepackage[applemac]{inputenc}

\usepackage{fullpage}
\usepackage{
 ulem,
 amsfonts}
\usepackage{amsmath,esint, amssymb, amsthm,upgreek}

 \usepackage{mathtools}

\usepackage[colorlinks=true, pdfstartview=FitV, linkcolor=blue, citecolor=blue,
 urlcolor=blue]{hyperref}
\usepackage{color}

\date{}

\newcommand{\CC}{\mathbb{C}}  
\newcommand{\NN}{\mathbb{N}}  
\newcommand{\RR}{\mathbb{R}}  

\theoremstyle{plain}
\newtheorem{definition}{Definition}

\newtheorem*{gracies}{Acknowledgements}
\newtheorem{theorem}{Theorem}
\newtheorem{proposition}{Proposition}
\newtheorem{lemma}{Lemma}
\newtheorem{remark}{Remark}
\newtheorem*{remarks}{Remarks}
\newtheorem{coro}{Corollary}

\usepackage[textmath,displaymath,floats,graphics,delayed]{preview}
  \title[]{On the V-states for the generalized quasi-geostrophic equations}
\author[]{ZINEB HASSAINIA}
\address{IRMAR, Universit\'e de Rennes 1 \\ Campus de Beaulieu \\  35 042 Rennes cedex, France}
 \email{zineb.hassainia@univ-rennes1.fr}
 \author[T. Hmidi]{Taoufik Hmidi}
\address{IRMAR, Universit\'e de Rennes 1\\ Campus de
Beaulieu\\ 35~042 Rennes cedex\\ France}
\email{thmidi@univ-rennes1.fr}
\begin{document}
\begin{abstract}
We prove the existence of the V-states for the generalized inviscid SQG equations with $\alpha\in ]0,1[.$ These structures  are special  rotating simply connected patches with $m-$ fold symmetry bifurcating from the trivial solution at some explicit values of the angular velocity.  This produces, inter alia,  an infinite family of non stationary  global solutions with uniqueness. 
\end{abstract}

\newpage 

\maketitle{}
\tableofcontents

\section{Introduction}
In this paper we  shall investigate  some special structures of the vortical motions  for the   generalized inviscid  surface quasi-geostrophic  equation arising in fluid dynamics.  This model describes the evolution  of the potential temperature $\theta$ by the transport  equation,
\begin{equation}\label{E}
\left\{ \begin{array}{ll}
\partial_{t}\theta+u\cdot\nabla\theta=0,\quad(t,x)\in\RR_+\times\RR^2, &\\
u=-\nabla^\perp(-\Delta)^{-1+\frac{\alpha}{2}}\theta,\\
\theta_{|t=0}=\theta_0.
\end{array} \right.
\end{equation}
Here $u$ refers to the velocity field, $\nabla^\perp=(-\partial_2,\partial_1)$  and $\alpha$ is a real parameter taken in the interval  $[0,1[$. 
The operator  $(-\Delta)^{-1+\frac{\alpha}{2}}$ is of convolution type and is defined by 
\begin{equation}\label{Integ1}
(-\Delta)^{-1+\frac{\alpha}{2}} \theta(x)=\frac{C_\alpha}{2\pi}{\int}_{\RR^2}\frac{\theta(y)}{\vert x-y\vert^\alpha}dy
\end{equation}
with  $C_\alpha=\frac{\Gamma(\alpha/2)}{2^{1-\alpha}\Gamma(\frac{2-\alpha}{2})}$. This model was proposed by C\'ordoba { et al.} in \cite{C-F-M-R} as an interpolation between Euler equations and the surface quasi-geostrophic model, hereafter denoted by SQG,  corresponding to $\alpha=0$ and $\alpha=1$, respectively. The SQG equation was used  by Juckes \cite{Juk} and Held {et al.} \cite{Held} as a concise model of   the atmosphere  circulation  near the tropopause. It  was also developed  by Lapeyre and Klein \cite{Lap} to describe the  ocean dynamics in  the upper layers. We note  that there is a strong  mathematical and physical analogy with the three-dimensional incompressible Euler equations, and it can be viewed as a simplified model for that system; see \cite{C-M-T} for details. 

  The local well-posedness of classical solutions can be performed in various function spaces. For instance, this was implemented  in the framework of Sobolev space  \cite{C-C-C-G-W} by using the commutator theory. However, it is so delicate to extend  the  Yudovich theory of weak solutions known for the two-dimensional Euler equations \cite{Y1} to  the case $\alpha>0$ because the velocity is in general  below the Lipschitz class. Nonetheless, one can say more about this issue for some special class of concentrated vortices. More precisely,  when the initial datum has a vortex patch structure, that is, $\theta_0(x)=\chi_D$ is  the characteristic function of a bounded simply connected smooth domain $D$, then there is a unique local solution in the patch form $\theta(t)=\chi_{D_t}. $ In this case,  the boundary   motion of the domain $D_t$  is described by the contour dynamics formulation; see the papers \cite{G,R}.  The global persistence of the boundary regularity is only known for $\alpha=0$  according to the result of  Chemin \cite{Ch}. For $\alpha>0$ there are some numerical simulations showing the singularity formation  in finite  time, see for instance \cite{C-F-M-R}. 
  
The technique of contour dynamics was originally devised by Zabusky et al. \cite{Zab1} and has found many applications in the study of two-dimensional flows. We shall use  this technique to track the boundary motion of the patch for the generalized SQG equation.  According to Green formula one can recover the velocity   from the boundary through the formula,
\begin{equation}\label{veltu}
u(t,x)=\frac{C_\alpha}{2\pi}{\int}_{\partial D_t}\frac{1}{\vert x-\xi\vert^\alpha}d\xi
\end{equation}
where $d\xi$ denotes the complex integration over the positively oriented curve $\partial D_t$. To write down the equation of the boundary, one can use for instance the Lagrangian \mbox{parametrization $\gamma_t:[0,2\pi]\to \CC$,} given by the nonlinear ode,
\begin{equation*}\label{CD}
\left\{ \begin{array}{ll}
\partial_t\gamma(t,\sigma)=u(t,\gamma(t,\sigma)),&\\
\gamma(0,\sigma)=\gamma_0(\sigma)
\end{array} \right.
\end{equation*}
where $\gamma_0$ is a periodic smooth parametrization of the initial boundary and consequently  the contour dynamics equation becomes
\begin{equation}\label{veloc0}
\partial_t\gamma(t,\sigma)=\frac{C_\alpha}{2\pi}{\int}_0^{2\pi}\frac{\partial_s\gamma(t,s)}{\vert \gamma(t,\sigma)-\gamma(t,s)\vert^\alpha}ds.
\end{equation}

The main objective  of this paper is to focus on some special vortices, called V-states  or  rotating patches, whose dynamics is described by a rigid body transformation. The  problem  consists in finding some domains $D$ subject to a uniform rotation around their  centers of mass. In which  case  the support of the patch  $D_t$ does not change its shape and is given by $D_t= {\bf R}_{x_0,\Omega t}D$, where $ {\bf R}_{x_0,\Omega t}$  stands for  the planar rotation with center $x_0$ and angle $\Omega t.$ The parameter $\Omega$ is called the angular velocity of the rotating domain. 

This problem was investigated first  for the two-dimensional  Euler equations ($\alpha=0$) a long time  ago  and still a subject of intensive research combining analytical and numerical studies. It is worthy noting that  explicit non trivial  rotating patches are known in the literature and  goes back to  Kirchhoff \cite{Kirc} who discovered that an ellipse of semi-axes $a$ and $b$  is subject to a perpetual rotation with uniform angular velocity $\Omega = ab/(a+b)^2$; see for
instance \mbox{\cite[p. 304]{BM}} and \cite [p. 232]{L}. In the seventies of the last century, Deem and Zabusky \cite{DZ} wrote an
equation for the V-states and gave  partial numerical solutions. They put in evidence the existence of the 
V-states with $m$-fold symmetry for each integer $m \geq 2$ and in this countable cascade the case $m=2$ corresponds to the known  Kirchhoff's ellipses. Recall that a 
domain is said  $m$-fold symmetric if it has the same group invariance of a regular polygon with $m$ sides. This means that the domain is invariant by the action of the dihedral group $\textnormal{D}_m$.  At each frequency $m$ these V-states can be seen as a continuous deformation of the disc with respect to a hidden  bifurcation parameter corresponding to the angular velocity. An analytical  proof was given by 
Burbea in \cite{B} and his approach consists in writing the problem with the conformal mapping of the domain and to look at the  non trivial solutions  by using the technique of the bifurcation theory. Actually, Burbea's proof is not completely rigorous and one can find a complete one  \mbox{in \cite{HMV}.} In this latter  paper Burbea's approach was revisited  with  more details and explanations. We also studied  the boundary regularity of the V-states and  showed  that they are of class   $C^\infty$ and convex  close to the disc. 

The formulation of the rotating patches can be done in several ways requiring different levels of regularity for the solution. We shall give  here a short glimpse with an emphasis on two different approaches. The first one  uses the elliptic equation governing the stream function $\psi$ associated to the domain $D$ of the initial patch. As to the second approach,  it uses the conformal  parametrization of the boundary combined with the contour dynamics formulation.  To be more precise, recall that the function $\psi$ is defined by the Newtonian potential through the  formula,
$$
\psi(x)=\frac{1}{2\pi}\int_{D}\log|x-y|\, dy,\quad \Delta \psi= \chi_{D}.
$$ 
Note that a  patch with a smooth boundary rotates uniformly around its center, which can be  taken equal to zero, means that in its own  frame the boundary  is stationary. In other words,  the relative stream function $x\mapsto \psi(x)-\frac12\Omega|x|^2$ should be  constant on the boundary and therefore we get the equation
\begin{equation}\label{Ell1}
\frac{1}{2\pi}\int_{D}\log|x-y|\,dy-\frac12\Omega|x|^2=\mu,\quad\forall x\in \partial D,
\end{equation} 
with $\mu$  a constant. By virtue of  this equation, the domains $D$ are in fact defined through  a strong  interaction between the Newtonian and the quadratic potentials. The issue depends heavily on the sign of $\Omega$. To fix the terminology, we say that the potential is repulsive when $\Omega\le 0$  and attractive when it has an  opposite sign. It seems that the situation in the repulsive case $\Omega\le0$ is trivial in the sense that only the discs are solutions of the rotating patch problem. This means that all the V-states must rotate counterclockwise. This result is the subject of a work in progress  by the second author \cite{Hmidi}. The proof relies on  the moving plane method which allows to show that any solution of \eqref{Ell1} must be radial with respect to some specific point, which is the center of mass of the domain $D$,  and  is strictly monotone. In the attractive case $\Omega>0$, the interaction between the potentials is more fruitful and leads to infinite  nontrivial solutions called the V-states as we have already mentioned. We point out that   Burbea shows that  for each frequency $m\geq2$ the V-states $V_m$ can be assimilated to  a bifurcating curve from the disc at the angular velocity $\Omega_m=\frac{m-1}{2m}$. His idea is to use the conformal mapping parametrization $\phi:\mathbb{D}^c \to D^c$ which satisfies  the nonlinear integral equation
\begin{eqnarray}\label{vortexd1}
\nonumber F\big(\Omega, \phi(w)\big)&\triangleq&\textnormal{Im}\bigg\{\Big((1-2\Omega)\overline{\phi(w)}-\frac{1}{2i\pi}\mathop{{\int}}_\mathbb{T}\frac{\overline{\phi(\tau)-\phi(w)}}{\phi(\tau)-\phi(w)}{\phi^\prime(\tau)d\tau}\Big){w}\,{{\phi^\prime}(w)}\bigg\}\\
&=&0,\quad \forall w\in \mathbb{T},
\end{eqnarray}
where  $\mathbb{D}$ denotes the open unit disc and $\mathbb{T}$ its boundary. Now we observe that $F(\Omega,\hbox{Id})=0$ and thus we may try to find non trivial solutions by using the bifurcation theory. For this end Burbea computes  the linearized operator of $F$ around this solution  and shows that it has a nontrivial kernel if and only if $\Omega\in \{\Omega_m, m\geq2\}$. In  this case $\partial_f F(\Omega, \hbox{Id})$ is a Fredholm operator with  one-dimensional kernel. Consequently, one may  apply the bifurcation theory through for  instance Crandall-Rabinowitz theorem. This allows   to prove the existence of non trivial branch of solutions emerging from the trivial one at each frequency  level $\Omega_m$.

One cannot escape mentioning that other explicit vortex solutions are discovered in the literature for the incompressible Euler equations in the presence of an  external shear flow; see for instance \cite{Chapl,Kida,Neu}. A general review about vortex dynamics can be found in the papers  \cite{A,New}. Another closely related subject is to conduct a similar study for the patches with multiple interfaces which is inherently complicated due to the strong interaction between the interfaces. In this context,  Flierl and Polvani \cite{Flierl}  proved that confocal ellipses with some compatibility relations rotate as a rigid body motion. Recently, we developed a complete characterization of rotating patches with two interfaces provided one of them is prescribed in the ellipses class.

In this paper, we shall address the same problem for the generalized SQG equations and look for the existence of the V-states The question was raised by Diego C\'ordoba and was the initial motivation for this work. As we shall see later in Proposition \ref{prop-bound}, the equation \eqref{vortexd1}  becomes  
\begin{equation*}
F_\alpha\big(\Omega, \phi(w)\big)\triangleq\textnormal{Im}\bigg\{\Big(\Omega\phi(w)-\frac{C_\alpha}{2i\pi}\mathop{{\int}}_\mathbb{T}\frac{\phi^\prime(\tau)}{\vert \phi(w)-\phi(\tau)\vert^\alpha}d\tau\Big)\overline{w}\,{\overline{\phi^\prime}(w)}\bigg\}=0,\quad \forall w\in \mathbb{T},
\end{equation*}
with $\displaystyle{C_\alpha=\frac{\Gamma(\frac\alpha2)}{2^{1-\alpha}\Gamma(\frac{2-\alpha}{2})}}$. Note that the  structure of the singular  nonlinear part is different \mbox{from \eqref{vortexd1}.} Indeed, the singular kernel is not algebraic with respect to the conformal mapping which is holomorphic outside the unit disc. This property is profoundly important for Euler equations because it  yields  at different levels of the analysis, especially in the spectral study, to simple computations through Residue Theorem.  Another disadvantage of the kernel structure concerns the computations of the regularity of the functional $F_\alpha$ which are  heavy and more involved.

 The main contribution of this paper is to give  a positive answer for the existence of the V-states when $\alpha\in ]0,1[.$  For the sake of clarity we shall now give an elementary statement   and a complete  one is postponed to Theorem \ref{thmV2}.
 \begin{theorem}\label{thmV1}
Let $\alpha\in]0,1[$ and $m\in \NN^\star\backslash\{1\}$. Then, there exists a family of  m-fold symmetric  V-states $(V_m)_{m\geq2}$ for the equation \eqref{E}. Moreover,  for each $m\geq2$ the curve $V_m$ bifurcates from the trivial solution $\theta_0=\chi_{\mathbb{D}}$ at the angular velocity
$$
\Omega_m^\alpha\triangleq\frac{\Gamma(1-\alpha)}{2^{1-\alpha}\Gamma^2(1-\frac\alpha2)}\bigg(\frac{\Gamma(1+\frac\alpha2)}{\Gamma(2-\frac\alpha2)}-\frac{\Gamma(m+\frac\alpha2)}{\Gamma(m+1-\frac\alpha2)}\bigg),
$$
where $\Gamma$ denotes the gamma function.

In addition, the boundary of the  V-states belongs to the  H\"{o}lder class $C^{2-\alpha}.$
\end{theorem}
The proof of this theorem will be done  in the spirit of the incompressible Euler equations  by using the bifurcation theory through Crandall-Rabinowitz Theorem.  In the framework of this theory one should understand the structure of the linearized operator around the trivial solution $\hbox{Id}$
 and  identify the range of $\Omega $ where this operator is not invertible. More precisely, we should determine where this operator  belongs to the  Fredholm class   with zero index and possesses a simple kernel. By using some tricky integral formulae summarized in \mbox{Lemma \ref{lem}} one finds:  for $\displaystyle{h(w)=\sum_{n\in\NN} b_n \overline{w}^n}$
 
$$
\partial_\phi F_\alpha(\Omega, \hbox{Id})h(w)=\frac{1}{2}b_0\Omega\, i\big(w-\overline{w}\big)+\frac i2{\sum}_{n\geq1} (n+1)\big(\Omega-\Omega_{n+1}^\alpha) b_n \Big(w^{n+1}-\overline{w}^{n+1}\Big).
$$
Consequently, the linearized operator acts as a Fourier multiplier in the phase space. It behaves as  a differential operator of order one   because $\sup_{n}\Omega_n^\alpha<\infty$. Afterwards,  we prove that this  operator sends $C^{2-\alpha}$ to $C^{1-\alpha}$ and fulfills  the required  assumptions of Crandall-Rabinowitz Theorem: it is of Fredholm type with zero index and satisfies the transversality assumption. This latter one means that when we look for the linearized operator with $\Omega$ close to $\Omega_m^\alpha$, then the  eigenvalue which is close to zero (it depends on $\Omega$) must cross the real axis with non zero velocity at the value $\Omega=\Omega_m^\alpha$.

Next, we shall make few  comments about the statement of the main theorem.
\begin{remarks}

${\bf{1)}}$ For the incompressible Euler equations the dilation has no effects on the angular vorticity of the V-states. However, this property fails for the generalized SQG model because we change the homogeneity of the equation. As we shall see later in Proposition $\ref{pro-dila}$ the angular velocity depends on the inverse of  the dilation parameter raised to the power $\alpha$. This means that we can find small patches rotating  quickly and also  big ones rotating very slowly. In addition, the bifurcation set $\{\Omega_m^\alpha, m\geq2\}$ introduced in Theorem $\ref{thmV1}$ concerns only the bifurcation from the unit disc. However, to get a bifurcation from a disc of radius $r$ we have to scale this set as follows $\{r^{-\alpha}\Omega_m^\alpha, m\geq2\}$.
\vspace{0,3cm}

${\bf{2)}}$  For the SQG equation corresponding to $\alpha=1$ the situation is more delicate  as we shall discuss later in the end of the paper. Indeed, one can modify the function $F_\alpha$ in order to get a less singular kernel but  we note a regularity loss  for the linearized operator.  This appears more clearly when we compute the linearized operator which is given by
$$
\partial_\phi F_1(\Omega, \textnormal{Id})h(w)=\frac{1}{2}b_0\Omega\, i\big(w-\overline{w}\big)+\frac i2\sum_{n\geq1} (n+1)\big(\Omega-\Omega_{n+1}^1) b_n\Big(w^{n+1}- \overline{w}^{n+1}\Big), 
$$
with
\begin{equation}\label{dis1}
\Omega_n^1=\frac{2}{\pi}\sum_{k=1}^{n-1}\frac{1}{2k+1}\cdot
\end{equation}
We see that this operator acts as a Fourier multiplier with an additional logarithmic growth compared to the case $\alpha\in [0,1[$. 
As a consequence, this operator does not send $C^{1+\varepsilon}$ to $C^\varepsilon$ and it seems complicated to find suitable function spaces $X$ and $Y$ such that  Crandall-Rabinowitz Theorem can be applied. More discussion will be brought forward  the end of this paper; see Section $\ref{Limitc}.$ We also mention that the  preceding dispersion relation was computed  formally \mbox{in \cite{A-H}} by using Bessel functions. In Section $\ref{Limitc}$ we shall give another proof of this relation.
\vspace{0,3cm}

${\bf{3)}}$ The  boundary  of the rotating patches belongs  to H\"older space $C^{2-\alpha}$. For $\alpha=0$, we get better result as  it was shown \mbox{in \cite{HMV};} the boundary  is $C^\infty$ and convex when the V-states are close to the circle. The proof in this particular case uses in a deep way the algebraic structure of the kernel according to some recurrence formulae. It is not clear whether this approach can be implemented for the generalized SQG equation but we do  believe that the boundary is also $C^\infty$.
\vspace{0,3cm}

${\bf{4)}}$ The global existence of non stationary solutions for \eqref{E} is not known for $\alpha>0$. The V-states offer a suitable class of initial data with global existence  because they generate periodic solutions in time.
\vspace{0,3cm}

${\bf{5)}}$ As we shall see later in Lemma $\ref{lemz1}$, there is  continuity  of the spectrum $\Omega_m^\alpha$ with respect \mbox{to $\alpha$. }This means that,
$$
\lim_{\alpha\to0}\Omega_m^\alpha=\frac{m-1}{2m}\quad\hbox{and}\quad \lim_{\alpha\to1}\Omega_m^\alpha=\Omega_m^1,
$$
where $\Omega_m^1$ is defined in \eqref{dis1}.
\end{remarks}

The paper is organized as follows. In the next section we shall fix some notation. In \mbox{Section $3$,} we discuss some general properties of  the V-states. In Section $4$, we shall introduce and review  some background material on the bifurcation theory and singular integrals. In Section $5$, we will study the elliptic patches and show that they never rotate. This was recently proved \mbox{in \cite{Cor}} and we intend to give another proof  by using complex analysis formulation. Section $6$ is devoted to a general statement of Theorem \ref{thmV1}. The proof of Theorem \ref{thmV2} will be discussed in Sections $7, 8 $ and $9$. Last, in Section $10$  we will pay a special attention to the  SQG model corresponding to the limit case $\alpha=1$. We shall reformulate the boundary equation in order to kill the violent singularity of the kernel. In this case we give a complete description of the linearized operator and the dispersion relation. However we  are not able to give a complete proof of  the bifurcation of the V-states which should require a slightly different mathematical machinery than does the sub-critical case $\alpha\in [0,1[.$

\section{Notation}
In this section  we shall  fix some notation that will  be frequently used along this paper.
\begin{enumerate}
\item[$\bullet$] We denote by C any positive constant that may change from line to line.
\item[$\bullet$] For any positive real numbers $A$ and $B$, the notation $A \lesssim B$ means that there exists a
positive constant $C$ independent of $A$ and $B$ such that $A\leq CB$.
\item[$\bullet$] We denote by $\mathbb{D}$ the unit disc. Its boundary, the unit circle, is denoted by  $\mathbb{T}$. 
\item[$\bullet$] Let $f:\mathbb{T}\to \CC$ be a continuous function. We define  its  mean value by,
$$
\fint_{\mathbb{T}} f(\tau)d\tau\triangleq \frac{1}{2i\pi}\int_{\mathbb{T}}  f(\tau)d\tau,
$$
where $d\tau$ stands for the complex integration.
\item[$\bullet$] Let $X$ and $Y$ be two normed spaces. We denote by $\mathcal{L}(X,Y)$ the space of  all continuous linear maps $T: X\to Y$ endowed with its usual strong topology. 
\item[$\bullet$] For a linear operator $T:X\to Y,$ we denote by $N(T)$ and $R(T)$  the kernel and the range of $T$, respectively.  
 \item[$\bullet$] If $Y$ is a vector space and $R$ is a subspace, then $Y/ R$ denotes the quotient space.

\end{enumerate}

\section{Preliminaries on the V-states}
In this introductory section we will focus on  some general results on the rotating patches called also V-states  according to Deem and Zabusky terminology. These results were proved in \cite{HMV2} for Euler equations and will be likewise extended to   the SQG \mbox{model \eqref{E}.}
\subsection{General facts}
Now we intend to fix some vocabulary and prove in particular that the center of rotation of any V-state should coincide with its center of mass. We shall also deal with the effects of the dilation of the geometry on the angular velocity of the rotating patches.
\begin{definition}\label{defa1}
Let $D_0$ be a simply connected domain in the plane with smooth boundary. We say that $\theta_0={\bf{1}}_{D_0}$ is a rotating patch if the associated solution of \eqref{E} is given by 
$$
\theta(t,x)={\bf{1}}_{D_t}\quad \hbox{with}\quad  D_t=\bold{R}_{x_0,\varphi(t)}D_0.$$
Here we denote by $\bold{R}_{x_0,\varphi(t) }$  the planar rotation of
center  $x_0$ and \mbox{angle $\varphi(t)$}.  In addition, we assume that
the function $t\mapsto\varphi(t)$ is smooth and non-constant.

\end{definition}
The velocity  dynamics in the framework of rotating patches is
described as follows.
\begin{proposition}\label{vel-eq}
Let $\theta_0$ be a rotating patch as in Definition
${\ref{defa1}}$. Then the velocity $u(t)$ can be recovered from its
initial value $u_0$ according to the formula
$$
u(t,x)=\bold{R}_{x_0,\varphi(t)  } u_0(\bold{R}_{x_0,-\varphi(t)}x).
$$
\end{proposition}
\begin{proof}
We shall use the formula
$$
-(-\Delta)^{1-\frac\alpha2} u(t,x)=\nabla^\perp\theta(t,x).
$$
Performing some algebraic computations we get
\begin{eqnarray*}
\nabla^\perp\theta(t,x)&=&\bold{R}_{x_0,\varphi(t})\nabla^\perp\theta_0(\bold{R}_{x_0,-\varphi(t)}x)\\
&=&- \bold{R}_{x_0,\varphi(t)}(-\Delta)^{1-\frac\alpha 2} v_0(\bold{R}_{x_0,-\varphi(t)}x)\\
&=&-(-\Delta)^{1-\frac\alpha 2} \big(\bold{R}_{x_0,\varphi(t)}
v_0(\bold{R}_{x_0,-\varphi(t)}x)\big).
\end{eqnarray*}
Here we have used the commutation between the operator $(-\Delta)^{1-\frac\alpha 2}$ and the rotation  transformations which can be checked easily from the integral representation of the fractional Laplacian.
Therefore, the result follows from a uniqueness argument.  \end{proof}
Now, we will discuss a special result concerning the evolution of the center of mass of the  patch $\theta(t)={\bf{1}}_{D_t}$, defined by
\begin{eqnarray*}
X(t)&=&\frac{1}{|D_t|}\int_{D_t} x\, dx=\frac{1}{|D_0|}\int_{D_t} x\, dx.
\end{eqnarray*}
We have used the fact that the volume of a patch is an invariant of the motion since the velocity is divergence free.
Next, we prove that the center  of mass is stationary for any  patch solution of  \eqref{E}. This is known for Euler equation and we shall give here a similar  proof.
\begin{proposition}\label{cons1}
Let $\theta(t)={\bf{1}}_{D_t}$ be a  solution of  \eqref{E} then the center of mass is fixed, that is $$
 X(t)=X(0).
$$
\end{proposition}
\begin{proof}
The invariance of the center of mass follows from the constancy of
the functions
$$
f_j(t)\triangleq \int_{\RR^2}x_j\,\theta(t,x)\,dx, \quad j=1,2.
$$
Differentiating this function with respect to the time
variable combined with the  \mbox{equation \eqref{E}} and integration by parts  yields
\begin{eqnarray*}
f_j^\prime(t)&=&\int_{\RR^2}x_j\,\partial_t\theta(t,x)\,dx\\
&=&-\int_{\RR^2}x_j\, (u\cdot\nabla \theta)(t,x)\,dx\\
&=&\int_{\RR^2}u^j(t,x)\, \theta(t,x)\,dx.
\end{eqnarray*}
Using the relation between $u$ and $\theta$ and integrating once again by parts we get
\begin{eqnarray*}
f_1^\prime(t)
&=&-\int_{\RR^2}\theta \,\nabla^\perp(-\Delta)^{-1+\frac\alpha2}\theta\,dx\\
&=&-\int_{\RR^2}\big\{(-\Delta)^{-\frac12+\frac\alpha4}\theta\big\} \,\nabla^\perp\big\{(-\Delta)^{-\frac12+\frac\alpha4}\theta\big\}\,dx\\
&=&0.
\end{eqnarray*}
This completes the proof of the desired result.
\end{proof}
In consequence, we obtain the following result.
\begin{coro}\label{centremass}
Let $\theta_0={\bf{1}}_{D_0}$ be a rotating  patch center
 around some  point $x_0$. Then necescentersarily $x_0$ is
the center of mass of the domain $D_0.$
\end{coro}
\begin{proof}
By a change of variables
\begin{eqnarray*}
X(t)&=&\frac{1}{|D_0|}\int_{\RR^2}x\,\theta_0(\bold{R}_{x_0,-\varphi(t)} x)dx\\
&=&\frac{1}{|D_0|}\int_{\RR^2}(\bold{R}_{x_0,\varphi(t)}x)\theta_0(x)dx\\
&=&\frac{1}{|D_0|}\bold{R}_{x_0,\varphi(t)}\left(\int_{\RR^2}x\theta_0(x)dx\right)\\
&=&\bold{R}_{x_0,\varphi(t)} X(0).
\end{eqnarray*}
Since $X(t)=X(0)$ by Proposition \eqref{cons1},  $X(0)$ is fixed by
the rotation and thus $X(0)=x_0$, as claimed.
\end{proof}
Next we shall discuss how the dilation affects the angular velocity. We point out  that the following  notation that we shall use $D_\lambda=\lambda D$ means a dilation of the domain $D$ with respect to its center of mass.
\begin{proposition}\label{pro-dila}
Let $\theta_0={\bf{1}}_D$ be a rotating patch with constant angular velocity $\Omega.$ Let $\lambda>0$ and denote by $D_\lambda=\lambda D$. Then $D_\lambda$ is also a rotating patch with angular velocity $\Omega_\lambda=\frac{\Omega}{\lambda^{\alpha}}\cdot$
\end{proposition}
\begin{proof}
Without loss of generality we can assume that the center of rotation is the origin. Then according to the equation \eqref{model0} we have
$$
\Omega\textnormal{Re}\big\{z\overline{z'}\big\}=\textnormal{Im}\Big\{\frac{C_\alpha}{2\pi}\int_{\partial D}\frac{d\zeta}{\vert z-\zeta\vert^\alpha}\overline{z'}\Big\},\quad \forall\, z\in \partial D.
$$
Let $\ z\in D$ and $\tau=\lambda z, $ then multiplying the preceding equation by $\lambda^2$ and using the change of variables $w=\lambda \zeta$ we get
$$
\big(\lambda^{-\alpha}\Omega\big)\textnormal{Re}\big\{\tau\overline{\tau'}\big\}= \textnormal{Im}\Big\{\frac{C_\alpha}{2\pi}\int_{\partial D_\lambda}\frac{dw}{\vert \tau-w\vert^\alpha}\overline{\tau'}\Big\},\quad \forall\, \tau\in \partial D_\lambda.
$$
This shows that $D_\lambda$ rotates with the angular velocity $\lambda^{-\alpha}\Omega$ as it is claimed.
\end{proof}
\subsection{Boundary equation }

Before proceeding further with the consideration of the V-states,  we shall recall Riemann mapping theorem which  is one of the most important results in complex analysis. To restate this result we shall recall the definition of {\it simply connected} domains. Let $\widehat{\mathbb{C}}\triangleq \mathbb{C}\cup\{\infty\}$ denote the Riemann sphere. We  say that a domain $\Omega\subset \widehat{\mathbb{C}}$ is {\it simply connected} if the set $ \widehat{\mathbb{C}}\backslash \Omega$ is connected.

{\it Riemann Mapping Theorem.} Let $\mathbb{D}$ denote the unit open ball and  $\Omega\subset \CC$ be a simply connected bounded domain. Then there is a unique bi-holomorphic map called also conformal map,  $\Phi: \CC\backslash\overline{\mathbb{D}}\to  \CC\backslash\overline{\Omega}$ taking the form
$$
\Phi(z)= az+\sum_{n\in\NN} \frac{a_n}{z^n} \quad \textnormal{with}\quad a>0.
$$
 In this theorem the regularity of the boundary has no effect regarding the existence of the conformal mapping  but  it contributes in the boundary behavior of the conformal mapping, see for instance \cite{P,WS}. Here, we shall recall the following result. 
\vspace{0,2cm}


{\it  Kellogg-Warschawski's theorem.} It can be found in  \cite{WS} or in  \cite[Theorem 3.6]{P}.  It asserts that if the
boundary $\Phi(\mathbb{T})$ is a Jordan curve of class $C^{n+1+\beta},$ with
$n\in \NN$  and $ 0<\beta<1$, then the conformal map
$\Phi:\CC\backslash\overline{\mathbb{D}}\to  \CC\backslash\overline{\Omega}$ has a continuous
extension to $\CC\backslash\mathbb{D}$  which is of \mbox{class
$C^{n+1+\beta}.$}

Next, we shall write down the equation governing the boundary of the V-states; it is highly nonlinear and non local as the next proposition shows.  
\begin{proposition}\label{prop-bound}
Let $\alpha\in ]0,1[$, $D_0$ be a smooth simply connected domain and $D_t=R_{x_0, \varphi(t)} D_0$ be a V-state of the model \eqref{E}. Then, the following claims hold true.
\begin{enumerate}
\item The point $x_0$ is the center of mass of $D_0$ and $\dot\varphi(t)=\Omega$ is constant.
\item  Assume that $x_0=0$ and let $\phi:\mathbb{D}^c\to D_0^c$ be the conformal mapping, then
\begin{equation}\label{Rotaeq}
\textnormal{Im}\Bigg\{\bigg(\Omega\,\phi(w)-{C_\alpha}\mathop{{\fint}}_\mathbb{T}\frac{\phi'(\tau)d\tau}{\vert \phi(w)-\phi(\tau)\vert^\alpha}\bigg)\overline{w}\,{\overline{\phi'(w)}}\Bigg\}=0,\quad \forall w\in \mathbb{T},
\end{equation}
with $\displaystyle{C_\alpha=\frac{\Gamma(\alpha/2)}{2^{1-\alpha}\Gamma(\frac{2-\alpha}{2})}}$.
\end{enumerate}
\end{proposition}
\begin{proof}

$\bf{(1)}$  The first claim was proved in Proposition \ref{centremass} and so it remains to check that the angular velocity is constant. For this aim we shall start with writing the boundary  equation of a V-state.
Loosely speaking, the boundary  $\partial D_{t}$ is a material surface and 
there is no flux matter across it.  In other words, it  is transpocenterrted
by  the flow $\psi(t)$ defined in the next few lines. For a smooth initial boundary, say of class $C^1,$  there exists a function $\varphi_0: \RR^2\to\RR$  of class  $C^1$  such that 
\begin{equation*}\label{gammaj}
\partial D_0 = \Big\{x \in \mathbb{R}^2;\varphi_0(x) = 0 \Big\},
\end{equation*}
with the additional constraints: $\forall x\in \partial D_0,\,\,\, \nabla \varphi_0(x) \neq 0,\quad $
 $$
 \varphi_0 < 0\,\,\,\, \textnormal{on}\,\, \,\, D_0\quad \textnormal{and}\qquad \varphi_0 > 0\,\,\,\, \textnormal{on}\,\, \,\,  \mathbb{R}^2
\setminus \overline{D_0}.
$$  One says in this case that $\varphi_0$ is a defining
function for $\partial D_0.$ Set
\begin{equation*}\label{fijt}
F(t,x)= \varphi_0(\psi^{-1}(t,x)),
\end{equation*}
where $\psi$ is the flow associated to the velocity $u$ and given by the integral equation
$$
\psi(t,x)=x+\int_0^tu\big(\tau,\psi(\tau,x)\big)d\tau.
$$
It follows that the maps  $x \mapsto F(t,x)$
is a defining function for $\partial D_t = \psi(t,\partial D_0)$ and satisfies the transport equation
 $$
 \partial_t F+u\cdot\nabla F=0.
 $$
 Now, let $\sigma\in [0,2\pi]\mapsto \gamma_t(\sigma)$ be a parametrization of  $\partial D_t,$ continuously differentiable in $t$,
  and let $\vec{n}_t$ be the unit outward normal vector to $\partial D_t.$ Differentiating
the equation $F(t,\gamma_t(\sigma))=0$ with respect to $t$ yields
 $$
 \partial_t F+\partial_t\gamma_t\cdot\nabla F=0.
 $$
 Since for $x \in \partial D_t$ the vector  $\nabla F(t,x)$ is colinear to the normal vector   $\vec{n}_t$ then 
 \begin{equation}\label{boundarys}
 (\partial_t\gamma_t-u(t,\gamma_t))\cdot\vec{n}_t=0.
 \end{equation}
 The meaning of \eqref{boundarys} is that the velocity of the boundary and the the velocity of the fluid particle occupying
 the same position   have the same normal components. 
 We observe that the  equation \eqref{boundarys} can be written in a complex form which seems to be more convenient in our case,
 \begin{equation}\label{boundary}
 \hbox{Im}  \Big\{(\partial_t\gamma_t-v(t,\gamma_t))\overline{\gamma_t^\prime}\Big\}=0,
 \end{equation}
 where the "prime" denotes  the derivative with respect to the $\sigma$ variable.

 We now take a closer look at the case of a rotating  connected  patch.
 Assume that the boundary  rotates with the angular  velocity $\dot{\theta}(t)$ around its center of mass  which can be assumed to be the origin.
  According to the  Proposition \ref{vel-eq} the  velocity $u(t)$ can be recovered from the initial velocity $u_0$ through  to the formula
\begin{equation}\label{veloc}
u(t,x)=e^{i\theta(t)}u_0(e^{-i\theta(t)}x).
\end{equation}
Hence
$$
\textnormal{Im}\big\{u(t,\gamma_t)\overline{\gamma'_t}\big\}=\textnormal{Im}\big\{u_0(\gamma_0)\overline{\gamma'_0}\big\}.
$$
The rotating patch has a standard  parametrization  given by  $\gamma_t=e^{i\theta(t)}\gamma_0$  which yields
$$
\textnormal{Im}\big\{\partial_t\gamma_t\overline{\gamma'_t}\big\}=\dot\theta(t)\textnormal{Re}\big\{\gamma_0\overline{\gamma'_0}\big\}.
$$
Consequently the equation \eqref{boundary} becomes
\begin{equation*}
\dot\theta(t)\textnormal{Re}\big\{\gamma_0\overline{\gamma'_0}\big\}=\textnormal{Im}\big\{u_0(\gamma_0)\overline{\gamma'_0}\big\}
\end{equation*}
which is equivalent to
$$
\frac{\dot{\theta}(t)}{2}\frac{d}{ds}|\gamma_0(s)|^2=\hbox{Im}\Big\{
u_0(\gamma_0)\,\overline{\gamma_0^\prime}  \Big\}.
$$
If there exists some $s$ with $\frac{d}{ds}|\gamma_0(s)|^2\neq 0$
then, since the right-hand side does not depend on the time
variable, we conclude that $\dot{\theta}(t)=\Omega$ is constant. Otherwise,
$\frac{d}{ds}|\gamma_0(s)|^2$ vanishes everywhere, which tells us
that the initial domain is a disc and therefore it   rotates with any angular velocity. Finally we get the boundary equation
\begin{equation}\label{rotsq1}
\Omega \,\textnormal{Re}\big\{ z\, \overline{z^\prime}\big\}=\textnormal{Im}\big\{u_0(z)\,\overline{z^\prime}\big\},\quad \forall z\in D_0.
\end{equation}
Recall that $z^\prime$ is a tangent vector to the boundary $\partial D_0$ at the point $z.$
\vspace{0,3cm}

${\bf{(2)}}$ 
Combining \eqref{rotsq1} with   the velocity formula \eqref{veltu}  we get
\begin{equation}\label{model0}
\Omega\,\textnormal{Re}\big\{z\overline{z'}\big\}= C_\alpha\textnormal{Im}\Bigg\{\frac{1}{2\pi}\int_{\partial D_0}\frac{d\zeta}{\vert z-\zeta\vert^\alpha}\overline{z'}\Bigg\},\quad \forall\, z\in \partial D_0.
\end{equation}
We shall now parametrize the domain with the outside conformal mapping  $\phi: \mathbb{D}^c\to D_0^c$. 
\begin{equation}\label{Eq45}
\phi(w)=w+\sum_{n\geq 0}\frac{b_n}{w^n}
\end{equation}
Setting $z=\phi(w)$ and $\zeta=\phi(\tau)$, then for $w\in \mathbb{T}$ a tangent vector is given by
$$
\overline{z'}=-i\overline{w}\, {\overline{\phi'(w)}}.
$$
Inserting this in the  equation \eqref{model0} gives
\begin{equation}\label{model}
G(\Omega,\phi)(w)\triangleq\textnormal{Im}\Bigg\{\bigg(\Omega\phi(w)-\frac{C_\alpha}{2i\pi}\mathop{{\int}}_\mathbb{T}\frac{\phi'(\tau)d\tau}{\vert \phi(w)-\phi(\tau)\vert^\alpha}\bigg)\, \overline{w}\,\overline{\phi'(w)}\Bigg\}=0,\quad\forall\,  w\in \mathbb{T}.
\end{equation}
This achieves the proof of the proposition.
\end{proof}

\section{Tools } 
\quad The purpose of this introductory  section  is to review and collect  some technical tools that will be used quite often in the remainder of this paper. We will firstly recall some basic elements of the bifurcation theory. We will focus on  the Crandall-Rabinowitz's theorem, hereafter referred by C-R Theorem, which is very crucial for the proof of our main result. Secondly, some simple facts about   H\"older spaces $C^{n+\gamma}(\mathbb{T})$ will be recalled  and  we shall also explore  some results on the action of singular integral operators on these spaces.  Last, we end this section with  some integral computations that will be frequently used in the study  of the linearized operator. 

  \subsection{Elements of the bifurcation theory}
 We intend now to  give some formal explanations and general principles of the bifurcation theory. This discussion will be closed by stating C-R theorem.  Roughly speaking, the main objective  of this theory is to look for the solutions of the equation
 $$
 F(\lambda,x)=0
 $$
 where $F:\RR\times X\to Y$ is continuous function and satisfies some additional regularity assumptions. The vector spaces  $X$ and $Y$ are   Banach spaces. We assume in addition \mbox{that $x=0$} is a  trivial solution for any $\lambda$, that is, $F(\lambda, 0)=0$. Whether close to the  solution  $(\lambda_0, 0)$ one  can find a branch of non trivial ones is the main problem discussed in this theory. If this is the case  we say that there is a bifurcation at the point $(\lambda_0, 0)$. As the Implicit Function Theorem tells us, the first test that should be carried out  is to analyze  the linear operator  $\mathcal{L}_\lambda\triangleq \partial_xF(\lambda, 0):X\to Y$. If this operator  is an isomorphism then such non trivial solutions cannot exist. Thus a necessary condition for the  bifurcation is to get  a nontrivial  kernel of $\mathcal{L}_\lambda$.  In many instances, the involved Banach spaces are infinite-dimensional  and thus the bifurcation analysis  is in general very complex. However, if the linearized operator is of Fredholm type one can reduce the problem to finite-dimensional spaces by using the  so-called Lyapunov-Schmidt reduction.  Recall that  a Fredholm operator means that it is continuous and whose kernel $N(\mathcal{L}_\lambda)$ and cokerel $Y/R(\mathcal{L}_\lambda)$ are finite-dimensional, \mbox{where $R(\mathcal{L}_\lambda)$} denotes the range of  $\mathcal{L}_\lambda$.  If moreover the index of this operator is zero then the bifurcation may occur despite that some suitable conditions are satisfied. Here we shall only discuss the bifurcation with one dimensional kernel which is the most common one and appears in many dynamical systems as for our generalized SQG model. With the preceding   assumptions on the linear operator  a one-parameter curve bifurcates from the trivial solution provided a transversality assumption is satisfied. Roughly speaking, this latter assumption  means that the linear operator $\mathcal{L}_\lambda$  possesses  a one-parameter eigenvalues $\lambda\mapsto \mu(\lambda)$ that should cross the real axis at $\lambda_0$ with non zero velocity. This is the classical theorem proved by  Crandall and Rabinowitz  \cite{CR} which is a basic tool in the bifurcation theory and that will be used in this paper. More general results are summarized in the book of Kielh\"{o}fer \cite{Kil}. Now we recall Crandall-Rabinowitz Theorem.

\begin{theorem}\label{C-R} Let $X, Y$ be two Banach spaces, $V$ a neighborhood of $0$ in $X$ and let 
$
F : \RR \times V \to Y
$
with the following  properties:
\begin{enumerate}
\item $F (\lambda, 0) = 0$ for any $\lambda\in \RR$.
\item The partial derivatives $F_\lambda$, $F_x$ and $F_{\lambda x}$ exist and are continuous.
\item $N(\mathcal{L}_0)$ and $Y/R(\mathcal{L}_0)$ are one-dimensional. 
\item {\it Transversality assumption}: $F_{tx}(0, 0)x_0 \not\in R(\mathcal{L}_0)$, where
$$
N(\mathcal{L}_0) = span\{x_0\}, \quad \mathcal{L}_0\triangleq \partial_x F(0,0).
$$
\end{enumerate}
If $Z$ is any complement of $N(\mathcal{L}_0)$ in $X$, then there is a neighborhood $U$ of $(0,0)$ in $\RR \times X$, an interval $(-a,a)$, and continuous functions $\varphi: (-a,a) \to \RR$, $\psi: (-a,a) \to Z$ such that $\varphi(0) = 0$, $\psi(0) = 0$ and
$$
F^{-1}(0)\cap U=\Big\{\big(\varphi(\xi), \xi x_0+\xi\psi(\xi)\big)\,;\,\vert \xi\vert<a\Big\}\cup\Big\{(\lambda,0)\,;\, (\lambda,0)\in U\Big\}.
$$
\end{theorem}
\subsection{Singular integrals}
In this paragraph we shall  briefly recall the classical H\"{o}lder spaces on the periodic case and state some classical facts on the continuity of singular integrals over these spaces.
 It is convenient to think of  $2\pi$-periodic function $f:\RR\to\CC$ as a function of the complex variable $w=e^{i\eta}$ rather than a function of the real variable $\eta.$ To be more precise, let  $f:\mathbb{T}\to \RR^2$, be a continuous function, then it  can be assimilated to  a $2\pi-$ periodic function $g:\RR\to\RR$ via the relation
$$
f(w)=g(\eta),\quad w=e^{i\eta}.
$$
Hence when  $f$ is  smooth enough we get
$$
f^\prime(w)\triangleq\frac{df}{dw}=-ie^{-i\eta}g'(\eta).
$$  
Because  $d/dw$ and $d/d\eta$ differ only by a smooth factor with modulus one  we shall  in the sequel work with $d/dw$ instead of $d/d\eta$ which appears to be  more convenient in the computations.\\
Moreover, if $f$ has real Fourier coefficients and is of class $C^1$ then  we have the identity 
\begin{equation}\label{deriv-bar}
{\{\overline{f}\}^\prime}(w)=-\frac{1}{w^2}\overline{f^\prime(w)}.
\end{equation}
Now we shall introduce  H\"older spaces  on the unit circle $\mathbb{T}$.
\begin{definition}
Let 
$0<\gamma<1$. We denote by $C^\gamma(\mathbb{T}) $  the space of continuous functions $f$ such that
$$
\Vert f\Vert_{C^\gamma(\mathbb{T})}\triangleq \Vert f\Vert_{L^\infty(\mathbb{T})}+\sup_{x\neq y\in \mathbb{T}}\frac{\vert f(x)-f(y)\vert}{\vert x-y\vert^\alpha}<\infty.
$$
For any integer $n$ the space $C^{n+\gamma}(\mathbb{T})$ stands for the set of functions $f$ of class $C^n$ whose $n-$th order derivatives are H\"older continuous  with exponent $\gamma$. It is equipped with the usual  norm,
$$
\Vert f\Vert_{C^{n+\gamma}(\mathbb{T})}\triangleq \Vert f\Vert_{L^\infty(\mathbb{T})}+\Big\Vert \frac{d^n f}{dw^n}\Big\Vert_{C^\gamma(\mathbb{T})}.
$$ 
\end{definition}
Recall that the Lipschitz (semi)-norm is defined as follows.
$$
\|f\|_{\textnormal{Lip}(\mathbb{T})}=\sup_{x\neq y}\frac{|f(x)-f(y)|}{|x-y|}.
$$
Now we list some classical properties that will be used later especially in Section \ref{Reg}.
\begin{enumerate}
\item For $n\in \mathbb{N}, \gamma\in ]0,1[$ the space $C^{n+\gamma}(\mathbb{T})$ is an algebra.
\item For $K\in L^1(\mathbb{T})$ and $f\in C^{n+\gamma}(\mathbb{T})$ we have the convolution law,
$$
\|K*f\|_{C^{n+\gamma}(\mathbb{T})}\le \|K\|_{L^1(\mathbb{T})}\|f\|_{C^{n+\gamma}(\mathbb{T})}.
$$

\end{enumerate}

The next result is used often. It deals with singular integrals of the following type,
\begin{equation}\label{opsin}
\mathcal{T}(f)(w)=\int_{\mathbb{T}}K(w,\tau)\, f(\tau)d\tau,
\end{equation}
with $K:\mathbb{T}\times\mathbb{T}\to \mathbb{C}$ a singular kernel satisfying some properties. This problem will appear naturally when we shall deal with the regularity of the nonlinear operator in the rotating patches formalism, see Section \ref{Reg}. 
 The result  that we shall discuss with respect to this  subject  is classical and for the self-containing of the paper  we shall provide a complete proof which is similar to \cite{MOV}. 
\begin{lemma}\label{noyau}
Let $0\leq \alpha< 1$ and consider a function $K:\mathbb{T}\times\mathbb{T}\to \mathbb{C}$ with the following properties. There exits $C_0>0$ such that,
\begin{enumerate}
\item $K$ is measurable on $\mathbb{T}\times\mathbb{T}\backslash\{(w,w),\, w\in \mathbb{T}\}$ and
$$
\big\vert K(w,\tau)\big\vert\leq \frac{C_0}{\vert w-\tau\vert^\alpha}, \quad\forall\,  w\neq \tau\in \mathbb{T}.
$$
\item For each $\tau\in \mathbb{T}$, $w\mapsto K(w,\tau)$ is differentiable in $\mathbb{T}\backslash\{\tau\}$ and 
$$
  \big\vert\partial_w K(w,\tau)\big\vert\leq \frac{C_0}{\vert w-\tau\vert^{1+\alpha}}, \quad \forall\, w\neq \tau\in \mathbb{T}.
$$
\end{enumerate}
Then the operator $\mathcal{T} $ defined by \eqref{opsin} is continuous from $L^\infty(\mathbb{T})$ to $C^{1-\alpha}(\mathbb{T})$. More precisely, there exists a constant $C_\alpha$ depending only on $\alpha$ such that
$$
\Vert \mathcal{T}(f)\Vert_{1-\alpha}\leq C_\alpha C_0\Vert f\Vert_{L^\infty}.
$$

\end{lemma}

\begin{proof}

We first prove that $\mathcal{T}(f)$ is bounded on $\mathbb{T}$ . Let $w\in\mathbb{T}$, then by  the condition (1),
\begin{eqnarray*}
\vert \mathcal{T}(f)(w)\vert & \leq & C_0\Vert f\Vert_{L^\infty}\Big\vert \int_{\mathbb{T}} \frac{d\tau}{\vert w-\tau\vert^\alpha}\Big\vert\\ &\leq & C_\alpha C_0\Vert f\Vert_{L^\infty}.
\end{eqnarray*}
Next, take $w_1,w_2\in \mathbb{T}$, set $r=\vert w_1-w_2\vert$ and define $B_{r}(w_1)=\big\{ \tau\in \mathbb{T};\, \vert \tau-w_1\vert \leq r\big\}$. Then,
\begin{eqnarray*}
\Big\vert \mathcal{T}(f)(w_1)-\mathcal{T}(f)(w_2)\Big\vert &\leq &\Big\vert \int _{B_{2r}(w_1)}\vert f(\tau)\vert \vert K(w_1,\tau)\vert d\tau\Big\vert+\Big\vert \int _{B_{2r}(w_1)}\vert f(\tau)\vert \vert K(w_2,\tau)\vert d\tau\Big\vert \\ &+&\Big\vert \int _{B^c_{2r}(w_1)}\vert f(\tau)\vert \vert K(w_1,\tau)-K(w_2,\tau)\vert d\tau\Big\vert\\ &\triangleq & J_1+J_2+J_3. 
\end{eqnarray*}
By using again the condition (1), $J_1$ and $J_2$ can be estimated by
\begin{eqnarray*}
J_1+J_2 &\leq & C_0\Vert f\Vert_{L^\infty}\Big(  \Big\vert \int _{B_{2r}(w_1)}\frac{d\tau}{\vert w_1-\tau\vert^\alpha}\Big\vert+\Big\vert \int _{B_{3r}(w_2)}\frac{d\tau}{\vert w_2-\tau\vert^\alpha}\Big\vert\Big)\\ &\leq &  C_\alpha C_0\Vert f\Vert_{L^\infty}\vert w_1-w_2\vert^{1-\alpha}.
\end{eqnarray*}
To estimate the third  term $J_3$ we shall use the condition (2) combined with the  Mean Value Theorem, 
$$
|K(w_1,\tau)-K(w_2,\tau)|\le C C_0\frac{|w_1-w_2|}{|w_1-\tau|^{1+\alpha}},\quad \forall \,\tau\in B_{2r}^c(w_1).
$$
Consequently we get
\begin{eqnarray*}
J_3 &\leq &C C_0\Vert f\Vert_{L^\infty} \Big\vert \int _{B^c_{2r}(w_1)}\frac{\vert w_1-w_2\vert}{\vert w_1-\tau\vert^{1+\alpha}}d\tau\\ & \leq & C_\alpha C_0\Vert f\Vert_{L^\infty}\vert w_1-w_2\vert^{1-\alpha}.
\end{eqnarray*}
This concludes the result.
\end{proof}
As a by-product we obtain the result.
\begin{coro}\label{cor}
Let $0\le \alpha<1$, $\phi:\mathbb{T}\to \phi(\mathbb{T})$ be a bi-Lipschitz function with real Fourier coefficients  and define the operator
$$
\mathcal{T}_\phi: f\mapsto\displaystyle{\mathop{{\fint}_{\mathbb{T}}}}\frac{f(\tau)}{\vert \phi(w)-\phi(\tau)\vert^\alpha}d\tau,\quad w\in \mathbb{T}.
$$
Then  $\mathcal{T}_\phi:L^\infty\big(\mathbb{T}\big)\to C^{1-\alpha}\big(\mathbb{T}\big)$ is continuous with the estimation,
$$
\Vert \mathcal{T}_\phi(f)\Vert_{C^{1-\alpha}(\mathbb{T})}\leq C\Big(\|\phi^{-1}\|_{\textnormal{Lip}(\mathbb{T})}^\alpha+\|\phi\|_{\textnormal{Lip}(\mathbb{T})}^2\|\phi^{-1}\|_{\textnormal{Lip}(\mathbb{T})}^{1+\alpha}\Big)\Vert f\Vert_{L^\infty(\mathbb{T})},
$$
where $C$ is a positive constant depending only on $\alpha$.
\end{coro}
\begin{proof}
We set
$$
K(w,\tau)=\frac{1}{\vert \phi(w)-\phi(\tau)\vert^\alpha}, \quad \forall\,w\neq \tau\in \mathbb{T}.
$$
Since $\phi$ is bi-Lipschitz  then we deduce that
\begin{equation}\label{k}
\vert K(w,\tau)\vert \leq \|\phi^{-1}\|_{\textnormal{Lip}(\mathbb{T})}^\alpha \frac{1}{\vert w-\tau\vert^{\alpha}}\quad \,\forall w\neq \tau\in \mathbb{T}.
\end{equation}
To get the second assumption $(2)$ of Lemma \ref{noyau} we shall compute $\partial_w K(w,\tau)$.
\begin{eqnarray}\label{for22}
\nonumber\partial_w K(w,\tau)
\nonumber&=&\frac{-\alpha}{2}\bigg(\phi'(w)\frac{\overline{\phi(w)}-\overline{\phi(\tau)}}{\vert \phi(w)-\phi(\tau)\vert^{\alpha+2}}+{(\overline{\phi})^\prime(w)}\frac{\phi(w)-\phi(\tau)}{\vert \phi(w)-\phi(\tau)\vert^{\alpha+2}}\bigg)\\
&=&\frac{-\alpha}{2}\bigg(\phi'(w)\frac{\overline{\phi(w)}-\overline{\phi(\tau)}}{\vert \phi(w)-\phi(\tau)\vert^{2}}-\frac{\overline{\phi'(w)}}{w^2}\frac{\phi(w)-\phi(\tau)}{\vert \phi(w)-\phi(\tau)\vert^{2}}\bigg)K(w,\tau)\quad w\neq \tau\in \mathbb{T}.
\end{eqnarray}
We have used the fact that the Fourier coefficients of $\phi$ are real and therefore we can apply the identity  \eqref{deriv-bar}. It follows that,
\begin{eqnarray*}
\vert \partial_w K(w,\tau)\vert &\leq & C \|\phi\|_{\textnormal{Lip}(\mathbb{T})}^2\frac{1}{\vert \phi(w)-\phi(\tau)\vert^{\alpha+1}}\notag\\ &\leq &C \|\phi\|_{\textnormal{Lip}(\mathbb{T})}^2\|\phi^{-1}\|_{\textnormal{Lip}(\mathbb{T})}^{1+\alpha}\frac{ 1}{\vert w-\tau\vert^{\alpha+1}}\cdot
\end{eqnarray*}
We can conclude by  Lemma \ref{noyau} and get  the desired result.
\end{proof}
\subsection{Basic integrals}
This section presents some basic computations of few integrals that will appear later in the study of the linearized operator. But before going further into the details we shall recall some facts on the gamma function which emerges  in a natural way in our computations.  The  function  $\Gamma:\CC\backslash(-\NN)\to \CC$  refers to the gamma function which is the analytic continuation to the negative half plane of the usual gamma function defined on the positive half-plane $\big\{\hbox{Re}z>0\big\}$ by the integral representation
$$
\Gamma(z)=\int_{0}^{+\infty} t^{z-1}\, e^{-t}dt.
$$
It satisfies the relation
\begin{equation}\label{Gamma1}
\Gamma(z+1)=z\,\Gamma(z), \quad \forall z\in \CC \backslash(-\NN).
\end{equation}
 Note that this function does not vanish and its poles $\{-n, n\in \NN\}$ are simple and so the reciprocal gamma function $\frac{1}{\Gamma}$ is an entire function. There are some particular values of the gamma function that will be used later,
 \begin{equation}\label{for1}
 \Gamma(n+1)=n!,\quad \Gamma(1/2)=\sqrt{\pi}.
 \end{equation}
  Now we shall introduce another related function called  the digamma function which is nothing but the logarithmic derivative of the function gamma and often  denoted by $\digamma$. It is given by
 $$
 \digamma(x)=\frac{\Gamma^\prime(x)}{\Gamma(x)}\cdot 
 $$
 For a future use we need the following identity,
 \begin{equation}\label{digam}
\forall n\in \NN,\quad \digamma(n+\frac12)=-\gamma-2\ln2+2\sum_{k=0}^{n-1}\frac{1}{2k+1}\cdot
 \end{equation}
Now for $x\in\RR$  we denote by $(x)_n$ the Pokhhammer's symbol defined by

\begin{equation}\label{Poch}
(x)_n=\left\{ \begin{array}{ll}
x(x+1)...(x+n-1),\quad \hbox{if}\quad n\geq1, &\\
1,\quad \hbox{if}\quad n=0.
\end{array} \right.
\end{equation}
Note  that  in the literature the above notation is replaced by $(x)^n$ which can introduce in our  context a  lot of confusion with the power $x^n$ and for  this reason we prefer not to use it. 

It is obvious that
\begin{equation}\label{f1s}
(x)_n=x\,(1+x)_{n-1},\quad (x)_{n+1}=(x+n)\,(x)_n.
\end{equation}
From the identity \eqref{Gamma1} we deduce the relations
\begin{equation}\label{Poc}
(x)_n=\frac{\Gamma(x+n)}{\Gamma(x)},\quad (x)_n=(-1)^n\frac{\Gamma(1-x)}{\Gamma(1-x-n)},
\end{equation}
provided all the quantities in the right terms are well-defined.

In the sequel we shall prove the following lemma which is the main result of this section.
\begin{lemma}\label{lem} Let $n\in \NN$ and $\alpha\in (0,1)$. 
 Then for any $w\in \mathbb{T}$ we have the following formulae.
\begin{equation}\label{In}
\fint_{\mathbb{T}}\frac{\tau^n}{|\tau-w|^\alpha}d\tau=\frac{\Gamma(1-\alpha)}{\Gamma^2(1-\alpha/2)}\frac{\big(\frac\alpha2\big)_{n+1}}{(1-\frac\alpha2)_{n+1}}\,w^{n+1}\cdot
\end{equation}
\begin{equation}\label{Jn}
\fint_\mathbb{T}\frac{(w-\tau)(w^{n}-\tau^{n})}{\vert w-\tau\vert^{\alpha+2}}d\tau=\frac{\big(1+\frac\alpha2\big)\Gamma(1-\alpha)}{(2-\alpha)\Gamma^2(1-\frac\alpha2)}\Bigg(1- \frac{\big(2+\frac\alpha2\big)_{n}}{\big(2-\frac\alpha2\big)_{n}}\Bigg)w^{n+2}.
\end{equation}
\begin{equation}\label{Kn}
\fint_\mathbb{T}\frac{(\overline{w}-\overline{\tau})(\overline{w}^n-\overline{\tau}^{n})}{\vert 1-\tau\vert^{\alpha+2}}d\tau=- \frac{\Gamma(1-\alpha)}{2\Gamma^2(1-\frac\alpha2)}\Bigg(1-\frac{\big(\frac\alpha2\big)_{n}}{(-\frac\alpha2)_{n}} \Bigg)\overline{w}^n.
\end{equation}
\end{lemma}
\begin{proof}
We start with  the change of variables $\tau=w \zeta$, 
\begin{eqnarray*}
\fint_{\mathbb{T}}\frac{\tau^n}{|\tau-w|^\alpha}d\tau&=& w^{n+1}\fint_{\mathbb{T}}\frac{\zeta^n}{|\zeta-1|^\alpha}d\tau\\
&=&w^{n+1}\frac{1}{2^{\alpha+1}\pi}\int_{0}^{2\pi}\frac{e^{i(n+1)\eta}}{|\sin(\eta/2)|^\alpha}d\eta.
\end{eqnarray*}
Again by the change of variables $\eta/2\mapsto\eta$   one gets 
\begin{eqnarray*}
\fint_{\mathbb{T}}\frac{\tau^n}{|\tau-w|^\alpha}d\tau
&=&w^{n+1}\frac{1}{2^{\alpha}\pi}\int_{0}^{\pi}\frac{e^{2i(n+1)\eta}}{\sin^\alpha\eta}d\eta.
\end{eqnarray*}
We shall now recall the following identity, see for instance \cite[p.8]{M-O} and \cite[p.449]{W}.  

\begin{equation}\label{lem00}
\int_{0}^{\pi}\sin^{x}(\eta) e^{iy\eta}d\eta=\frac{\pi e^{i\frac{\pi y}{2}}\Gamma(x+1)}{2^{x}\Gamma(1+\frac{x+y}{2})\Gamma(1+\frac{x-y}{2})},\quad \forall\, x>-1,\quad\forall\,  y\in \RR.
\end{equation}
As it  was pointed before the gamma function  has no real zeros  but simple poles located at $-\NN$ and therefore   the function $\frac{1}{\Gamma}$ admits an analytic continuation on $\CC.$ Apply this formula  with $x=-\alpha$ and $y=2(n+1)$ yields,
\begin{equation}\label{imp}
\frac{1}{2^\alpha\pi}\int_{0}^{\pi}\frac{e^{2i(n+1)\eta}}{\sin^\alpha\eta}d\eta=\frac{(-1)^{n+1}\Gamma(1-\alpha)}{\Gamma(n+2-\frac\alpha2)\Gamma(-n-\frac\alpha2)}\cdot
\end{equation}
It is easy to see that from the  relations \eqref{Poc} we may write for any $n\in \NN,$
\begin{eqnarray*}
\Gamma(1+n-\alpha/2)&=&\Gamma(1-\alpha/2)\,\Big(1-\frac\alpha2\Big)_n\\
\Gamma(1-n-\alpha/2)&=&(-1)^n\frac{\Gamma(1-\alpha/2)}{\Big(\frac\alpha2\Big)_n}.
\end{eqnarray*}
It follows that
$$
\Gamma(1-n-\alpha/2)\Gamma(1+n-\alpha/2)=(-1)^n\Gamma^2\big(1-\frac\alpha2\big)\frac{\Big(1-\frac\alpha2\Big)_n}{\Big(\frac\alpha2\Big)_n}\cdot
$$
By replacing $n$ with $n+1$ we get
$$
\Gamma(-n-\alpha/2)\Gamma(2+n-\alpha/2)=(-1)^{n+1}\Gamma^2\big(1-\frac\alpha2\big)\frac{\Big(1-\frac\alpha2\Big)_{n+1}}{\Big(\frac\alpha2\Big)_{n+1}}\cdot
$$
Inserting this identity into \eqref{imp} gives
\begin{equation}\label{ineq001}
\frac{1}{2^\alpha\pi}\int_{0}^{\pi}\frac{e^{2i(n+1)\theta}}{\sin^\alpha\theta}d\theta=\frac{\Gamma(1-\alpha)}{\Gamma^2(1-\frac\alpha2)}\frac{\Big(\frac\alpha2\Big)_{n+1}}{\Big(1-\frac\alpha2\Big)_{n+1}}\cdot
\end{equation}
Consequently
\begin{eqnarray*}
\fint_{\mathbb{T}}\frac{\tau^n}{|\tau-w|^\alpha}d\tau
&=&\frac{\Gamma(1-\alpha)}{\Gamma^2(1-\frac\alpha2)}\frac{\Big(\frac\alpha2\Big)_{n+1}}{\Big(1-\frac\alpha2\Big)_{n+1}}\, w^{n+1}
\cdot
\end{eqnarray*}
This completes the proof of \eqref{In}.

We intend now to compute the second  integral. To this end we use a change of variable as before,
\begin{eqnarray*}
J_n &\triangleq & \fint_\mathbb{T}\frac{(w-\zeta)(w^n-\zeta^{n})}{\vert w-\zeta\vert^{\alpha+2}}d\zeta =w^{n+2}\fint_\mathbb{T}\frac{(1-\zeta)(1-\zeta^{n})}{\vert 1-\zeta\vert^{\alpha+2}}d\zeta.
\end{eqnarray*}
Using once again   the change of variables $\zeta\mapsto e^{i\eta}$ and $\eta\mapsto 2\eta$ one gets
\begin{eqnarray*}
J_n&=& 
\frac{w^{n+2}}{2\pi}\int_0^{2\pi}\frac{(1-e^{i\eta})(1-e^{in\eta})e^{i\eta}}{2^{\alpha+2}\vert\sin (\eta/2)\vert^{\alpha+2}}d\eta\notag\\ &=& \frac{w^{n+2}}{2^{\alpha+2}\pi}\int_0^{\pi}\frac{(1-e^{i2\theta})(1-e^{i2n\theta})e^{i2\theta}}{\big(\sin \theta\big)^{\alpha+2}}d\theta.\\
\end{eqnarray*}
Observe  that
$$
J_n=\frac{w^{n+2}}{2^{\alpha+1}i\pi}\int_{0}^\pi\frac{(e^{2i\theta}-e^{i2(n+1)\theta})e^{i\theta}}{\sin^{\alpha+1} \theta}d\theta
$$
and therefore
\begin{eqnarray*}
 J_n&=& \frac{w^{n+2}}{2^{\alpha+1}i\pi}\bigg(\int_0^{\pi}{\Big(e^{i2\theta}-e^{i2(n+1)\theta}\Big)}\frac{\cos\theta}{\sin^{\alpha+1} \theta}d\theta+i\int_0^{\pi}\frac{e^{i2\theta}-e^{i2(n+1)\theta}}{\sin^\alpha \theta}d\theta\bigg).
\end{eqnarray*}
Integrating by parts implies
$$
\int_0^{\pi}{\Big(e^{i2\theta}-e^{i2(n+1)\theta}\Big)}\frac{\cos\theta}{\sin^{\alpha+1} \theta}d\theta=\frac{2i}{\alpha}\int_0^{\pi}\frac{\big(e^{i2\theta}-(n+1)e^{i2(n+1)\theta}\big)}{\sin^\alpha \theta}d\theta.
$$
Note  that in this formula the contribution coming from the boundary terms is zero \mbox{for $\alpha\in [0,1[$.} Hence we get
$$
J_n= \frac{w^{n+2}}{2^{\alpha+1}\pi}\bigg(\frac{2+\alpha}{\alpha}\int_0^{\pi}\frac{e^{i2\theta}}{\sin^\alpha \theta}d\theta-\frac{2(n+1)+\alpha}{\alpha}\int_0^{\pi}\frac{e^{i2(n+1)\theta}}{\sin^\alpha \theta}d\theta\bigg).
$$
Combining this formula with the  identity \eqref{ineq001} gives
\begin{eqnarray*}
J_n&=& w^{n+2}\frac{(2+\alpha)\Gamma(1-\alpha)}{2(2-\alpha)\Gamma^2(1-\frac\alpha2)}-\frac{\Gamma(1-\alpha)}{\Gamma^2(1-\frac\alpha2)}\frac{2n+2+\alpha}{2\alpha}\frac{\big(\frac\alpha2\big)_{n+1}}{\big(1-\frac\alpha2\big)_{n+1}}\\
&=&w^{n+2}\frac{(1+\frac\alpha2)\Gamma(1-\alpha)}{(2-\alpha)\Gamma^2(1-\frac\alpha2)}\bigg(1- \frac{1-\frac\alpha2}{1+\frac\alpha2}\frac{n+1+\frac{\alpha}{2}}{\frac\alpha2}\,\frac{\big(\frac\alpha2\big)_{n+1}}{\big(1-\frac\alpha2\big)_{n+1}} \bigg).
\end{eqnarray*}
By \eqref{f1s} we may  transform this formula into,
$$
J_n=w^{n+2}\frac{(1+\frac\alpha2)\Gamma(1-\alpha)}{(2-\alpha)\Gamma^2(1-\frac\alpha2)}\bigg(1-\,\frac{\big(2+\frac\alpha2\big)_{n}}{\big(2-\frac\alpha2\big)_{n}} \bigg).
$$
Next we shall now move to the computation of the last  integral \eqref{Kn},
$$
Z_n\triangleq  \fint_\mathbb{T}\frac{(\overline{w}-\overline{\tau})(\overline{w}^n-\overline{\tau}^{n})}{\vert 1-\tau\vert^{\alpha+2}}d\tau =\overline{w}^{n} \fint_\mathbb{T}\frac{(1-\overline{\zeta})(1-\overline{\zeta}^{n})}{\vert 1-\zeta\vert^{\alpha+2}}d\zeta.
$$
Making a  standard change of variables as for the preceding integral we obtain
\begin{eqnarray*}
Z_n&=&\frac{\overline{w}^{n}}{2\pi}\int_0^{2\pi}\frac{(1-e^{-i\eta})(1-e^{-in\eta})e^{i\eta}}{2^{\alpha+2}\vert\sin (\eta/2)\vert^{\alpha+2}}d\eta\\ &=&\frac{\overline{w}^{n}}{2^{\alpha+2}\pi} \int_0^{\pi}\frac{(1-e^{-i2\eta})(1-e^{-i2n\eta})e^{i2\eta}}{\sin^{\alpha+2}\eta}d\eta\\ &=&
  \frac{i\overline{w}^{n}}{2^{\alpha+1}\pi}\int_0^{\pi}\frac{(1-e^{-i2n\theta})e^{i\theta}}{\sin^{\alpha+1}\theta}d\theta\\ &=&  \frac{i\overline{w}^{n}}{2^{\alpha+1}\pi}\bigg(\int_0^{\pi}\frac{\cos\theta(1-e^{-i2n\theta})}{\sin^{\alpha+1}\theta}d\theta+i\int_0^{\pi}\frac{(1-e^{-i2n\theta})}{\sin^{\alpha}\theta}d\theta\bigg).
\end{eqnarray*}
Integrating by parts gives
\begin{eqnarray*}
\int_0^{\pi}\frac{\cos\theta(1-e^{-i2n\theta})}{\sin^{\alpha+1}\theta}d\theta=\frac{2i n}{\alpha}\int_0^{\pi}\frac{e^{-i2n\theta}}{\sin^\alpha \theta}d\theta.
\end{eqnarray*}
This implies that
\begin{eqnarray*}
Z_n&=&-\frac{\overline{w}^{n}}{2^{\alpha+1}\pi}\bigg(\int_0^{\pi}\frac{1}{\sin^\alpha \theta}d\theta+\frac{2n-\alpha}{\alpha}\int_0^{\pi}\frac{e^{-i2n\theta}}{\sin^{\alpha}\theta}d\theta\bigg).
\end{eqnarray*}
Using once again  \eqref{ineq001} and \eqref{f1s} we obtain
\begin{eqnarray*}
Z_n&=&-\overline{w}^{n}\frac{\Gamma(1-\alpha)}{2\Gamma^2(1-\frac\alpha2)}\bigg(1+\frac{n-\frac\alpha2}{\frac\alpha2}\,\frac{(\frac\alpha2)_n)}{(1-\frac\alpha2)_n}\bigg)\\
&=&-\overline{w}^{n}\frac{\Gamma(1-\alpha)}{2\Gamma^2(1-\frac\alpha2)}\bigg(1-\frac{(\frac\alpha2)_n}{(-\frac\alpha2)_n}\bigg).
\end{eqnarray*}
Note  that we have used the following fact which can be deduced easily  from \eqref{ineq001} by conjugation, 
\begin{eqnarray*}
\int_0^{\pi}\frac{e^{-i2n\theta}}{\sin^{\alpha}\theta}d\theta&=&\int_0^{\pi}\frac{e^{i2n\theta}}{\sin^{\alpha}\theta}d\theta
\end{eqnarray*}
and therefore  the proof of the lemma is now completed.
\end{proof}

\section{Elliptic  patches}
Given a simply connected domain,  to check whether or not it is a rotating patch can be done  through  the equation of Proposition \ref{prop-bound} provided  that  a parametric representation of the boundary  is known (for example the one given by the conformal mapping ) and the  computations of the the integral term are feasible.  In what follows we shall concretize this program for  
 some elementary domains. We shall prove that the ellipses  never rotate except for the degenerate case where they coincide with discs. We point out this result was recently shown  in \cite{Cor} and we will give here a flexible  proof with less computations.
\begin{proposition}\label{pro-el}
The following holds true
\begin{enumerate}
\item The discs are rotating patches for any $\Omega\in \RR.$ 
\item The ellipses are not rotating patches.

\end{enumerate}
\end{proposition}
\begin{proof}

${\bf(1)}$ Recall from \eqref{model} that the conformal mapping of a rotating domain must satisfy the equation
$$
G(\Omega,\phi(w))=0,\quad\forall w\in \mathbb{T}
$$
To check whether or not the unit disc is a solution, it suffices to prove that
$$
G(\Omega, \textnormal{Id})=0.
$$
It is easy to see that, 
 \begin{eqnarray*}
G(\Omega,\textnormal{Id})(w)&=&\textnormal{Im}\bigg\{\Big(\Omega w-C_\alpha\fint_\mathbb{T}\frac{d\tau}{\vert w-\tau\vert^\alpha}\Big)\frac{1}{w}\bigg\}\\
 &=& -C_\alpha\textnormal{Im}\bigg\{\fint_\mathbb{T}\frac{d\tau}{w\vert w-\tau\vert^\alpha}\bigg\}.
\end{eqnarray*} 
Using the the formula \eqref{In} with $n=0$ we may conclude tha for any $\Omega\in \RR,$
$$
G(\Omega,\textnormal{Id})=0.
$$
We observe  that this result is known and expected because  the disc corresponds to  a stationary solution for \eqref{E} and is invariant by rotation. 
\vspace{0,2cm}

${\bf{(2)}}$ By translation, dilation and rotation we can assume that the ellipse $\mathcal{E}$ is parametrized by the conformal mapping
 $$\phi_Q: w\in \mathbb{T}\mapsto w+Q \overline{w},\quad \textnormal{with}\quad  Q=\frac{a-b}{a+b}\in (0,1)
 $$ where $a$ and $b$ denote the major and minor axes, respectively. This map sends conformally the exterior of the unit disc to the exterior of the ellipse. Performing straightforward computations leads in view of \eqref{model} to 
 \begin{eqnarray*}
G\big(\Omega,\phi_Q\big)(w)
 &=&-\textnormal{Im}\bigg\{2Q\Omega w^2+C_\alpha\big(\overline{w}-Qw\big)\fint_\mathbb{T}\frac{\big(1-Q\overline{\tau}^2\big)d\tau}{\vert w-\tau+Q(\overline{w}-\overline{\tau})\vert^\alpha}\bigg\}.
\end{eqnarray*}
By using the identity
$$
\big\vert z+Q\overline{z}\big\vert^2=(1+Q^2)|z|^2+2Q\,\textnormal{Re}(z^2),\quad\forall z\in \CC,
$$
one gets
 \begin{eqnarray*}
G\big(\Omega,\phi_Q\big)(w)=-\textnormal{Im}\bigg\{2Q\Omega w^2+C_\alpha\big(\overline{w}-Qw\big) \mathop{{\fint}}_\mathbb{T}\frac{\big(1-Q\overline{\tau}^2\big)d\tau}{\Big[(1+Q^2)\vert w-\tau\vert^2+2Q\textnormal{Re}\big\{(w-\tau)^2\big\}\Big]^{\alpha/2}}\bigg\}.
\end{eqnarray*}
Making  the change of variables $\tau=w\zeta$ and using the identity 
 $$\big(1-z\big)^2=-z\vert 1-z\vert^2,\quad\forall\, z\in \mathbb{T}
 $$ we find
 \begin{eqnarray*}
G\big(\Omega, \phi_Q\big)(w)=
-\textnormal{Im}\bigg\{2Q\Omega w^2+\frac{C_\alpha}{(1+Q^2)^{\frac\alpha2}}\big(1-Qw^2\big)\mathop{{\fint}}_\mathbb{T}\frac{\big(1-Q\overline{w}^2\overline{\zeta}^2\big)d\zeta}{\vert 1-\zeta\vert^\alpha\Big[1-\frac{2Q}{1+Q^2}\textnormal{Re}\{w^2\zeta\}\Big]^{\alpha/2}}\bigg\}.
\end{eqnarray*}
We shall transform the last  integral term as follows,
 \begin{eqnarray*}
\mathop{{\fint}}_\mathbb{T}\frac{1-Q\overline{w}^2\overline{\zeta}^2}{\vert 1-\zeta\vert^\alpha\Big[1-\frac{2Q}{1+Q^2}\textnormal{Re}\{w^2\zeta\}\Big]^{\alpha/2}}d\zeta=J(w)-Q\overline{w}^2\overline{J(w)},
 \end{eqnarray*}
with
$$
J(w)\triangleq\displaystyle{\mathop{{\fint}}_\mathbb{T}}\frac{d\zeta}{\vert 1-\zeta\vert^\alpha\big(1-\frac{2Q}{1+Q^2}\textnormal{Re}\{w^2\zeta\}\big)^{\alpha/2}}\cdot$$
Therefore we get,
 \begin{eqnarray}\label{brak}
\nonumber G\big(\Omega, \phi_Q\big)(w) &=& -\textnormal{Im}\bigg\{2Q\Omega w^2+\frac{C_\alpha}{(1+Q^2)^{\frac\alpha2}}\Big(J(w)+Q^2\overline{J(w)}-Q\Big[w^2J(w)+\overline{w}^2\overline{J(w)}\Big]\Big)\bigg\}\\ &=&
-\textnormal{Im}\Big\{2Q\Omega w^2+C_\alpha\frac{1-Q^2}{(1+Q^2)^{\frac\alpha2}}J(w)\Big\}.
\end{eqnarray}
Since $\big|\frac{2Q}{1+Q^2}\textnormal{Re}\{w^2\zeta\}\big|<1$ then  we can use  the   Taylor series 
\begin{eqnarray*}
\Big(1-\frac{2Q}{1+Q^2}\textnormal{Re}\{w^2\zeta\}\Big)^{-\alpha/2} = \sum_{n=0}^{\infty}2^nA_n \big(\textnormal{Re}\{w^2\zeta\}\big)^n,
\end{eqnarray*}
with
$$A_n=\frac{\big(\alpha/2\big)_n}{n!}\Big(\frac{Q}{1+Q^2}\Big)^n,\quad \forall \, n\in \NN.
$$ Consequently
we get
\begin{eqnarray*}
J(w)&=&\sum_{n=0}^{\infty}2^nA_n \fint_\mathbb{T}\frac{\big(\textnormal{Re}\{w^2\zeta\}\big)^{n}}{\vert 1-\zeta\vert^\alpha}d\zeta\\ &=&a_\alpha+ \sum_{n=1}^{\infty} A_n \sum_{k=0}^{n}\left( \begin{array}{c}n \\k\end{array} \right)w^{2(n-2k)} \fint_\mathbb{T}\frac{\zeta^{n-2k}}{\vert 1-\zeta\vert^\alpha}d\zeta.
\end{eqnarray*}
By the Lemma \ref{lem} the coefficient $a_\alpha$ is real and therefore it  does not contribute in $\textnormal{Im} J(w).$ Our goal now is to compute the coefficients of $w^4$ and $ \overline{w}^4$ of the function between the bracket in \eqref{brak},  denoted by $B_4$ and ${B}_{-4}$, respectively.   First we observe that the coefficient $B_4$ can be obtained by summing over  the set  
$$\Big\{n\geq 1,0\leq k\leq n \,\backslash\,  n-2k=2\Big\}=\Big\{n\geq 1,k\geq 0 \,\backslash\,  n=2k+2\Big\}.
$$ 
 This is equivalent to write
 \begin{eqnarray*}
B_4&=& \sum_{k=0}^{\infty} A_{2k+2} \left( \begin{array}{c}2k+2 \\ k\end{array} \right) \fint_\mathbb{T}\frac{\zeta^{2}d\zeta}{\vert 1-\zeta\vert^\alpha}\\ &=& \sum_{k=1}^{\infty} A_{2k} \left( \begin{array}{c}2k \\ k-1\end{array} \right)a_2,\qquad\quad\qquad  a_2\triangleq \displaystyle{\fint_\mathbb{T}\frac{\zeta^{2}d\zeta}{\vert 1-\zeta\vert^\alpha}}.
\end{eqnarray*}
Next we shall compute the coefficient of $\overline{w}^4$ denoted by $B_{-4}.$ This may be done    by summing over  the set  
$$\Big\{n\geq 1,0\leq k\leq n \,\backslash\,  n-2k=-2\Big\}=\Big\{n\geq 1, k\geq 2 \,\backslash\,  n=2k-2\Big\}.
$$ Hence by change of variables,
 \begin{eqnarray*}
{B}_{-4}&=& \sum_{k=2}^{\infty} A_{2k-2} \left( \begin{array}{c}2k-2 \\ k\end{array} \right)\fint_\mathbb{T}\frac{\xi^{-2}d\xi}{\vert 1-\xi\vert^\alpha}\\ &=& \sum_{k=1}^{\infty} A_{2k} \left( \begin{array}{c}2k \\ k+1\end{array} \right)\fint_\mathbb{T}\frac{d\zeta}{\vert 1-\zeta\vert^\alpha}\\&\triangleq& \sum_{k=1}^{\infty} A_{2k} \left( \begin{array}{c}2k \\ k-1\end{array} \right)a_0.
\end{eqnarray*}
But, in view of Lemma \ref{lem}, one has
\begin{eqnarray*}
\frac{a_2}{a_0}=\frac{(2+\alpha)(4+\alpha)}{(4-\alpha)(6-\alpha)}\neq 1\quad \textnormal{for}\,\, \alpha\neq 1.  
\end{eqnarray*}
Thus $B_4\neq  B_{-4}$ and therefore  the coefficient of $w^4-\overline{w}^4$ of $G(\Omega,\phi_Q)(w)$ does not vanish.  It follows that the equation $G(\Omega, \phi_Q)(w)=0, \forall w\in \mathbb{T}$ is not true for any $\Omega$. This concludes the proof of the  desired result.
\end{proof}
\section{General statement}
In this section we shall give a more precise statement of Theorem \ref{thmV1}. In particular we shall give a description of the conformal mapping which parametrizes the rotating patches close to the unit disc. 

\begin{theorem}\label{thmV2}
Let $\alpha\in]0,1[$ and $m\in\NN^*\backslash\{1\}$. Then there exists $a>0$ and two continuous functions $\Omega : (-a, a) \to\RR$, $\phi : (-a, a)\to C^{2-\alpha}\big(\mathbb{T}\big)$ satisfying $\Omega(0)=\Omega_m^\alpha $, $\phi(0)=\textnormal{Id}$, such that $(\phi_s)_{-a<s<a}$ is a one-parameter non trivial solution of  the equation \eqref{Rotaeq}, where
$$
\Omega_m^\alpha\triangleq\frac{\Gamma(1-\alpha)}{2^{1-\alpha}\Gamma^2(1-\frac\alpha2)}\bigg(\frac{\Gamma(1+\frac\alpha2)}{\Gamma(2-\frac\alpha2)}-\frac{\Gamma(m+\frac\alpha2)}{\Gamma(m+1-\frac\alpha2)}\bigg),
$$
Moreover, $\phi_s$ admits the expansion
$$
\phi_s(w)=w\Big(1+s\frac{1}{w^m}+s\sum_{n\geq2} a_{nm-1}(s)\frac{1}{w^{nm}}\Big),\quad \forall\, w\in \mathbb{T},
$$
and it is  conformal on $\mathbb{C}\backslash\mathbb{D}$ and the complement $D_s$ of $\phi_s\big(\mathbb{C}\backslash\mathbb{D}\big)$ is an $m-$fold rotating patch with the  angular velocity $\Omega(s)$. In addition, the boundary of this patch belongs to the \mbox{class $C^{2-\alpha}.$}
\end{theorem}

{\bf$\bullet$\, Outline of the proof}. The proof of this theorem will be divided into several steps. The main key is  Crandall Rabinowitz Theorem, sometimes denoted by C-R, which  requires to check many properties  for  the linear and the nonlinear  functionals of the equation  \eqref{Rotaeq} defining the V-states. Firstly, we shall check the regularity assumptions that will be separated into weak and strong ones. Secondly, we will conduct a spectral study of the linearized operator around the trivial solution. In this context, we are able to describe the complete bifurcation set made of the values $\Omega$ such that the linearized operator is Fredholm with one-dimensional kernel. We shall also check in this section the transversality assumption of C-R Theorem. In the last step, we  give the complete proof for the existence of the V-states  and check their $m$-fold structure.
\section{Regularity of the functional $F$}\label{Reg}
This section is devoted to the study  of the regularity assumptions stated in  C-R Theorem. The object that we shall study is the nonlinear  functional  $G$ introduced in  \eqref{model} and  given by
\begin{equation*}
G(\Omega,\phi)(w)\triangleq\textnormal{Im}\Bigg\{\bigg(\Omega\phi(w)-{C_\alpha}\mathop{{\fint}}_\mathbb{T}\frac{\phi'(\tau)}{\vert \phi(w)-\phi(\tau)\vert^\alpha}d\tau\bigg)\overline{w}\,{\overline{\phi'(w)}}\Bigg\},\quad\forall\,  w\in \mathbb{T}.
\end{equation*}
Because we are interested in the bifurcation from  the disc (corresponding to $\phi=\textnormal{Id}$), it is more  convenient to make a translation and study the bifurcation from zero. To this end, we introduce the function $F$ defined by
$$
F(\Omega, f)(w)=G(\Omega, w+f(w)),\quad\forall w\in \mathbb{T}.
$$
In order to apply C-R Theorem we need first to fix  the function spaces and check the regularity of the functional $F$ with respect to these spaces. We should look for Banach spaces $X$ and $Y$ such that $F: \RR\times X\mapsto Y$ is well-defined and satisfies the assumptions of Theorem \ref{C-R}. These spaces will be defined in the spirit of the work done for the incompressible Euler equations \cite{HMV}. They are given by,
$$
X=\Big\{f\in C^{2-\alpha}(\mathbb{T}),\, f(w)=\sum_{n\geq 0}b_n\overline{w}^n, b_n\in \RR,\, w\in\mathbb{T}\Big\}
$$
and
$$
Y=\Big\{g\in C^{1-\alpha}(\mathbb{T}),\, g(w)=i\sum_{n\geq 1}g_n\big(w^n-\overline{w}^n\big), g_n\in \RR,\, w\in\mathbb{T}\Big\}.
$$
For  $r\in (0,1)$ we denote by $B_r$  the open ball of $X$ with center $0$ and radius $r$,
$$
B_r=\Big\{f\in X,\quad \Vert f\Vert_{C^{2-\alpha}}\leq r\Big\}.
$$
It is straightforward that for any $f\in B_r$ the function $w\mapsto \phi(w)=w+f(w)$ is conformal on  $\CC\backslash\overline{\mathbb{D}}.$ Moreover according to Kellog-Warshawski result \cite{WS}, the boundary of $\phi(\CC\backslash\overline{\mathbb{D}})$  is a Jordan curve of \mbox{class $C^{2-\alpha} $.} This gives the proof of the last result of Theorem \ref{thmV2} provided that the regularity of $\phi$ is shown.
Note that we can prove the regularity of the boundary without making appeal to the result \cite{WS}. We just look for the conformal parametrization $\theta\mapsto \phi(e^{i\theta})$ which is regular and prove that it belongs to $C^{2-\alpha}$. This last fact is equivalent to  $\phi\in C^{2-\alpha}(\mathbb{T})$.

\subsection{Weak regularity} 
Our objective is to prove that the functional $F$ is well-defined and admits G\^ateaux derivatives for any given direction. More precisely, we shall prove the following result.
\begin{proposition}
For any $r\in (0,1)$  the following holds true.
\begin{enumerate}
\item $F: \RR\times B_r\to Y$ is well-defined.
\item For each point $(\Omega,f)\in \RR\times B_r,$
the Gâteaux derivative of $F$, $\partial_fF(\Omega,f): X\to Y$  exists and belongs to $\mathcal{L}(X,Y)$
 \end{enumerate}
\end{proposition}
\begin{proof}
${\bf{(1)}}$ First, because  the space $C^{1-\alpha}(\mathbb{T})$ is an algebra, it is clear that  the first part of the functional $G$ given by,  $w\mapsto\Omega \,\phi(w) \, \overline{w}\, \overline{\phi^\prime(w)}$ belongs to $C^{1-\alpha}(\mathbb{T})$. To prove that the second term of $G$ belongs to $C^{1-\alpha}(\mathbb{T})$  it suffices to check that 
$$
S(\phi): w\mapsto \mathop{{\fint}}_\mathbb{T}\frac{\phi'(\tau)}{\vert \phi(w)-\phi(\tau)\vert^\alpha}d\tau\in C^{1-\alpha}(\mathbb{T}).
$$
This follows  immediately  from   Corollary \ref{cor}.
Therefore it remains to check that the Fourier coefficients of $G(\Omega,\phi)$ belong to $i\RR$. By the assumption, the Fourier coefficients of $\phi=\textnormal{Id}+f$ are real and thus the coefficients of $\overline{\phi^\prime}$ are real too. Now using  the stability of this property under the multiplication and the conjugation we deduce that the Fourier coefficients of $w\mapsto\Omega \phi(w)\overline{\phi^\prime(w)} \overline{w}$ are real. To complete the proof we shall check that the Fourier coefficients of $S(\phi)$ are also real for every  $f\in B_r.$ From the regularity of $\phi\in C^{1-\alpha}(\mathbb{T})$ we can pointwise expand  this function into its Fourier series, that is,  
$$S(\phi)(w)=\sum_{n\in\mathbb{Z}} a_n w^n,\quad a_n=\fint_{\mathbb{T}}\frac{S(\phi)(w)}{w^{n+1}}dw=\fint_{\mathbb{T}}\fint_{\mathbb{T}}\frac{\phi^\prime(\tau)}{|\phi(\tau)-\phi(w)|^\alpha}d\tau\frac{dw}{w^{n+1}}\cdot
$$
This coefficient can also be written in the form
$$
a_n=\frac{1}{4\pi^2}\int_{0}^{2\pi}\int_{0}^{2\pi}\frac{\phi^\prime(e^{i\theta})e^{i\theta}e^{-in\eta}}{|\phi(e^{i\theta})-\phi(e^{i\eta})|^\alpha}d\theta \,d\eta.
$$
By taking the conjugate of $a_n$ and using the properties
 $$\overline{\phi(e^{i\theta})}=\phi(e^{-i\theta}),\quad \overline{\phi^\prime(e^{i\theta})}=\phi^\prime(e^{-i\theta})\quad\hbox{and} \quad |z|=|\overline{z}|$$
 one may obtain by  change of variables
 \begin{eqnarray*}
 \overline{a_n}&=&\frac{1}{4\pi^2}\int_{0}^{2\pi}\int_{0}^{2\pi}\frac{\phi^\prime(e^{-i\theta})e^{-i\theta}e^{in\eta}}{|\phi(e^{-i\theta})-\phi(e^{-i\eta})|^\alpha}d\theta \,d\eta\\
 &=&\frac{1}{4\pi^2}\int_{0}^{2\pi}\int_{0}^{2\pi}\frac{\phi^\prime(e^{i\theta})e^{i\theta}e^{-in\eta}}{|\phi(e^{i\theta})-\phi(e^{i\eta})|^\alpha}d\theta \,d\eta\\
 &=&a_n.
 \end{eqnarray*}
 Consequently the Fourier coefficients of $S(\phi)$ are real and therefore  $F(\Omega,f)$   belongs to $Y$.
 \vspace{0,3cm}
 
 ${\bf{(2)}}$  We shall compute the Gâteaux derivative of $F$ at the point $f\in B_r$ in the direction $h\in X$. A refined analysis concerning its connection with Fr\'echet derivative will be developed in the next section. 
The Gâteaux derivative of $\partial_fF(\Omega,f)h$ is defined through the formula,
\begin{eqnarray*}
\nonumber  \partial_f F(\Omega,f)h(w)&=&\lim_{t\to 0}\frac{F(\Omega, f(w)+th(w))-F(\Omega, f(w))}{t}\\
&=&\frac{d}{dt}_{\Big|t=0}F(\Omega, f+th)(w).
\end{eqnarray*}
This limit is taken in the strong topology of $C^{1-\alpha}(\mathbb{T})$. Thus we shall first prove the existence of this limit for every point $w\in \mathbb{T}$ and after  check  that this limit exists in $C^{1-\alpha}(\mathbb{T})$.\\
 With the notation $\phi=\hbox{Id}+f$,
\begin{eqnarray}\label{gat22}
\nonumber  \partial_f F(\Omega,f)h(w)
&=&\frac{d}{dt}_{\Big|t=0}F(\Omega, f+th)(w)\\ &=&
\nonumber  \Omega\, \textnormal{Im} \bigg\{\phi(w)\,\overline{w}\,{\overline{h'(w)}}+h(w)\, \overline{w}\,{\overline{\phi'(w)}}\bigg\}\\
\nonumber &-&C_\alpha\, \textnormal{Im} \bigg\{S(\phi(w))\overline{w}{\overline{h'(w)}}+\overline{w}\,{\overline{\phi^\prime(w)}}\frac{d}{dt}_{\Big|t=0}S(\phi+th)(w)\bigg\}\\
&\triangleq& \mathcal{L}(f)(h(w)).
\end{eqnarray}
We shall make use of  the following identity: let $A\in \CC^\star$, $ B\in \CC$, $\alpha\in \RR$ and introduce the function $K:t\mapsto |A+Bt|^\alpha$ which is smooth close to zero, then we have
\begin{equation}\label{der1}
K^\prime(0)=\alpha |A|^{\alpha-2}\textnormal{Re}(\overline{A}B).
\end{equation}
 Combining this formula  with few easy computations one gets
 \begin{eqnarray}\label{dfh}
\frac{d}{dt}_{\Big|t=0}S(\phi+th)(w)&=& \fint_\mathbb{T}\frac{h^\prime(\tau)}{\vert \phi(w)-\phi(\tau)\vert^\alpha}d\tau -\frac{\alpha }{2} \mathop{{\fint}}_\mathbb{T}\frac{\big(\phi(w)-\phi(\tau)\big)\big(\overline{h(w)}-\overline{h(\tau)}\big)}{\vert \phi(w)-\phi(\tau)\vert^{\alpha+2}}\phi^\prime(\tau)d\tau \notag\\ 
&-& \frac{\alpha}{2}\mathop{{\fint}}_\mathbb{T}\frac{\big(\overline{\phi(w)}-\overline{\phi(\tau)}\big)\big(h(w)-h(\tau)\big)}{\vert \phi(w)-\phi(\tau)\vert^{\alpha+2}}\phi^\prime(\tau)d\tau\notag\\ &\triangleq&{A}(\phi,h)(w)-\frac{\alpha}{2}\Big({B}(\phi,h)(w)+{C}(\phi,h)(w)\Big).
\end{eqnarray}
Therefore we obtain from \eqref{gat22} the identity
\begin{eqnarray}\label{gat66}
 \nonumber\mathcal{L}(f)(h)(w)&=& \textnormal{Im} \bigg\{\Omega\,\Big[\phi(w)\,\overline{w}\,{\overline{h'(w)}}+h(w)\, \overline{w}\,{\overline{\phi'(w)}}\Big]-C_\alpha\,S(\phi(w))\,\overline{w}\,{\overline{h'(w)}}\bigg\}\\
 &-&C_\alpha\, \textnormal{Im} \bigg\{\overline{w}\,{\overline{\phi^\prime(w)}}\Big[{A}(\phi,h)(w)-\frac{\alpha}{2}\Big({B}(\phi,h)(w)+{C}(\phi,h)(w)\Big)\Big]\bigg\}.
\end{eqnarray}
Set
$$
\mathcal{L}_1(f)h(w)\triangleq \textnormal{Im} \bigg\{\Omega\,\Big[\phi(w)\,\overline{w}\,{\overline{h'(w)}}+h(w)\, \overline{w}\,{\overline{\phi'(w)}}\Big]-C_\alpha\,S(\phi(w))\,\overline{w}\,{\overline{h'(w)}}\bigg\}\
$$
Since $C^{1-\alpha}(\mathbb{T})$ is an algebra and using some classical H\"{o}lder embeddings, we get
\begin{eqnarray*}
\|\mathcal{L}_1(f)h\|_{C^{1-\alpha}(\mathbb(T)}\lesssim\|\phi\|_{C^{2-\alpha}}\|h\|_{C^{2-\alpha}}+\|S(\phi)\|_{C^{1-\alpha}}\|h\|_{C^{2-\alpha}}.
\end{eqnarray*}
To estimate $S(\phi)$ we use Corollary \ref{cor} combined with the estimate $\|\phi\|_{\textnormal{Lip}}+\|\phi^{-1}\|_{\textnormal{Lip}}\le C(r)$. Therefore
$$
\|S(\phi)\|_{C^{1-\alpha}(\mathbb{T})}\le C.
$$
It follows that
\begin{equation}\label{Sing001}
\|\mathcal{L}_1(f)h\|_{C^{1-\alpha}(\mathbb(T)}\le C\|h\|_{C^{2-\alpha}}.
\end{equation}
Now using  once again Corollary \ref{cor} we get that  ${A}(\phi,h)\in C^{1-\alpha}$ and
\begin{eqnarray}\label{Sing02}
\nonumber\|{A}(\phi,h)\|_{C^{1-\alpha}(\mathbb{T})}&\le& C \|h^\prime\|_{L^\infty}\\
&\le& C\|h\|_{C^{2-\alpha}(\mathbb{T})}.
\end{eqnarray}
 So, it remains to show that ${B}(f,h)$ and ${C}(f,h)$ are of class $ C^{1-\alpha}\big(\mathbb{T}\big)$.  
 For this end we set
$$
K_1(w,\tau)\triangleq\frac{\big(\phi(w)-\phi(\tau)\big)\big(\overline{h(w)}-\overline{h(\tau)}\big)}{\vert \phi(w)-\phi(\tau)\vert^{\alpha+2}}\cdot
$$
Clearly we have for $\tau\neq w\in \mathbb{T}$,
\begin{eqnarray}\label{Sing11}
\nonumber\vert K_1(w,\tau)\vert 
 \nonumber&\leq &\frac{\|\phi\|_{\textnormal{Lip}} \|h\|_{\textnormal{Lip}}}{\vert w-\tau\vert^\alpha}\\
 &\le& C\frac{ \|h\|_{C^{2-\alpha}(\mathbb{T})}}{\vert w-\tau\vert^{\alpha}}\cdot
\end{eqnarray}
Moreover, in view of the  formula \eqref{deriv-bar} we readily obtain \begin{eqnarray*}
\partial_w K_1(w,\tau)&=&{\phi'}(w)\frac{\overline{h(w)}-\overline{h(\tau)}}{\vert \phi(w)-\phi(\tau)\vert^{\alpha+2}}-\frac{\overline{h^\prime(w)}}{w^2}\frac{{\phi}(w)-{\phi}(\tau)}{\vert \phi(w)-\phi(\tau)\vert^{\alpha+2}}\\ &-&\frac{\alpha+2}{2}\bigg(\frac{\phi'(w)}{\vert \phi(w)-\phi(\tau)\vert^{\alpha+2}}-\frac{\overline{\phi'(w)}}{w^2}\frac{\big({\phi}(w)-{\phi}(\tau)\big)^2}{\vert \phi(w)-\phi(\tau)\vert^{\alpha+4}}\bigg)\Big(\overline{h(w)}-\overline{h(\tau)}\Big).
\end{eqnarray*}
Therefore one has
\begin{eqnarray}\label{Sing12}
\nonumber\big\vert\partial_w K_1(w,\tau)\big\vert &\leq&C\frac{\|\phi^\prime\|_{L^\infty}\Vert h'\Vert_{L^\infty}}{\vert w-\tau\vert^{\alpha+1}}\\
&\le& C\frac{ \|h\|_{C^{2-\alpha}(\mathbb{T})}}{\vert w-\tau\vert^{1+\alpha}}.
\end{eqnarray}
Hence, combining the inequalities  \eqref{Sing11} and  \eqref{Sing12} with Lemma \ref{noyau} we get
\begin{eqnarray}\label{Sing03}
\nonumber\|{B}(\phi,h)\|_{C^{1-\alpha}(\mathbb{T})}&\le& C \|h\|_{C^{2-\alpha}(\mathbb{T})}\|\phi^\prime\|_{L^\infty(\mathbb{T})}\\
&\le& C\|h\|_{C^{2-\alpha}(\mathbb{T})}.
\end{eqnarray}
 To estimate the last term ${C}(\phi,h)$ we observe  that 
 $${C}(\phi,h)(w)= \displaystyle{ \fint_\mathbb{T}}\overline{K_1(w,\tau)} \phi'(\tau)d\tau$$  and consequently similar proof of the estimate \eqref{Sing03} allows to get,
 \begin{eqnarray}\label{Sing04}
\nonumber\|{C}(\phi,h)\|_{C^{1-\alpha}(\mathbb{T})}&\le& C\|h\|_{C^{2-\alpha}(\mathbb{T})}.
\end{eqnarray}
By putting together this estimate with \eqref{gat66}, \eqref{Sing001}, \eqref{Sing02} and \eqref{Sing03} one concludes 
$$
\|\mathcal{L}(f)h\|_{C^{1-\alpha}(\mathbb{T})}\le C\| h\|_{C^{2-\alpha}(\mathbb{T})}.
$$
This means  that $\mathcal{L}(f)\in \mathcal{L}(X,Y).$ To achieve the proof it remains to check  that the convergence in \eqref{gat22}  towards $\mathcal{L}(f)(h)$ occurs in the strong topology of $C^{1–\alpha}(\mathbb{T}).$ The convergence of the quadratic terms containing the parameter $\Omega$ can be easily obtained from the algebra structure of  $C^{1–\alpha}(\mathbb{T}).$ Therefore the problem reduces to verify only  the convergence in the formula \eqref{dfh}. We shall check only  the convergence for the term involving $A(\phi,h)$ and the analysis for the other terms leading to $B(\phi,h)$ and $ C(\phi,h)$ is quite similar and we omit here the details.
 We start with showing 
$$
\lim_{t\to 0}\fint_\mathbb{T}\frac{h^\prime(\tau)}{\vert \{\phi(w)-\phi(\tau)\}+t\{h(w)-h(\tau)\}\vert^\alpha}d\tau=\fint_\mathbb{T}\frac{h^\prime(\tau)}{\vert \phi(w)-\phi(\tau)\vert^\alpha}d\tau\quad \hbox{in}\quad C^{1-\alpha}.
$$
Set
\begin{eqnarray*}
K(t,w,\tau)&=&\frac{1}{\vert \{\phi(w)-\phi(\tau)\}+t\{h(w)-h(\tau)\}\vert^\alpha}-\frac{1}{\vert \phi(w)-\phi(\tau)\vert^\alpha}\\
&\triangleq& g(t,w,\tau)-g(0,w,\tau).
\end{eqnarray*}
Then according to Lemma \ref{noyau}, the convergence happens provided that
$$
\big|K(t,w,\tau)\big|\le C|t|\frac{1}{|w-\tau|^{\alpha}},\quad \big|\partial_wK(t,w,\tau)\big|\le C|t|\frac{1}{|w-\tau |^{1+\alpha}}\cdot
$$
Let $t>0$ such that $t||h\|_{\textnormal{Lip}(\mathbb{T})}\le\frac12$ then
\begin{eqnarray*}
\big|K(t,w,\tau)\big|&\le& C\frac{\Big|\vert \{\phi(w)-\phi(\tau)\}+t\{h(w)-h(\tau)\}\vert^\alpha-\vert \{\phi(w)-\phi(\tau)\}\vert^\alpha\Big|}{|\tau-w|^{2\alpha}}\\
&\le&C|t|\| h\|_{\textnormal{Lip}(\mathbb{T})}|\tau-w|\frac{|\tau-w|^{\alpha-1}}{|\tau-w|^{2\alpha}}\\
&\le&C|t|\| h\|_{\textnormal{Lip}(\mathbb{T})}\frac{1}{|\tau-w|^\alpha},
\end{eqnarray*}
where we have used the inequality: for $\alpha\in (0,1),$  there exists $C_\alpha>0,$ such that 
\begin{equation}\label{eq0}
\vert a^\alpha-b^\alpha\vert \leq C_\alpha\frac{\vert a-b\vert}{a^{1-\alpha}+b^{1-\alpha}},\quad \forall  a,b\in\RR^*_+.
\end{equation}
To estimate $\partial_wK(t,w,\tau)$ we shall use the Mean value Theorem,
$$
K(t,w,\tau)=\int_0^t\partial_s g(s,w,\tau)ds
$$
and therefore 
$$
\big|\partial_wK(t,w,\tau)\big|\le\int_0^t\big|\partial_w\partial_s g(s,w,\tau)\big|ds,
$$
with 
$$
g(t,w,\tau)=\frac{1}{\vert \{\phi(w)-\phi(\tau)\}+t\{h(w)-h(\tau)\}\vert^\alpha}\cdot
$$
 Using \eqref{for22} leads to
\begin{eqnarray*}
\partial_wg(t,w,\tau)&=&\frac{-\alpha}{2}\frac{g(t,w,\tau)}{\vert \phi(w)-\phi(\tau)+t\big(h(w)-h(\tau)\big)\vert^{2}}\times\\
&&\bigg\{\big(\phi'(w)+th^\prime(w)\big)\Big({\overline{\phi(w)-\phi(\tau)}+t\overline{h(w)-h(\tau)}}\Big)
\\
&-&\frac{\overline{\phi'(w)+th^\prime(w)}}{w^2}\Big({\phi(w)-\phi(\tau)+t\big(h(w)-h(\tau)\big)}\Big)\bigg\}.
\end{eqnarray*}
Using straightforward computations yield for any $s\in [0,t],$
$$
\Big|\partial_s\partial_wg(s,w,\tau)\Big|\le C\frac{1}{|w-\tau|^{1+\alpha}}\cdot
$$
Hence we get
$$
\big|\partial_wK(t,w,\tau)\big|\le C|t|\frac{1}{|w-\tau|^{1+\alpha}}\cdot
$$
This completes the proof of the estimate  of the kernel and the required statement follows immediately.
\end{proof}
\subsection{Strong regularity}
In this subsection we shall  discuss the existence of Fr\'echet derivative of $F$ and prove  that $F$ is continuously differentiable on the domain $\RR\times B_r$. More precisely, we shell establish the following result.
 \begin{proposition}\label{string11}
For any $r\in (0,1)$  the following holds true.
\begin{enumerate}
\item $F: \RR\times B_r\to Y$ is of class $C^1.$
\item The partial derivative  $\partial_\Omega\partial_fF: \mathbb{R}\times B_r\to \mathcal{L}(X, Y)$  exists and is continuous.
\end{enumerate}
\end{proposition}
\begin{proof}
${\bf{(1)}}$ This amounts to showing that the partial derivatives $\partial_\Omega F$ and $\partial_f F$ in the G\^ateaux sense exist and are continuous. For the first derivative, we observe the the linear  dependence  of $F$ on $\Omega$  allows to get, 
\begin{eqnarray*}
\partial_\Omega F(\Omega,f)(w)&=&\textnormal{Im}\Big\{\overline{w}\, \phi(w)\, {\overline{\phi^\prime(w)}}\Big\}.
\end{eqnarray*}
 Obviously this is polynomial on $\phi$ and $\phi^\prime$ and therefore it is continuous  in the strong  topology of $X$. The next step is to  prove that for given $\Omega$,  $\partial_fF\big(\Omega,f)$ is continuous as a function of $f$ taking values in the space of bounded linear operators from $X$ to $Y$. In other words, we will show that, for a fixed $f,g\in B_r$,
\begin{equation}\label{mohimma} 
\Vert \partial_fF(\Omega,f)(h)-\partial_gF(\Omega,g)(h)\Vert_{C^{1-\alpha}(\mathbb{T})}\leq C\Vert f-g\Vert_{C^{2-\alpha}(\mathbb{T})}\Vert h\Vert_{C^{2-\alpha}(\mathbb{T})}.
\end{equation}
 Now because $C^{1-\alpha}$ is an algebra and from \eqref{dfh} and \eqref{gat66}  the problem reduces to show the required inequality for the quantities $\mathcal{L}_1(f)h$, ${A}(\phi,h)$,  ${B}(\phi,h)$ and ${C}(\phi ,h)$. The crucial tool for this task is Lemma \ref{noyau} which will be frequently used here.  We shall start with proving  the estimate
$$
\|\mathcal{L}_1(f)h-\mathcal{L}_1(g)h\|_{C^{1-\alpha}}\le C\Vert f-g\Vert_{C^{2-\alpha}(\mathbb{T})}\Vert h\Vert_{C^{2-\alpha}(\mathbb{T})}.
$$
For this end , it sufficient to establish that
$$
\|S(\phi)-S(\psi)\|_{C^{1-\alpha}}\le C\Vert f-g\Vert_{C^{2-\alpha}(\mathbb{T})},
$$
with $\phi=\hbox{Id}+f$ and $\psi=\hbox{Id}+g$. Write

\begin{eqnarray*}
S(\phi)(w)-S(\psi)(w)&=&\mathop{{\fint}}_\mathbb{T}\Big(\frac{\phi^\prime(\tau)}{\vert \phi(w)-\phi(\tau)\vert^\alpha}-\frac{\psi^\prime(\tau)}{\vert \psi(w)-\psi(\tau)\vert^\alpha}\Big)d\tau\\
&=&\mathop{{\fint}}_\mathbb{T}\Big(\frac{1}{\vert \phi(w)-\phi(\tau)\vert^\alpha}-\frac{1}{\vert \psi(w)-\psi(\tau)\vert^\alpha}\Big)\psi'(\tau)d\tau\\
&+&\mathop{{\fint}}_\mathbb{T}\frac{\phi^\prime(\tau)-\psi^\prime(\tau)}{\vert \phi(w)-\phi(\tau)\vert^\alpha}d\tau.
\end{eqnarray*}
The estimate of the last term follows immediately from Corollary \ref{cor}, that is,
\begin{eqnarray*}
\Big\|\mathop{{\fint}}_\mathbb{T}\frac{\phi^\prime(\tau)-\psi^\prime(\tau)}{\vert \phi(\cdot)-\phi(\tau)\vert^\alpha}d\tau\Big\|_{C^{1-\alpha}(\mathbb{T})}&\le& C\|f^\prime-g^\prime\|_{L^\infty}\\
&\le&C\|f-g\|_{{C^{2-\alpha}(\mathbb{T})}}.
\end{eqnarray*}
As to the estimate of  first term it can be deduced easily from the next  general one: let $T$ be the operator defined by
$$
T\chi(w)=\mathop{{\fint}}_\mathbb{T}\Big(\frac{1}{\vert \phi(w)-\phi(\tau)\vert^\alpha}-\frac{1}{\vert \psi(w)-\psi(\tau)\vert^\alpha}\Big)\chi(\tau)d\tau
$$
then
\begin{equation}\label{Singds}
\big\|T\chi\big\|_{C^{1-\alpha}(\mathbb{T})}\le C\|\psi-\phi\|_{\textnormal{Lip}(\mathbb{T})}\|\chi\|_{L^\infty(\mathbb{T})}.
\end{equation}
To prove this control we shall introduce the kernel
$$ 
K_2(w,\tau)\triangleq \frac{1}{\vert \phi(w)-\phi(\tau)\vert^\alpha}-\frac{1}{\vert \psi(w)-\psi(\tau)\vert^\alpha}
$$
and prove that it satisfies  the estimates, 
$$\vert K_2(w,\tau)\vert\lesssim \frac{\Vert \psi'-\phi'\Vert_{L^\infty}}{\vert w-\tau\vert^\alpha}\quad\textnormal{and}\quad \vert \partial_wK_2(w,\tau)\vert\lesssim \frac{\Vert \psi'-\phi'\Vert_{L^\infty}}{\vert w-\tau\vert^{\alpha+1}}.$$
Whence these estimates are proved we can then apply Lemma \ref{noyau} and get the desired result.
The first estimate is easy to obtain by using \eqref{eq0}. On other hand, in view of \eqref{deriv-bar} the derivative of $K_2(w,\tau)$ with respect to $w$ is given by
\begin{eqnarray*}
\partial_w K_2(w,\tau)&=& -\frac{\alpha}{2}\Big(\overline{\mathcal{I}(w,\tau)}-\frac{\mathcal{I}(w,\tau)}{w^2}\Big)
\end{eqnarray*}
where
\begin{eqnarray*}
\mathcal{I}(w,\tau)&\triangleq&\overline{\phi'(w)}\frac{\phi(w)-\phi(\tau)}{\vert \phi(w)-\phi(\tau)\vert^{\alpha+2}}-\overline{\psi'(w)}\frac{\psi(w)-\psi(\tau)}{\vert \psi(w)-\psi(\tau)\vert^{\alpha+2}}\cdot
\end{eqnarray*}
We shall transform this quantity into, 
$$
\mathcal{I}(w,\tau)= \mathcal{I}_1(w,\tau)+\mathcal{I}_2(w,\tau)+\mathcal{I}_3(w,\tau),
$$
with
$$
\mathcal{I}_1(w,\tau)\triangleq \overline{\phi'(w)}\, \frac{({\phi}-{\psi})(w)-({\phi}-{\psi})(\tau)}{\vert \phi(w)-\phi(\tau)\vert^{\alpha+2}},
$$
$$
\mathcal{I}_2(w,\tau)\triangleq \big(\overline{\phi'(w)}-\overline{\psi'(w)}\big)\frac{{\psi}(w)-{\psi}(\tau)}{\vert \psi(w)-\psi(\tau)\vert^{\alpha+2}},
 $$
 and
 $$
 \mathcal{I}_3(w,\tau)\triangleq \overline{\phi'(w)}\big({\psi}(\tau)-{\psi}(w)\big)\frac{ \vert \phi(w)-\phi(\tau)\vert^{\alpha+2}-\vert \psi(w)-\psi(\tau)\vert^{\alpha+2}}{\vert \phi(w)-\phi(\tau)\vert^{\alpha+2}\vert \psi(w)-\psi(\tau)\vert^{\alpha+2}}.
$$
For the first and the second term one readily gets
\begin{eqnarray}\label{i1+i2}
\vert\mathcal{I}_1(w,\tau)\vert+\vert \mathcal{I}_2(w,\tau)\vert &\lesssim & \frac{\|\phi-\psi\|_{\textnormal{Lip}(\mathbb{T})}}{\vert w-\tau\vert^{\alpha+1}}.
\end{eqnarray}
Concerning  the last term we shall use the following inequality whose proof is classical.
\begin{equation}\label{eq}
\vert a^{k+1+\alpha}-b^{k+1+\alpha}\vert \leq C(k,\alpha)\vert a-b\vert \big( a^{k+\alpha}+b^{k+\alpha}\big),\quad a,b\in\RR_+,k\in \NN^*,\,0<\alpha <1.
\end{equation}
Thus we find
$$
 \Big|\vert \phi(w)-\phi(\tau)\vert^{\alpha+2}-\vert \psi(w)-\psi(\tau)\vert^{\alpha+2}\Big|\le C\|\phi-\psi\|_{\textnormal{Lip}(\mathbb{T})}\, |\tau-w|^{\alpha+2}
$$
and consequently,
\begin{eqnarray}\label{i03}
\vert\mathcal{I}_3(w,\tau)\vert\leq C\frac{\|\phi-\psi\|_{\textnormal{Lip}(\mathbb{T})}}{\vert w-\tau\vert^{\alpha+1}}.
\end{eqnarray}
Putting together \eqref{i1+i2} and  \eqref{i03} we find, 
$$
\vert\mathcal{I}(w,\tau)\vert\leq C\frac{\|\phi-\psi\|_{\textnormal{Lip}(\mathbb{T})}}{\vert w-\tau\vert^{\alpha+1}}.
$$
Therefore
$$
\vert\partial_w K_2(w,\tau)\vert\leq C\frac{\|\phi-\psi\|_{\textnormal{Lip}(\mathbb{T})}}{\vert w-\tau\vert^{\alpha+1}}.
$$
This achieves the suitable estimates for the kernel $K_2.$ Let us now move to  the continuity  estimate of ${A}(\phi,h)$. We write from the definition, 
\begin{eqnarray}
{A}(\phi,h)(w)-{A}(\psi,h)(w)=\mathop{{\fint}}_\mathbb{T}\Big(\frac{1}{\vert \phi(w)-\phi(\tau)\vert^\alpha}-\frac{1}{\vert \psi(w)-\psi(\tau)\vert^\alpha}\Big)h'(\tau)d\tau
\end{eqnarray}
Using \eqref{Singds} we immediately obtain,
\begin{eqnarray}\label{a}
\nonumber \big\Vert {A}(\phi,h)-{A}(\psi,h)\big\Vert_{C^{1-\alpha}}& \leq &C\Vert \phi-\psi\Vert_{\textnormal{Lip}(\mathbb{T})}\Vert h\Vert_{\textnormal{Lip}(\mathbb{T})}\\
& \leq &C\Vert f-g\Vert_{\textnormal{Lip}(\mathbb{T})}\Vert h\Vert_{C^{2-\alpha}(\mathbb{T})}.
\end{eqnarray}
This completes the proof of the estimate of  the term $A(\phi, h)$ which fits with \eqref{mohimma}.

Now we shall investigate  the continuity  estimate of $B(\phi,h)$ defined in \eqref{dfh}.
According to this definition, one has
\begin{eqnarray*}
\big({B}(\phi,h)-{B}(\psi,h)\big)(w)= \mathop{{\fint}}_\mathbb{T}K_3(w,\tau)d\tau,
\end{eqnarray*}
with
$$
K_3(w,\tau)\triangleq \Bigg({\phi'}(\tau)\frac{{\phi}(w)-{\phi}(\tau)}{\vert \phi(w)-\phi(\tau)\vert^{\alpha+2}}-{\psi'}(\tau)\frac{{\psi}(w)-{\psi}(\tau)}{\vert \psi(w)-\psi(\tau)\vert^{\alpha+2}}\Bigg)\Big(\overline{h(w)}-\overline{h(\tau)}\Big).
$$
We can rewrite this kernel in the form,
$$
K_3(w,\tau)= \Big(K_3^1(w,\tau)+K_3^2(w,\tau)+K_3^3(w,\tau)\Big)\Big(\overline{h(w)}-\overline{h(\tau)}\Big),
$$
with
$$
K_3^1(w,\tau)\triangleq \phi'(\tau)\,\frac{({\phi}-{\psi})(w)-({\phi}-{\psi})(\tau)}{\vert \phi(w)-\phi(\tau)\vert^{\alpha+2}},
$$
$$
K_3^2(w,\tau)\triangleq \big(\phi'-\psi'\big)(\tau)\frac{{\psi}(w)-{\psi}(\tau)}{\vert \psi(w)-\psi(\tau)\vert^{\alpha+2}},
$$
and
$$
K_3^3(w,\tau)\triangleq\phi'(\tau)\big({\psi}(\tau)-{\psi}(w)\big)\frac{\vert \phi(w)-\phi(\tau)\vert^{\alpha+2}-\vert \psi(w)-\psi(\tau)\vert^{\alpha+2}}{\vert \psi(w)-\psi(\tau)\vert^{\alpha+2}\vert \phi(w)-\phi(\tau)\vert^{\alpha+2}}.
$$

In view of the inequality \eqref{eq} we may conclude that
\begin{eqnarray*}
\vert K_3^1(w,\tau)\vert+\vert K_3^2(w,\tau)\vert+\vert K_3^3(w,\tau)\vert &\leq & C\frac{\Vert \phi-\psi\Vert_{\textnormal{Lip}(\mathbb{T})}}{\vert w-\tau\vert^{\alpha+1}}\\
&\le & C\frac{\Vert f-g\Vert_{\textnormal{Lip}(\mathbb{T})}}{\vert w-\tau\vert^{\alpha+1}}\cdot
\end{eqnarray*}
Consequently we find for $w\neq\tau\in \mathbb{T}$
\begin{equation}\label{Noy1}
\vert K_3(w,\tau)\vert\leq C\|h\|_{\textnormal{Lip}(\mathbb{T})}\frac{\Vert f-g\Vert_{\textnormal{Lip}(\mathbb{T})}}{\vert w-\tau\vert^{\alpha}}\cdot 
\end{equation}
Now we intend to estimate $\partial_w K_3(w,\tau)$. Easy computations yield
\begin{eqnarray}\label{par3}
\nonumber \partial_w K_3(w,\tau)&=&\Big(-\frac{\alpha}{2}\mathcal{N}_1(w,\tau)-\big(1+\frac{\alpha}{2}\big) \mathcal{N}_2(w,\tau)\Big)\big(\overline{h(w)}-\overline{h(\tau)}\big)\\
&-& \mathcal{N}_3(w,\tau)\frac{\overline{h'(w)}}{w^2},
\end{eqnarray}
with
$$
 \mathcal{N}_1(w,\tau)\triangleq \frac{\phi'(\tau){\phi'}(w)}{\vert \phi(w)-\phi(\tau)\vert^{\alpha+2}}-\frac{\psi'(\tau){\psi'}(w)}{\vert \psi(w)-\psi(\tau)\vert^{\alpha+2}},
$$
$$
 \mathcal{N}_2(w,\tau) \triangleq \frac{\overline{\phi'(w)}}{w^2}\frac{\phi'(\tau)\big({\phi}(w)-{\phi}(\tau)\big)^2 }{\vert \phi(w)-\phi(\tau)\vert^{\alpha+4}}-\frac{\overline{\psi'(w)}}{w^2}\frac{\psi'(\tau)\big({\psi}(w)-{\psi}(\tau)\big)^2 }{\vert \psi(w)-\psi(\tau)\vert^{\alpha+4}},
$$
and
$$
 \mathcal{N}_3(w,\tau)\triangleq \phi'(\tau)\frac{{\phi}(w)-{\phi}(\tau)}{\vert \phi(w)-\phi(\tau)\vert^{\alpha+2}}-\psi'(\tau)\frac{{\psi}(w)-{\psi}(\tau)}{\vert \psi(w)-\psi(\tau)\vert^{\alpha+2}}\cdot
$$
The estimate of the last term $\mathcal{N}_3$ can be done exactly as for $\mathcal{I}$. Concerning $ \mathcal{N}_1(w,\tau)$ we may write 
\begin{eqnarray*}
 \mathcal{N}_1(w,\tau)&=& ({\phi'(w)}-{\psi'(w)})\frac{\phi'(\tau)}{\vert \phi(w)-\phi(\tau)\vert^{\alpha+2}}+\big(\phi'(\tau)-\psi'(\tau)\big)\frac{{\psi'}(w)}{\vert \psi(w)-\psi(\tau)\vert^{\alpha+2}}\\ &-&\phi'(\tau){\psi'}(w)\frac{\vert \phi(w)-\phi(\tau)\vert^{\alpha+2}-\vert \psi(w)-\psi(\tau)\vert^{\alpha+2}}{\vert \phi(w)-\phi(\tau)\vert^{\alpha+2}\vert \psi(w)-\psi(\tau)\vert^{\alpha+2}}\cdot
\end{eqnarray*}
Hence, using inequality \eqref{eq} we immediately deduce that
\begin{eqnarray*}
\vert  \mathcal{N}_1(w,\tau)\vert&\leq& \frac{C\Vert \phi'-\psi'\Vert_{L^\infty}}{\vert w-\tau\vert^{\alpha+2}}\cdot
\end{eqnarray*}

Now we shall split the term $ \mathcal{N}_2(w,\tau)$ as follows,
\begin{eqnarray*}
 \mathcal{N}_2(w,\tau)&= &\sum_{k=1}^4 \mathcal{N}_{2,k}(w,\tau)
 \end{eqnarray*}
 with
 $$
 \mathcal{N}_{2,1}(w,\tau)= \phi'(\tau)\big(\overline{\phi'(w)}-\overline{\psi'(w)}\big)\frac{\big({\phi}(w)-{\phi}(\tau)\big)^2}{w^2\vert \phi(w)-\phi(\tau)\vert^{\alpha+4}},
$$
$$
 \mathcal{N}_{2,2}(w,\tau)=\phi'(\tau)\overline{\psi'(w)}\,\frac{\big({\phi}(w)-{\phi}(\tau)\big)^2-\big({\psi}(w)-{\psi}(\tau)\big)^2}{w\vert \phi(w)-\phi(\tau)\vert^{\alpha+4}},
$$
$$
 \mathcal{N}_{2,3}(w,\tau)= \big(\phi'(\tau)-\psi'(\tau)\big)\overline{\psi^\prime(w)}\frac{\big({\psi}(w)-{\psi}(\tau)\big)^2}{w^2\vert \phi(w)-\phi(\tau)\vert^{\alpha+4}},
$$
and
\begin{eqnarray*}
 \mathcal{N}_{2,4}(w,\tau)&=&\psi'(\tau)\overline{\psi^\prime(w)}\big({\psi}(w)-{\psi}(\tau)\big)^2\frac{\vert \psi(w)-\psi(\tau)\vert^{\alpha+4}-\vert \phi(w)-\phi(\tau)\vert^{\alpha+4}}{\vert \psi(w)-\psi(\tau)\vert^{\alpha+4}\vert \phi(w)-\phi(\tau)\vert^{\alpha+4}}.
\end{eqnarray*}
Similar  computations as before lead to, 
\begin{eqnarray*}
\vert  \mathcal{N}_2(w,\tau)\vert  &\lesssim &  \frac{\Vert \phi'-\psi'\Vert_{L^\infty}}{\vert w-\tau\vert^{\alpha+2}}\cdot
\end{eqnarray*}
Hence, in view of the identity \eqref{par3} we obtain
\begin{eqnarray}\label{Noy2}
\vert  \partial_w K_3(w,\tau)\vert  &\lesssim &  \frac{\Vert h'\Vert_{L^\infty}\Vert \phi'-\psi'\Vert_{L^\infty}}{\vert w-\tau\vert^{\alpha+1}}.
\end{eqnarray}
At this stage we can use \eqref{Noy1}, \eqref{Noy2} and Lemma \ref{noyau},
\begin{eqnarray}\label{b}
\nonumber\big\Vert {B}(\phi,h)-{B}(\psi,h)\big\Vert_{C^{1-\alpha}} &\lesssim& \Vert f-g\Vert_{\textnormal{Lip}(\mathbb{T})}\Vert h\Vert_{\textnormal{Lip}(\mathbb{T})}\\
&\lesssim&\Vert f-g\Vert_{C^{2-\alpha}(\mathbb{T})}\Vert h\Vert_{C^{2-\alpha}(\mathbb{T})}.
\end{eqnarray}
Finally, we observe from \eqref{dfh} that
$$
{C}(\phi,h)(w)-{C}(\psi,h)(w)=\fint_{\mathbb{T}} \overline{K_3(w,\tau)} \phi^\prime(\tau) d\tau
$$
and therefore we find similar estimate to \eqref{b}. This ends the proof of \eqref{mohimma}.
\vspace{0,3cm}

{$\bf{(2)}$} Now we shall compute $\partial_\Omega\partial_f F(\Omega, f)$ and prove the continuity of this function. Let $f\in B_r$ and $h\in C^{2-\alpha}(\mathbb{T})$ be a fixed direction, then in view of \eqref{gat22} one has
$$
\partial_\Omega\partial_f F(\Omega, f)h(w)=  \textnormal{Im} \bigg\{\phi(w)\,\overline{w}\,{\overline{h'(w)}}+h(w)\,\overline{w}\,{\overline{\phi'(w)}}\bigg\}.
$$
 It follows  that for $f,g\in B_r$,
\begin{eqnarray*}
\|\partial_\Omega\partial_f F(\Omega, f)h-\partial_\Omega\partial_f F(\Omega, f)h\|_{C^{1-\alpha}(\mathbb{T})}&\le& C\|f-g\|_{C^{1-\alpha}(\mathbb{T})}\|h\|_{C^{2-\alpha}(\mathbb{T})}\\
&\le& C\|f-g\|_{C^{2-\alpha}(\mathbb{T})}\|h\|_{C^{2-\alpha}(\mathbb{T})}.
\end{eqnarray*}
This proves the continuity of $\partial_\Omega\partial_f F(\Omega,f):\mathbb{R}\times B_{r}\to \mathcal{L}(X,Y)$ and therefore the proof of the second point is now achieved.
\end{proof}

\section{Spectral study}
In this section we concentrate on the spectral study of the linearized operator of $F$ around zero and denoted by  $\partial_f F (\Omega, 0)$. We shall peculiarly  look for the values of $\Omega$ where  the kernel   is non trivial. We will be seeing that  the kernel is necessarily simple and all the required assumptions  of the   C-R Theorem are satisfied.  According to the Proposition \ref{string11}, the functional $F:\RR\times B_r\to Y$ is $C^1$ and therefore Gâteaux and Fr\'echet derivatives with respect to $f$ and in the direction $h\in X$ coincide. Now putting together the    formulas \eqref{dfh} and \eqref{gat66} with  $\phi=\textnormal{Id}$, we find
\begin{eqnarray}\label{df00}
\partial_fF(\Omega,0)h(w)&=&\textnormal{Im}\bigg\{\Omega\Big(\overline{h'(w)}+\frac{h(w)}{w}\Big)-{C_\alpha}\frac{\overline{h'(w)}}{w}\fint_\mathbb{T}\frac{d\tau}{\vert w-\tau\vert^\alpha}-\frac{C_\alpha}{w}\fint_\mathbb{T}\frac{h'(\tau)}{\vert w-\tau\vert^\alpha}d\tau\notag\\ &+& \frac{\alpha C_\alpha}{2w}\fint_\mathbb{T}\frac{(w-\tau)\big(\overline{h(w)}-\overline{h(\tau)}\big)}{\vert w-\tau\vert^{\alpha+2}}d\tau+\frac{\alpha C_\alpha}{2w}\fint_\mathbb{T}\frac{(\overline{w}-\overline{\tau})\big(h(w)-h(\tau)\big)}{\vert w-\tau\vert^{\alpha+2}}d\tau\bigg\}\notag\\ &\triangleq& \textnormal{Im}\Big\{\hbox{I}_1(h(w))+\hbox{I}_2(h(w))+\hbox{I}_3(h(w))+\hbox{I}_4(h(w))+\hbox{I}_5(h(w))\Big\} .
\end{eqnarray}
Recall that the spaces $X$ and $Y$ are successively  given by,
$$
X=\Big\{f\in C^{2-\alpha}(\mathbb{T}),\, f(w)=\sum_{n\geq 0}b_n\overline{w}^n, b_n\in \RR,\, w\in\mathbb{T}\Big\}
$$
and
$$
Y=\Big\{g\in C^{1-\alpha}(\mathbb{T}),\, g(w)=i\sum_{n\geq 1}g_n\big(w^n-\overline{w}^n\big), g_n\in \RR,\, w\in\mathbb{T}\Big\}.
$$
To state our main result we shall  introduce a special set $\mathcal{S}$  describing the dispersion relation  which plays a central role in  the bifurcation of non trivial solutions.
\begin{equation}\label{disper}
\mathcal{S}\triangleq \Bigg\{\Omega\in \RR, \,\exists\,  m\geq 2, \quad \Omega=  
\Omega_m^\alpha\triangleq\frac{\Gamma(1-\alpha)}{2^{1-\alpha}\Gamma^2(1-\frac\alpha2)}\bigg(\frac{\Gamma(1+\frac\alpha2)}{\Gamma(2-\frac\alpha2)}-\frac{\Gamma(m+\frac\alpha2)}{\Gamma(m+1-\frac\alpha2)}\bigg)
\Bigg\}.
\end{equation}
 We shall discuss soon  some elementary properties   of this set. Now we state our result.
\begin{proposition}\label{propgh}
The following assertions hold true. 
\begin{enumerate}
\item The kernel  of $\partial_fF(\Omega,0)$ is non trivial if and only if $\Omega=\Omega_m^\alpha\in \mathcal{S}$ and, in this case,  it is   one-dimensional vector space generated by
$$
v_m(w)=\overline{w}^{m-1}.
$$  

\item The range of  $\partial_fF(\Omega_m^\alpha,0)$ is closed in $Y$ and  is of co-dimension one. It  is given by 
$$
R\big(\partial_fF(\Omega_m^\alpha,0)\Big)=\Big\{g\in C^{1-\alpha}(\mathbb{T});\quad g(w)=i\sum_{n\geq1\atop\\ n\neq m}^\infty g_n(w^{n}-\overline{w}^{n}), g_n\in \RR\Big\}.
$$
\item Transversality assumption:   
$$
\partial_\Omega\partial_f F(\Omega_m^\alpha,0)(v_m)\not\in R\big(\partial_fF(\Omega_m^\alpha,0)\Big).
$$
\end{enumerate}
\end{proposition}
Before proving this result we  collect some properties on the asymptotic behavior of the sequence $\{\Omega_n^\alpha\}$ with respect to $\alpha $ and $n$. This is  summarized in the next lemma.
\begin{lemma}\label{lemz1}
We have the following results.
\begin{enumerate}
\item Let  $n\geq2$, then
$${\lim_{\alpha\to0}\,\Omega_n^\alpha=\frac{n-1}{2n}},\qquad {\lim_{\alpha\to1}}\,\Omega_n^\alpha=\frac{2}{\pi}\sum_{k=1}^{n-1}\frac{1}{2k+1}\cdot
$$
\item For any $\alpha\in (0,1)$, we get  $\Omega_n^\alpha>0$ and $n\mapsto \Omega_n^\alpha$ is strictly increasing.
Moreover, 
$$
\mathcal{S}\subset \Theta_\alpha\Big[\frac{1-\alpha}{2-\frac{\alpha}{2}},1\Big[,
$$
with
$$
\Theta_\alpha\triangleq\frac{2^\alpha\Gamma(1+\frac\alpha2)\Gamma(1-\alpha)}{(2-\alpha)\Gamma^3(1-\frac\alpha2)}\cdot
$$
\item For $\alpha\in (0,1)$ fixed and $n$ sufficiently large,
\begin{equation}\label{As1}
\Omega_n^\alpha=\Theta_\alpha-\big(1-{\alpha}/{2}\big)\Theta_\alpha\frac{e^{\alpha\gamma+c_\alpha}}{n^{1-\alpha}}+O\Big(\frac{1}{n^{2-\alpha}}\Big), \quad 
\end{equation}
where $\gamma$ denotes Euler constant,  $c_\alpha$ is the sum of the series 
$$
c_\alpha\triangleq\sum_{m=1}^\infty\frac{\alpha^{2m+1}}{2^{2m-1}(2m+1)}\zeta(2m+1).
$$
and $s\mapsto \zeta(s)$ is the Riemann zeta function.
\end{enumerate}
\end{lemma} 
\begin{proof}
${\bf(1)}$ Recall first that for $n\geq2,$
$$
\Omega_n^\alpha\triangleq\frac{\Gamma(1-\alpha)}{2^{1-\alpha}\Gamma^2(1-\frac\alpha2)}\bigg(\frac{\Gamma(1+\frac\alpha2)}{\Gamma(2-\frac\alpha2)}-\frac{\Gamma(n+\frac\alpha2)}{\Gamma(n+1-\frac\alpha2)}\bigg).
$$
Passing to the limit in this formula when $\alpha$ goes to zero yields
\begin{eqnarray*}
\lim_{\alpha\to 0}\Omega_n^\alpha&=&\frac12\Big(\frac{\Gamma(1)}{\Gamma(2)}-\frac{\Gamma(n)}{\Gamma(n+1)}\Big)\\
&=&\frac12\Big(1-\frac{(n-1)!}{n!}\Big)\\
&=&\frac{n-1}{2n}\cdot
\end{eqnarray*}
As to the second limit, we shall introduce for a fixed $n$ the function
$$
\phi_n(\alpha)=\frac{\Gamma(n+\alpha/2)}{\Gamma(n+1-\alpha/2)}\cdot
$$
Therefore we obtain according to \eqref{Gamma1} and \eqref{for1} and the relation $\phi_n(1)=1,$
\begin{eqnarray*}
\lim_{\alpha\to 1}\Omega_n^\alpha&=&\frac{-1}{\Gamma^2(1/2)}\lim_{\alpha\to 1}\big\{(1-\alpha)\Gamma(1-\alpha)\big\}\lim_{\alpha\to 1}\Big\{\frac{\phi_1(\alpha)-\phi_1(1)}{\alpha-1}-\frac{\phi_n(\alpha)-\phi_n(1)}{\alpha-1}\Big\}\\
&=&\frac{-1}{\pi}\Big\{{\phi_1^\prime(1)}-{\phi_n^\prime(1)}\Big\}.
\end{eqnarray*}
By applying the logarithm function to $\phi_n$ and differentiating with respect to $\alpha$ one obtains the relation
$$
2\frac{\phi_n^\prime(\alpha)}{\phi_n(\alpha)}=\digamma(n+\alpha/2)+\digamma(n+1-\alpha/2).
$$
Now using the fact that $\phi_n(1)=1$ combined with the preceding identity and \eqref{digam}, we find
\begin{eqnarray*}
\lim_{\alpha\to 1}\Omega_n^\alpha
&=&\frac{-1}{\pi}\Big\{{\digamma(3/2)}-{\digamma(n+1/2)}\Big\}\\
&=&\frac{2}{\pi}\sum_{k=1}^{n-1}\frac{1}{2k+1},
\end{eqnarray*}
which is the desired result.
\vspace{0,5cm}

${\bf(2)}$ Using the identities \eqref{Poc} we find the alternative formula
\begin{equation}\label{disper00}
\Omega_n^\alpha\triangleq \Theta_\alpha\bigg(1-\frac{\big(1+\frac{\alpha}{2}\big)_{n-1}}{\big(2-\frac{\alpha}{2}\big)_{n-1}}\bigg),
\end{equation}
with
$$
\Theta_\alpha\triangleq\frac{2^\alpha\Gamma(1+\frac\alpha2)\Gamma(1-\alpha)}{(2-\alpha)\Gamma^3(1-\frac\alpha2)}\cdot
$$
and $(x)_n$ denotes Pokhhammer's symbol introduced in \eqref{Poch}.  Now because  $x\mapsto (x)_{n-1}$ is increasing in the set $\RR_+$ provided that $\alpha<1$ we conclude easily that  $\Omega_n^\alpha>0$. 

To prove that $n\mapsto \Omega_n^\alpha$ is strictly increasing, it suffices according to \eqref{disper00} to check that 
 the sequence $n\mapsto u_n= \frac{\big(1+\frac{\alpha}{2}\big)_{n-1}}{\big(2-\frac{\alpha}{2}\big)_{n-1}}$ is strictly decreasing. This follows from the obvious fact that  for  $\alpha\in(0,1),$ one has
$$
\frac{u_{n+1}}{u_n}=\frac{n+\frac\alpha2}{n+1-\frac\alpha2}<1.
$$
From this it is apparent  that 
$$
\mathcal{S}\subset \big[\Omega_2^\alpha,\lim_{n\to\infty}\Omega_n^\alpha[\subset  \Theta_\alpha\Big[\frac{1-\alpha}{2-\frac{\alpha}{2}},1\Big[.
$$ 
Note that we have used in the last limit that  $\displaystyle{\lim_{n\to\infty}u_n=0}$ which can be deduced for instance from the  proof of the point ${\bf{(3)}}$ of this lemma.

\vspace{0,5cm}

{$\bf{(3)}$} First recall that  Riemann zeta  function is defined by 
$$
\zeta(s)=\sum_{n=1}^\infty\frac{1}{n^s}, s>1.
$$
To get the required asymptotic behavior we shall first study the sequence,
\begin{eqnarray*}
U_n\triangleq\log\bigg(\frac{\big(1+\frac{\alpha}{2}\big)_{n}}{\big(1-\frac{\alpha}{2}\big)_{n}}\bigg).
\end{eqnarray*}
Making use of the definition of $(x)_n$, we can rewrite this sequence in the manner
$$
U_n=\sum_{k=1}^{n}\Big\{\log\Big(1+\frac{\alpha}{2k}\Big)-\log\Big(1-\frac{\alpha}{2k}\Big)\Big\}.
$$
Using the Taylor expansion  of $\log(1+x)$ around zero one gets
\begin{eqnarray*}
U_n&=&\sum_{k=1}^{n}\Big\{2\sum_{m=0}^\infty \frac{1}{2m+1}\Big(\frac{\alpha}{2k}\Big)^{2m+1}\Big\}\\ &=& \sum_{k=1}^{n}\frac{\alpha}{k}+2\sum_{m=1}^\infty\frac{\alpha^{2m+1}}{2^{2m}(2m+1)}\sum_{k=1}^{n} \frac{1}{k^{2m+1}}
\\ &=& \sum_{k=1}^{n}\frac{\alpha}{k}+2\sum_{m=1}^\infty\bigg(\frac{\alpha^{2m+1}}{2^{2m}(2m+1)}\zeta(2m+1)+O\Big(\frac{1}{n^{2m}}\Big)\bigg)\\
&=&\sum_{k=1}^{n}\frac{\alpha}{k}+c_\alpha+O\Big(\frac{1}{n^2}\Big).
\end{eqnarray*}
Note that we have used the following estimate for the remainder term of the zeta function
$$
\sum_{k=n+1}^{\infty} \frac{1}{k^{2m+1}}\le \frac{1}{n^{2m}}\cdot
$$
Now we use the classical expansion  of the harmonic series
$$
\sum_{k=1}^{n}\frac{1}{k}=\log n+\gamma+O\Big(\frac{1}{n}\Big)
$$
with $\gamma $ the Euler constant. Therefore we get
$$
U_n=\alpha\log n+\alpha\gamma+c_\alpha+O\Big(\frac{1}{n}\Big)
$$
and consequently by raising to the exponential we find
\begin{eqnarray*}
e^{U_n}&=&\frac{\big(1+\frac{\alpha}{2}\big)_{n}}{\big(1-\frac{\alpha}{2}\big)_{n}}\\
&=& e^{\alpha\gamma+c_\alpha}n^\alpha e^{O(1/n)}\\
&=&e^{\alpha\gamma+c_\alpha}n^\alpha+O\Big(\frac{1}{n^{1-\alpha}}\Big).
\end{eqnarray*}
It is apparent that
\begin{eqnarray*}
\frac{\big(1+\frac{\alpha}{2}\big)_{n-1}}{\big(2-\frac{\alpha}{2}\big)_{n-1}}&=&\frac{1-\frac\alpha2}{n-\frac\alpha2}\frac{\big(1+\frac{\alpha}{2}\big)_{n-1}}{\big(1-\frac{\alpha}{2}\big)_{n-1}}\\
&=&\frac{1-\frac\alpha2}{n-\frac\alpha2}e^{U_{n-1}}
\end{eqnarray*}
and consequently by making appeal to  the formula \eqref{disper00} we obtain
\begin{eqnarray*}
\Omega_n^\alpha&=& \Theta_\alpha\Big(1-\frac{1-\frac\alpha2}{n-\frac\alpha2}e^{U_{n-1}}\Big)\\
&=& \Theta_\alpha\Big(1-\big({1-\frac\alpha2}\big)\frac{e^{\alpha\gamma+c_\alpha}}{n^{1-\alpha}}\Big)+O\Big(\frac{1}{n^{2-\alpha}}\Big).
\end{eqnarray*}
This concludes the proof of Lemma \ref{lemz1}.
\end{proof}
In what follows we shall give the proof of the Proposition \ref{propgh}.
\begin{proof}

${\bf{(1)}}$ 
We begin by calculating $\hbox{I}_1(h)$ in \eqref{df00} which is easy compared to the other terms. Let $h\in X$ taking the form $\displaystyle{h(w)=\sum_{n\geq0}\frac{b_n}{{w}^n}}$, then straightforward computations give
\begin{equation}\label{i001}
\hbox{I}_1(h(w))=\Omega\sum_{n\geq0}\Big( b_n \overline{w}^{n+1}-nb_n w^{n+1}\Big).
\end{equation}
To compute  the second term $I_2(h(w))$ we write
\begin{eqnarray*}
\hbox{I}_2(h(w))&\triangleq&-{C_\alpha}\overline{w}\,{\overline{h'(w)}}\fint_\mathbb{T}\frac{d\tau}{\vert w-\tau\vert^\alpha}\notag\\ &=&{C_\alpha}\sum_{n\geq 1}nb_nw^{n}\fint_\mathbb{T}\frac{d\tau}{\vert w-\tau\vert^\alpha}\cdot
\end{eqnarray*}
Applying the formula \eqref{In} with $n=0$ we get
\begin{eqnarray}\label{i3}
\hbox{I}_2(h(w))= \frac{\alpha C_\alpha\Gamma(1-\alpha)}{(2-\alpha)\Gamma^2(1-\alpha/2)}\sum_{n\geq 1}nb_nw^{n+1}.
\end{eqnarray}
Regarding  the third term $\hbox{I}_3(h(w))$ it may be rewritten in the manner
\begin{eqnarray*}
\hbox{I}_3(h(w)) &\triangleq&-{C_\alpha}\overline{w}\fint_\mathbb{T}\frac{h'(\tau)}{\vert w-\tau\vert^\alpha}d\tau\notag\\ &=&{C_\alpha}\sum_{n\geq 1}nb_n\overline{w}\fint_\mathbb{T}\frac{\overline{\tau}^{n+1}}{\vert w-\tau\vert^\alpha}d\tau.\notag\\ \end{eqnarray*}
Using change of variables allows to get
$$
\fint_\mathbb{T}\frac{\overline{\tau}^{n+1}}{\vert w-\tau\vert^\alpha}d\tau=\overline{w}^n{\fint_\mathbb{T}\frac{{\tau}^{n-1}}{\vert 1-\tau\vert^\alpha}}d\tau
$$
which yields in view of  the formula  \eqref{In} to the expression
\begin{equation}\label{i2}
\hbox{I}_3(h(w))=\frac{C_\alpha\Gamma(1-\alpha)}{\Gamma^2(1-\alpha/2)}\sum_{n\geq 1}nb_n\frac{\big(\frac{\alpha}{2}\big)_n}{\big(1-\frac{\alpha}{2}\big)_n}\overline{w}^{n+1}.
\end{equation}
Concerning the term $I_4(h(w))$ we start with the expansion,
\begin{eqnarray}\label{I4}
\hbox{I}_4(h(w))&\triangleq& \frac{\alpha C_\alpha}{2}\fint_\mathbb{T}\frac{(w-\tau)\big(\overline{h(w)}-\overline{h(\tau)}\big)}{w\vert w-\tau\vert^{\alpha+2}}d\tau\notag\\ &=& \frac{\alpha C_\alpha}{2}\sum_{n\geq 1}b_n\fint_\mathbb{T}\frac{(w-\tau)(w^{n}-\tau^{n})}{w\vert w-\tau\vert^{\alpha+2}}d\tau.
\end{eqnarray}
Hence,  using the identity  \eqref{Jn} one gets
\begin{eqnarray}\label{i4}
\hbox{I}_4(h(w))&=& \frac{\alpha(1+\frac{\alpha}{2})}{2(2-\alpha)} \frac{C_\alpha\Gamma(1-\alpha)}{\Gamma^2(1-\alpha/2)}\sum_{n\geq 1}b_n\bigg(1-\frac{\big(2+\frac{\alpha}{2}\big)_n}{\big(2-\frac{\alpha}{2}\big)_n}\bigg)w^{n+1}.
\end{eqnarray}
It remains to compute the last term $\hbox{I}_5$ of \eqref{der1} which can be written in the form
\begin{eqnarray}\label{h45}
\nonumber \hbox{I}_5(h(w))&=& \frac{\alpha C_\alpha}{2}\fint_\mathbb{T}\frac{(\overline{w}-\overline{\tau})\big(h(w)-h(\tau)\big)}{w\vert w-\tau\vert^{\alpha+2}}d\tau\\ 
\nonumber&=& \frac{\alpha C_\alpha}{2}\sum_{n\geq 1}b_n\fint_\mathbb{T}\frac{(\overline{w}-\overline{\tau})(\overline{w}^{n}-\overline{\tau}^{n})}{w\vert w-\tau\vert^{\alpha+2}}d\tau.
\end{eqnarray}
Using the  identity \eqref{Kn} gives
\begin{eqnarray}\label{i5}
\hbox{I}_5 (h(w))&=& -\frac{\alpha C_\alpha \Gamma(1-\alpha)}{4\Gamma^2(1-\alpha/2)}\sum_{n\geq 1}b_n\bigg(1-\frac{\big(\frac{\alpha}{2}\big)_n}{\big(-\frac{\alpha}{2}\big)_n}\bigg)\overline{w}^{n+1}.
\end{eqnarray}
Collecting the  identities \eqref{i2}, \eqref{i5} and using \eqref{f1s} we find 
\begin{eqnarray*}
\hbox{I}_3(h(w))+\hbox{I}_5(h(w))&=&\frac{C_\alpha\Gamma(1-\alpha)}{2\Gamma^2(1-\alpha/2)}\sum_{n\geq 1}b_n\bigg(\frac{2n\big(\frac{\alpha}{2}\big)_n}{\big(1-\frac{\alpha}{2}\big)_n}-\frac{\alpha }{2}-\frac{\big(-\frac{\alpha}{2}\big)\big(\frac{\alpha}{2}\big)_n}{\big(-\frac{\alpha}{2}\big)_{n}}\bigg)\overline{w}^{n+1}\\ &=& \frac{C_\alpha\Gamma(1-\alpha)}{2\Gamma^2(1-\alpha/2)}\sum_{n\geq 1}b_n\bigg(\frac{2n\big(\frac{\alpha}{2}\big)_n}{\big(1-\frac{\alpha}{2}\big)_n}-\frac{\alpha }{2}-\big(n-\frac{\alpha}{2}\big)\frac{\big(\frac{\alpha}{2}\big)_n}{\big(1-\frac{\alpha}{2}\big)_{n}}\bigg)\overline{w}^{n+1}\\
&=&\frac{C_\alpha\Gamma(1-\alpha)}{2\Gamma^2(1-\alpha/2)}\sum_{n\geq 1}b_n\bigg(-\frac{\alpha }{2}+\big(n+\frac{\alpha}{2}\big)\frac{\big(\frac{\alpha}{2}\big)_n}{\big(1-\frac{\alpha}{2}\big)_{n}}\bigg)\overline{w}^{n+1}\\
&=&-\frac{C_\alpha\Gamma(1-\alpha)}{2\Gamma^2(1-\alpha/2)}\sum_{n\geq 1}b_n\bigg(\frac{\alpha }{2}-\frac{\big(\frac{\alpha}{2}\big)_{n+1}}{\big(1-\frac{\alpha}{2}\big)_{n}}\bigg)\overline{w}^{n+1}\\ &\triangleq& -\sum_{n\geq 1}b_n\beta_n\overline{w}^{n+1},
\end{eqnarray*}
with
$$\beta_n=\frac{C_\alpha\Gamma(1-\alpha)}{2\Gamma^2(1-\alpha/2)}
\bigg(\frac{\alpha }{2}-\frac{\big(\frac{\alpha}{2}\big)_{n+1}}{\big(1-\frac{\alpha}{2}\big)_{n}}\bigg).
$$
Now by summing up  \eqref{i3} and \eqref{i4} we deduce that,
\begin{eqnarray*}
\hbox{I}_2(h(w))+\hbox{I}_4(h(w))&=&  \frac{\alpha C_\alpha\Gamma(1-\alpha)}{2(2-\alpha)\Gamma^2(1-\alpha/2)}\sum_{n\geq 1}b_n\bigg(2n+1+\frac{\alpha}{2}-\frac{(1+\frac{\alpha}{2})\big(2+\frac{\alpha}{2}\big)_n}{\big(2-\frac{\alpha}{2}\big)_n}\bigg)w^{n+1}\\&=&  \frac{\alpha C_\alpha\Gamma(1-\alpha)}{2(2-\alpha)\Gamma^2(1-\alpha/2)}\sum_{n\geq 1}b_n\bigg(2n+1+\frac{\alpha}{2}-\frac{\big(1+\frac{\alpha}{2}\big)_{n+1}}{\big(2-\frac{\alpha}{2}\big)_n}\bigg)w^{n+1}\\ &\triangleq& \sum_{n\geq 1}b_n\alpha_n{w}^{n+1},
\end{eqnarray*}
with
$$
\alpha_n\triangleq \frac{\alpha C_\alpha\Gamma(1-\alpha)}{2(2-\alpha)\Gamma^2(1-\alpha/2)}\bigg(2n+1+\frac{\alpha}{2}-\frac{\big(1+\frac{\alpha}{2}\big)_{n+1}}{\big(2-\frac{\alpha}{2}\big)_n}\bigg).
$$
Then inserting \eqref{i001} and   the two preceding identities into \eqref{df00} one can readily verify that
\begin{eqnarray}\label{df}
\partial_fF(\Omega,0)(h)(w)&=&\textnormal{Im}\bigg\{\Omega b_0\overline{w}- \sum_{n\geq 1}b_n\Big(n\Omega -\alpha_n\Big)w^{n+1}+\sum_{n\geq 1}b_n\Big(\Omega-\beta_n\Big)\overline{w}^{n+1} \bigg\}\notag\\ &=&\frac{\Omega b_0}{2} i\big(w-\overline{w}\big)+i\sum_{n\geq 1}\frac{b_n}{2}\Big(\big(n+1\big)\Omega-\big(\alpha_n+\beta_n\big)\Big)\Big(w^{n+1}-\overline{w}^{n+1}\Big).
\end{eqnarray}
By using \eqref{f1s} combined with the foregoing expressions for $\alpha_n$ and $\beta_n$ one may write, 
\begin{eqnarray}\label{iden1}
\alpha_n+\beta_n &=& \frac{\alpha C_\alpha\Gamma(1-\alpha)}{2(2-\alpha)\Gamma^2(1-\alpha/2)}\bigg(2n+2-\frac{\big(1+\frac{\alpha}{2}\big)_{n+1}}{\big(2-\frac{\alpha}{2}\big)_n}-\frac{(1-\frac{\alpha}{2})\big(\frac{\alpha}{2}\big)_{n+1}}{\frac{\alpha}{2}\big(1-\frac{\alpha}{2}\big)_{n}}\bigg)\notag\\ &=& \frac{\alpha C_\alpha\Gamma(1-\alpha)}{2(2-\alpha)\Gamma^2(1-\alpha/2)}\bigg(2n+2-\frac{\big(1+\frac{\alpha}{2}\big)_{n+1}}{\big(2-\frac{\alpha}{2}\big)_n}-\frac{\big(1+\frac{\alpha}{2}\big)_{n}}{\big(2-\frac{\alpha}{2}\big)_{n-1}}\bigg)\notag\\ &=& \frac{\alpha C_\alpha\Gamma(1-\alpha)}{(2-\alpha)\Gamma^2(1-\alpha/2)}\big(n+1\big)\bigg(1-\frac{\big(1+\frac{\alpha}{2}\big)_{n}}{\big(2-\frac{\alpha}{2}\big)_{n}}\bigg).
\end{eqnarray}
Coming back to the definition of $C_\alpha$, see for instance Proposition \ref{prop-bound}, and setting
$$
\Theta_\alpha\triangleq\frac{\alpha C_\alpha\Gamma(1-\alpha)}{(2-\alpha)\Gamma^2(1-\frac\alpha2)}=\frac{\alpha\Gamma(\frac\alpha2)\Gamma(1-\alpha)}{2^{1-\alpha}(2-\alpha)\Gamma^3(1-\frac\alpha2)},
$$
one finds  that
$$
\alpha_n+\beta_n=\Theta_\alpha\big(n+1\big)\bigg(1-\frac{\big(1+\frac{\alpha}{2}\big)_{n}}{\big(2-\frac{\alpha}{2}\big)_{n}}\bigg).
$$
Making appeal to the definition \eqref{disper00}, the linearized operator \eqref{df} takes the form,
\begin{equation}\label{dfz}
\partial_fF(\Omega,0)(h)(w)=\frac{\Omega b_0}{2} i\big(w-\overline{w}\big)+\frac12i\sum_{n\geq 1}\big(n+1\big){b_n}\Big(\Omega-\Omega_{n+1}^\alpha\Big)\Big(w^{n+1}-\overline{w}^{n+1}\Big).
\end{equation}
We should mention in passing that the linearized operator has a special structure: it acts as a Fourier multiplier and as we shall see this will be very useful in the explicit computations for the kernel and the range of this operator. Now let us look for the values of $\Omega$ corresponding to non trivial kernel. 
It is easy to see that this will be the case if and only if $\Omega$ belongs to the dispersion set $\mathcal{S}$ introduced in \eqref{disper}. This corresponds to the values of $\Omega$ such that there exists 
 $m\geq 1$ with
 \begin{eqnarray*}
\Omega &=&\Omega_{m+1}^\alpha \\ &=&\Theta_\alpha\bigg(1-\frac{\big(1+\frac{\alpha}{2}\big)_{m}}{\big(2-\frac{\alpha}{2}\big)_{m}}\bigg).
\end{eqnarray*}
From  Lemma \ref{lemz1}-$(2)$ the sequence $n\mapsto \Omega_n^\alpha$ is strictly increasing and therefore  for any $n\neq m$  
$$
(1+n)\big(\Omega_{m+1}^\alpha-\Omega_{n+1}^\alpha\big)\neq0.
$$
From these last facts it is apparent that the kernel of $\partial_fF(\Omega_{m+1}^\alpha,0)$  is one-dimensional vector space  generated by the function $v_m(w)=\overline{w}^m$.The claim of Proposition \ref{propgh} follows by shifting the index $m.$ \\ 

\vspace{0,2cm}

${\bf(2)}$ Now we are going to show that for any $m\geq 2$ the range $R(\partial_fF(\Omega_m^\alpha,0))$ coincides with the  subspace
$$
{Z_m}\triangleq\Big\{g\in C^{1-\alpha}(\mathbb{T});\quad g(w)=\sum_{n\geq1\atop\\\, n\neq m}^\infty ig_n(w^{n}-\overline{w}^{n}), g_n\in \RR\Big\}.
$$
Note that this sub-space is closed and of co-dimension one in the ambient  space $Y$. In addition, one may easily deduce  from \eqref{df} the trivial inclusion  $R(\partial_fF(\Omega_m^\alpha,0))\subset Z_m$ and therefore it remains to check just the converse. For this end, let $g\in Z_m$ we shall look for a pre-image $h(w)=\displaystyle{\sum_{n\geq 0}}b_n\overline{w}^n\in X$ satisfying $\partial_fF(\Omega_m^\alpha,0)(h)=g$. From the relation \eqref{dfz} this is equivalent to
$$
\frac{\Omega_m^\alpha}{2} b_0=g_1\quad\hbox{and}\quad  \frac{nb_{n-1}}{2}\Big(\,\Omega_m^\alpha-\Omega_n^\alpha\Big)=g_{n},\,\, \forall n\geq 2,\, n\neq m.
$$
This determines uniquely the sequence $(b_n)_{n\neq m-1}$ and  one has
$$
b_0=\frac{2 g_1}{\Omega_m^\alpha}\quad\hbox{and}\quad b_n=\frac{2g_{n+1}}{(n+1)\big(\Omega_m^\alpha-\Omega_{n+1}^\alpha\big)},\quad \forall n\neq m-1, n\geq1.
$$
However the value $b_{m-1}$ is free and it can be taken   zero.
Then the proof of  $h\in X$ reduces to show  that  $h\in C^{2-\alpha}\big(\mathbb{T}\big)$. For this end, it suffices to show that the function $ \displaystyle{H(w)=\sum_{n\geq m}b_n\overline{w}^n}$ belongs to this latter H\"{o}lder  space. First we shall transform $H$ in the form
\begin{eqnarray*}
H(w)&=&2\sum_{n\geq m}\frac{g_{n+1}}{(n+1)\big(\Omega_{m}^\alpha-\Omega_{n+1}^\alpha\big)}\, \overline{w}^n\\ &=&
   2w\sum_{ n\geq m+1}\frac{g_{n}}{n\big(\Omega_m^\alpha-\Omega_n^\alpha\big)}\overline{w}^n.
   \end{eqnarray*}
 Using \eqref{As1} one may write down
 \begin{eqnarray*}
H(w)&=& 2w\sum_{ n\geq m+1}\frac{g_{n}}{n\Big(\Omega_m^\alpha-\Theta_\alpha-d_n\Big)}\overline{w}^n,
\end{eqnarray*}
where
\begin{equation}\label{Ass11}
 d_n=-\Theta_\alpha\big(1-\alpha/2\big)\frac{e^{\alpha\gamma+c_\alpha}}{n^{1-\alpha}}+O\big(\frac{1}{n^{2-\alpha}}\big)\cdot
\end{equation}
Denote  $A=\Omega_m^\alpha-\Theta_\alpha$ then one may use  the general  decomposition: for $k\in \NN,$
$$
\frac{1}{A-d_n}=\frac{A^{-k-1}d_n^{k+1}}{A-d_n}+\sum_{j=0}^kA^{-j-1}{d_n^j}.
$$
This allows to rewrite  $H(w) $ in the manner
\begin{eqnarray*}
H(w)&=&2A^{-k-1}w\sum_{n\geq m+1}\frac{g_nd_n^{k+1}}{n(A-d_n)}\overline{w}^{n}+2w\sum_{j=0}^kA^{-j-1}\sum_{n\geq m+1}\frac{g_nd_n^{j}}{n}\overline{w}^{n}\\
&\triangleq&2A^{-k-1}w H_{k+1}(\overline{w})+2w\sum_{j=0}^kA^{-j-1}L_j(\overline{w}).
\end{eqnarray*}
Fix  $k$ such that $(1-\alpha)(k+1)>2$ then $H_{k+1}\in C^{2}(\mathbb{T}).$ Indeed
as the sequence $(g_n)_n$ is bounded then we get by \eqref{Ass11}
$$
\Big|\frac{g_nd_n^{k+1}}{n(A-d_n)}\Big|\lesssim \frac{|d_n|^{k+1}}{n}\lesssim\frac{1}{n^{1+(1-\alpha)(k+1)}}\cdot
$$
Therefore the regularity follows from the polynomial decay of the Fourier coefficients.
Concerning the estimate of $L_j$ we shall restrict the analysis to $j=0$ and $j=1$ and the higher  terms can be treated in a similar way. We write
\begin{eqnarray*}
L_0({w})&=&\sum_{n\geq m+1}\frac{g_n}{n}{w}^{n}.
\end{eqnarray*}
Using Cauchy-Schwarz  we deduce that
\begin{eqnarray*}
\Vert L_0\Vert_{L^\infty} &\lesssim & \sum_{n\geq m+1}\frac{\vert g_{n}\vert}{n}\\ &\lesssim & \bigg(\sum_{n\geq 1}\frac{1}{n^2}\bigg)^{1/2} \Big(\sum_{n\geq m+1}\vert g_{n}\vert^2\Big)^{1/2}\\ & \lesssim &\Vert g\Vert_{L^2}.
\end{eqnarray*}
Hence, by the embedding $ C^{1-\alpha}\big(\mathbb{T}\big)\hookrightarrow L^\infty\big(\mathbb{T}\big) \hookrightarrow L^2\big(\mathbb{T}\big)$ we conclude that
$$
\Vert L_0\Vert_{L^\infty} \lesssim \Vert g\Vert_{1-\alpha}.
$$
It remains to prove that $L_0^\prime\in C^{1-\alpha}(\mathbb{T})$. For this end one need first to check that one can differentiate   the series term by term. Fix $N\geq m+1$ and define
$$
L_0^N(w)\triangleq\sum_{n= m+1}^N\frac{g_n}{n}{w}^{n}.
$$
Then it is obvious from Cauchy-Schwarz inequality  that 
\begin{equation}\label{uni1}
\lim_{N\to\infty}\|L_0^N-L_0\|_{L^\infty(\mathbb{T})}=0.
\end{equation}
Now differentiating $L_0^N$ term by term one should get
\begin{eqnarray*}
(L_0^N)^\prime(w)&=&\overline{w}\sum_{n= m+1}^N{g_n}{w}^{n}\\
&\triangleq &\overline{w}\,  G_N(w).
\end{eqnarray*}
Assume for a while that  $\displaystyle{w\mapsto G(w)=\sum_{n\geq m+1}{g_n}{w}^{n}}$ belongs to $C^{1-\alpha}(\mathbb{T})$, then   by virtue of a classical result on Fourier series one gets
\begin{equation*}
\lim_{N\to\infty}\|G_N-G\|_{L^\infty(\mathbb{T})}=0
\end{equation*}
and consequently 
\begin{equation}\label{uni2}
\lim_{N\to\infty}\|(L_0^N)^\prime-\overline{w}\,  G\|_{L^\infty(\mathbb{T})}=0.
\end{equation}
Putting together \eqref{uni1} and \eqref{uni2} we obtain that $L_0$ is differentiable and
$$
L_0^\prime(w)=\overline{w}\,  G(w),\quad w\in \mathbb{T}.
$$
This concludes that $L_0\in C^{2-\alpha}.$ Now to complete rigorously the reasoning it remains to prove the preceding claim asserting that $G\in C^{1-\alpha}(\mathbb{T}).$ Actually,  this  is based on the continuity of Szeg\"o projection 
$$\Pi:\sum_{n\in \mathbb{Z}}a_n w^n\mapsto \sum_{n\in \mathbb{N}}a_n w^n
$$
on H\"{o}lder spaces $C^\varepsilon, \varepsilon\in (0,1).$ To see this we write
$$
G(w)=\Pi\big(-i\,g(w)-\sum_{n=0}^m g_n w^n\big).
$$
From which we  deduce that 
\begin{eqnarray}\label{zeg}
\nonumber\| G\|_{C^{1-\alpha}(\mathbb{T})}&\le &C\Big(\|g\|_{C^{1-\alpha}}+\sum_{n=0}^m|g_n|\|w^n\|_{C^{1-\alpha}}\Big)\\
\nonumber &\le&C_m\big(\|g\|_{C^{1-\alpha}}+\|g\|_{L^2}\big)\\
&\le&C_m\|g\|_{C^{1-\alpha}}
\end{eqnarray}
and this concludes the proof of the  claim. \\
As to  the term $L_1$ we write down by the definition 
\begin{eqnarray*}
L_1({w})&=&\sum_{n\geq m+1}\frac{g_n d_n}{n}{w}^{n}.
\end{eqnarray*}
As before we can easily get $L_1\in L^\infty$ and we shall  check that $L_1^\prime \in C^{1-\alpha}\big(\mathbb{T}\big).$ Arguing in a similar way to $L_0$ we can differentiate term by term the series defining $L_1$ leading to
$$
w\,L_1^\prime({w})=  \sum_{n\geq m+1}{ g_n d_n}{w}^{n}.
$$
We shall write down this series in the convolution form. With the notation  $w=e^{i\theta},$ we may write
\begin{eqnarray*}
w\,L_1^\prime({w})&=&(K* G)(w),\\
&=&\frac{1}{2\pi}\int_{0}^{2\pi}K(e^{i\eta})G(e^{i(\theta-\eta)})d\eta
\end{eqnarray*}
where
$$
K(w)\triangleq \sum_{n\geq m+1}{ d_n}{w}^{n}.
$$
Making use of  the definition \eqref{Ass11} we find the expansion 
\begin{eqnarray*}
K(w)
&=&-\Theta_\alpha\big(1-\alpha/2\big)e^{\alpha\gamma+c_\alpha}\sum_{n\geq m+1}\frac{w^n}{n^{1-\alpha}}+\sum_{n\geq m+1}O\big(\frac{1}{n^{2-\alpha}}\big) w^n.\\
&\triangleq&-\Theta_\alpha\big(1-\alpha/2\big)e^{\alpha\gamma+c_\alpha} K_1(w)+K_2(w).
\end{eqnarray*}
The second term is easy to analyze because we have an absolute series as follows,
$$
\|K_2\|_{L^\infty}\lesssim \sum_{n\geq m+1}\frac{1}{n^{2-\alpha}}\le C
$$
and therefore $K_2\in L^1(\mathbb{T})$. It suffices now to combine this fact  with the classical convolution law $L^1(\mathbb{T})*C^{1-\alpha}(\mathbb{T})\to C^{1-\alpha}(\mathbb{T})$ with \eqref{zeg}.
Next, we shall concentrate on the  first  term $K_1*G$  and prove that it belongs to $C^{1-\alpha}(\mathbb{T})$. For  this end  it is enough to show that  $K_1\in L^1(\mathbb{T})$ for $\alpha\in [0,1[$ which is more tricky. This claim is an immediate consequence of a more precise estimate: for any $\beta\in (\alpha,1)$
\begin{equation}\label{kern1}
|K_1(e^{i\theta})|\lesssim\frac{1}{\sin^\beta(\frac\theta2)},\quad\forall \theta\in (0,2\pi)\cdot
\end{equation}
This estimate sounds classical and for the convenience of the reader we shall give here a complete proof.
The basic tool is  Abel transform.  We set 
$$K_1^n(w)\triangleq\sum_{k=m+1}^n\frac{w^k}{k^{1-\alpha}}\quad\hbox{and}\quad U_n(w)\triangleq\sum_{k=0}^n w^k.
$$ 
Then it is apparent that
\begin{eqnarray*}
K_1^n(w)&\triangleq&\sum_{k=m+1}^n\frac{U_k(w)-U_{k-1}(w)}{k^{1-\alpha}}\\
&=&\sum_{k=m+1}^n\frac{U_k(w)}{k^{1-\alpha}}-\sum_{k=m}^{n-1}\frac{U_{k}(w)}{(1+k)^{1-\alpha}}\\
&=&\sum_{k=m+1}^{n-1}{U_k(w)}\Big(\frac{1}{k^{1-\alpha}}-\frac{1}{(1+k)^{1-\alpha}}\Big)+\frac{U_n(w)}{n^{1-\alpha}}-\frac{U_{m}(w)}{(m+1)^{1-\alpha}}\\
&\triangleq&K_{1,1}^n(w)+K_{1,2}^n(w)+K_{1,3}(w).
\end{eqnarray*}
The last term is bounded independently of $n$ and $w$. For the second term, it converges to zero as $n$ goes to infinity for any  $w\in \mathbb{T}\backslash\{1\}$. This follows easily from the estimate,
\begin{eqnarray*}
|K_{1,2}^n(w)|&\le& \frac{|1-w^{n+1}|}{|1-w|}\frac{1}{n^{1-\alpha}}\\
&\le&\frac{2}{|1-w|}\frac{1}{n^{1-\alpha}}\cdot
\end{eqnarray*}
As regards  the first term $K_{1,1}^n$, we shall use the mean value theorem  through the simple fact
$$
0\le\frac{1}{k^{1-\alpha}}-\frac{1}{(1+k)^{1-\alpha}}\lesssim\frac{1}{k^{2-\alpha}}\cdot
$$
Hence we get
$$
|K_{1,1}^n(w)|\lesssim\sum_{k=m+1}^{n-1}\frac{|U_k(w)|}{k^{2-\alpha}}\cdot
$$
Now we use the classical estimates
$$
|U_k(w)|\le k+1,\quad |U_k(w)|\le \frac{1}{|\sin\frac\theta2|},\,\, w=e^{i\theta}.
$$
By an obvious  convexity inequality we get for any $\beta\in [0,1]$
$$
|U_k(w)|\le \frac{(k+1)^{1-\beta}}{|\sin\frac\theta2|^{\beta}}
$$
and therefore
$$
|K_{1,1}^n(w)|\lesssim\frac{1}{|\sin\frac\theta2|^{\beta}}\sum_{k=m+1}^{n-1}\frac{1}{k^{1-\alpha+\beta}}\cdot
$$
The partial sum of the series converges  provided that we choose $\beta\in (\alpha,1)$. Collecting the preceding estimates and passing to the limit when $n$ goes to infinity we may write,
$$
|K_{1}(w)|\lesssim\frac{1}{|\sin\frac\theta2|^{\beta}}
$$
and this completes the proof of the inequality \eqref{kern1}. 

\vspace{0,3cm}

${\bf{(3)}}$ Now, we intend to  check the transversality assumption. According to the continuity property of the second derivative $\partial_\Omega\partial_f F$ seen in Proposition \ref{string11}, this assumption reduces to 
$$
\big\{\partial_\Omega\partial_f F(\Omega,0)(h)\big\}_{\Big|\Omega=\Omega_m^\alpha, h=v_m}\not\in R\big(\partial_fF(\Omega_m^\alpha,0)\Big).
$$
Differentiating \eqref{df00} with respect to $\Omega$ one gets
$$
\partial_\Omega\partial_fF(\Omega,0)(h)(w)=\textnormal{Im}\Big\{\overline{h'}(w)+\overline{w}{h(w)}\Big\}.
$$
Then obviously
$$
\partial_\Omega\partial_fF(\Omega_m^\alpha,0)(\overline{w}^{m-1})=i\frac{m}{2}\big(w^{m}-\overline{w}^{m}\big).
$$
which is not in the range of $\partial_fF(\Omega_m^\alpha,0)$ as it was described in the part $(2)$ of \mbox{Proposition \ref{propgh}.}
\end{proof}
\section{$m-$fold symmetry}  
Now we are ready to complete the proof of Theorem \ref{thmV2} started and developed throughout  the preceding sections. We have gathered all the required elements to apply  Theorem \ref{C-R} of Crandal-Rabinowitz. Combining  Proposition \ref{string11} and Proposition \ref{propgh} we deduce the existence of non trivial curves $\{\mathcal{C}_m, m\geq2\}$ bifurcating at the points $\Omega_m^\alpha$ of the dispersion set $\mathcal{S}$ introduced in \eqref{disper}. Each point of the branch $\mathcal{C}_m$ represents a V-state and we shall now see that it is an $m$-fold symmetric in a similar way to the case of the incompressible Euler equations.  This will be done by showing the bifurcation in spaces including the $m$-fold symmetry. To be more precise, let $m\geq2$ and define  the spaces $X_m$ and $Y_m$ as follows: the space $X_m$ is the set  of those functions $f\in X$ with a Fourier expansion of the type
$$
f(w)=\sum_{n=1}^{\infty}a_{nm-1}\overline{w}^{nm-1},\quad w\in\mathbb{T}.
$$
equipped with the usual strong  topology of $C^{2-\alpha}(\mathbb{T})$. We define the ball of radius $r\in(0,1)$ by
$$
B_r^m=\Big\{f\in X_m, \,\|f\|_{C^{2-\alpha}(\mathbb{T})}\le r\Big\}.
$$
If $f\in B_r^m$ the expansion of the associated conformal mapping $\phi$
in $\{z : \vert z\vert \geq 1\}$ is given by
$$
\phi(z)=z+f(z)=z\Big(1+\sum_{n=1}^{\infty}\frac{a_{nm-1}}{z^{nm}}\Big).
$$
This will provide the $m-$fold symmetry of the associated patch $\phi(\mathbb{T})$, via the relation
\begin{equation}\label{mfo1}
\phi\big(e^{i2\pi/m}z\big)=e^{i2\pi/m}\phi(z),\quad\vert z\vert\geq1.
\end{equation}
The space $Y_m$ is the subspace of $Y$ consisting of those $g \in Y$ whose Fourier expansion is of the type
$$
g(w)=i\sum_{n=1}^{\infty}g_{nm}\big(w^{nm}-\overline{w}^{nm}\big),\quad w\in\mathbb{T}.
$$
To apply Crandall-Rabinowitz's Theorem and get the symmetry property stated  in \mbox{Theorem \ref{thmV2}}  it suffices to show the following result.
\begin{proposition}\label{prozq} The following assertions hold  true. Let $m\geq2$ and $r\in (0,1)$, then 
\begin{enumerate}
\item $F:\mathbb{R}\times B_r^m\to Y_m$ is well-defined.
\item The kernel of $\partial_fF(\Omega_m^\alpha,0)$ is one-dimensional and generated by $w\mapsto \overline{w}^{m-1}$.
\item The range of  $\partial_fF(\Omega_m^\alpha,0)$ is closed in $Y_m$ and is of co-dimension one.
\end{enumerate}
\end{proposition} 
\begin{proof}

${\bf{(1)}}$ Let $f\in B_r^m$, we shall show that  $F(\Omega,f)=G(\Omega,\phi)\in Y_m$. Recall that the functional $G$ is defined by
$$
G(\Omega,\phi)(w)=\textnormal{Im}\bigg\{\Big(\Omega\,\overline{w}\,\phi(w)-\overline{w}\, S(\phi)(w)\Big)\,{\overline{\phi'(w)}}\bigg\},
$$
with
$$
S(\phi)(w)={C_\alpha}\mathop{{\fint}}_\mathbb{T}\frac{\phi'(\tau)}{\vert \phi(w)-\phi(\tau)\vert^\alpha}d\tau.
$$

It is easy to verify from \eqref{mfo1}  that the functions  $\phi'$ and $w\mapsto \frac{\phi(w)}{w}$ belong to the space $ C^{1-\alpha}(\mathbb{T})$ and  their   Fourier coefficients vanish at frequencies which are not integer multiples of $m$. Since this latter space is an algebra and stable by conjugation then the map $w\mapsto \textnormal{Im}\big\{\overline{\phi'(w)}\frac{\phi(w)}{w}\big\}$ belongs to the space $Y_m$.  Therefore, it remains to show that $w\mapsto \textnormal{Im}\big\{{\overline{\phi'(w)}}\big(\overline{w}S(\phi)(w)\big)\big\}\in Y_m$. This follows easily once  we have proved that $w\mapsto \overline{w}S(\phi)(w)$ satisfies \eqref{mfo1}.  For this end, set
\begin{eqnarray*} 
\Phi(w)&\triangleq& {\overline{w}}S(\phi)(w)\\
&=&{\overline{w}}\fint_{\mathbb{T}}\frac{\phi'(\tau)}{\vert \phi(w)-\phi(\tau)\vert^\alpha}d\tau.
\end{eqnarray*}
Then
\begin{eqnarray*}
\Phi(e^{i2\pi/m}w)={e^{-i2\pi/m}\overline{w}}\fint_{\mathbb{T}}\frac{\phi'(\tau)}{\vert \phi(e^{i2\pi/m}w)-\phi(\tau)\vert^\alpha}d\tau.
\end{eqnarray*}
By the change of variables $\tau=e^{i2\pi/m}\zeta$ and according to \eqref{mfo1}  we get for any $w\in\mathbb{T}$
\begin{eqnarray*}
\Phi(e^{i2\pi/m}w)&=&\frac{1}{w}\fint_{\mathbb{T}}\frac{\phi'(e^{i2\pi/m}\zeta)}{\vert \phi(e^{i2\pi/m}w)-\phi(e^{i2\pi/m}\zeta)\vert^\alpha}d\zeta\\ &=&\frac{1}{w}\fint_{\mathbb{T}}\frac{\phi'(\zeta)}{\vert \phi(w)-\phi(\zeta)\vert^\alpha}d\zeta\\
&=&\Phi(w).
\end{eqnarray*}
Hence, The Fourier coefficients of $\Phi$ vanish at frequencies which are not integer multiples of $m$ and this concludes the proof of the result,
$$
f\in B_r^m\Longrightarrow F(\Omega, f)=G(\Omega,\phi)\in Y_m.
$$

\vspace{0,2cm}

${\bf{(2)}}$
Since the generator $w \mapsto \overline{w}^{m-1}$ of the kernel of $\partial_fF(\Omega_m^\alpha,0)$ belongs to   $X_m$, we still have that the dimension of the kernel is $1$. 

\vspace{0,3cm}

${\bf{(3)}}$ Using Proposition \ref{propgh}, we deduce that
\begin{eqnarray*}
R\Big(\partial_f F (\Omega_m^\alpha, 0)\Big)&=&\Big\{g\in C^{1-\alpha}(\mathbb{T});\quad g(w)=i\sum_{n\geq1,n\neq m}^\infty g_n(w^{n}-\overline{w}^{n}), g_n\in \RR\Big\}\cap Y_m\\
&=&\Big\{g\in C^{1-\alpha}(\mathbb{T});\quad g(w)=i\sum_{n\geq2}^\infty g_{nm}(w^{nm}-\overline{w}^{nm}), g_{nm}\in \RR\Big\}.
\end{eqnarray*}
Obviously,  the range of $\partial_f F (\Omega_m, 0)$ is of co-dimension $1$ in the space $Y_m$.

Therefore we can apply Crandall-Rabinowitz’s Theorem to $X_m$ and $Y_m$ and 
obtain the existence of the $m-$fold symmetric patches for each integer $m \geq  2$. This achieves the proof of \mbox{Theorem \ref{thmV2}.}

\end{proof}

\section{Limiting case $\alpha=1$}\label{Limitc}
In this section we shall discuss the limiting case $\alpha=1$ corresponding to the SQG model. This case was excluded from Theorem \ref{thmV1} at least because   the rotating patch model seen in \eqref{model} does not work due to  the higher singularity of the kernel. Thus we shall modify a little bit this model as in \cite{C-F-M-R} and give an equation of the boundary of the V-states. Although the model seems to be coherent and satisfactory, it is  completely different from the sub-critical one $\alpha\in [0,1[$ and generates more technical difficulties in studying the rotating patches.  As we shall see later when we compute formally the linearized operator $\mathcal{L}_\Omega$ around the trivial solution we find that it behaves as a Fourier multiplier with an extra loss compared to the case $\alpha\in[0,1[$ which is of logarithmic type. Thus the property $\mathcal{L}_\Omega: C^{1+\varepsilon}\to C^{\varepsilon}$ fails and one should find other suitable spaces $X$ and $Y$ satisfying the assumptions of   C-R Theorem. We do  believe that such spaces must exist but certainly this would require more sophisticated analysis than what we shall do here. Among our objective is to    describe in details  the dispersion relation which tells us where the bifurcating curves emerge from the trivial one  and  also  shed light on the main difficulties encountered in this case.
\subsection{Rotating patch model} 
First recall from in the  equation \eqref{rotsq1}  one can change the velocity at the boundary by subtracting a tangential vector to the boundary without changing the full equation. Thus we shall work with the following modified velocity:
let $\gamma_0$ be a $2\pi$ periodic  parametrization of   the boundary of $D$, and define
\begin{equation}\label{veloc0}
u_0(\gamma_0(\sigma))=\frac{1}{2\pi}\int_0^{2\pi}\frac{\partial_s\gamma_0(s)-\partial_\sigma\gamma_0(\sigma)}{\vert \gamma_0(s)-\gamma_0(\sigma)\vert}ds.
\end{equation}
Then, by substituting the expression of the velocity in the equation of the boundary \eqref{rotsq1} one gets
\begin{equation}\label{model1}
\Omega\textnormal{Re}\big\{z\,\,\overline{z^\prime}\big\}=\textnormal{Im}\Big\{\frac{1}{2\pi}\int_{\partial D}\frac{\big(\zeta'-z'\big)}{\vert z-\zeta\vert^\alpha}\frac{d\zeta}{\zeta'}\,\,\overline{z^\prime}\Big\},\quad\forall z\in \partial D.
\end{equation}
Next, we parametrize the domain with the outside conformal mapping $\phi:\mathbb{D}^c\to {D}^c$,
\begin{equation}
\phi(z)=z+\sum_{n\geq 0}\frac{b_n}{z^n}
\end{equation}
by setting $z=\phi(w)$ and $\zeta=\phi(\tau)$. Then, we obtain the equation
\begin{equation}\label{modelq}
G(\Omega,\phi)(w)\triangleq\textnormal{Im}\Bigg\{\bigg(\Omega\phi(w)-\fint_\mathbb{T}\frac{\tau\phi'(\tau)-w\phi'(w)}{\vert \phi(w)-\phi'(\tau)\vert}\frac{d\tau}{\tau}\bigg)\frac{\overline{\phi^\prime(w)}}{w}\Bigg\}=0,\quad \forall w\in \mathbb{T},
\end{equation}
which is nothing but the boundary equation of the rotating patches. As for the sub-critical case we define,
$$
F(\Omega,f)(w)\triangleq G(\Omega,\textnormal{Id}+f)(w),\quad f(w)=\displaystyle{\sum_{n\geq 0}\frac{b_n}{w^n}},\,w\in \mathbb{T},\, b_n\in\RR.
$$
We point out that the Rankine vortices correspond to the trivial solutions $F(\Omega, 0)=0$. This can be checked as follows.
\begin{eqnarray*}
F(\Omega,0)(w)&=&\textnormal{Im}\bigg\{\Big(\Omega w-\fint_\mathbb{T}\frac{\tau-w}{\vert w-\tau\vert}\frac{d\tau}{\tau}\Big)\frac{1}{w}\bigg\}\\ &=& -\textnormal{Im}\bigg\{\frac{1}{w}\fint_\mathbb{T}\frac{\tau-w}{\vert w-\tau\vert}\frac{d\tau}{\tau}\bigg\}.
\end{eqnarray*}
In view of the  next identity \eqref{int1} applied with $n=1$ we conclude that
$$
F(\Omega,0)(w)=0,\quad \forall \Omega\in \RR,\, \forall w\in\mathbb{T}.
$$
We should mention in passing that we can get a similar result to the Proposition \ref{pro-el} and prove that the ellipses never rotate.
\subsection{Integral computations}
We shall discuss some elementary integrals that will appear later in the computations of the linearized operator.
\begin{lemma}\label{lem3}
Let $n\geq1$ and $w\in\mathbb{T}$, then we have
\begin{eqnarray}\label{int1}
\fint_\mathbb{T}\frac{{\tau}^{n}-{w}^n}{\vert w-\tau\vert}\frac{d\tau}{\tau} =-\frac{2{w}^{n}}{\pi}\sum_{k=0}^{n-1}\frac{1}{2k+1}\cdot
\end{eqnarray}
\begin{eqnarray}\label{int2}
\fint_\mathbb{T}\frac{(\tau-w)^2(\tau^n-w^n)}{\vert w-\tau\vert^3}\frac{d\tau}{\tau}=\frac{2w^{n+2}}{\pi}\sum_{k=1}^{n}\frac{1}{2k+1}\cdot
\end{eqnarray}
\end{lemma}
\begin{proof}
To prove \eqref{int1} we use successively the change of variables $\tau=w\zeta$ and $\zeta=e^{i\eta}$ 
\begin{eqnarray*}
\fint_\mathbb{T}\frac{{\tau}^{n}-{w}^n}{\vert w-\tau\vert}\frac{d\tau}{\tau} &=&{w}^{n}\fint_\mathbb{T}\frac{{\zeta}^{n}-1}{\vert 1-\zeta\vert}\frac{d\zeta}{\zeta}\\ &=&-{w}^{n}\frac{1}{2\pi}\int_0^{2\pi}\frac{1-e^{in\eta}}{\vert 1-e^{i\eta}\vert}d\eta\\ &=&-{w}^{n}\frac{1}{\pi}\int_0^{\pi}\frac{1-e^{i2n\eta}}{\vert 1-e^{i2\eta}\vert}d\eta.
\end{eqnarray*}
Using the identity  
$$\vert 1-e^{i2\eta}\vert=i(1-e^{i2\eta})e^{-i\eta},\quad \eta\in [0,\pi].
$$ we find easily 
\begin{eqnarray*}
\fint_\mathbb{T}\frac{{\tau}^{n}-{w}^n}{\vert w-\tau\vert}\frac{d\tau}{\tau}&=&-{w}^{n}\frac{1}{i\pi}\int_0^{\pi}\frac{1-e^{i2n\eta}}{1-e^{i2\eta}}e^{i\eta}d\eta\\ &=&-{w}^{n}\frac{1}{i\pi}\sum_{k=0}^{n-1}\int_0^{\pi}e^{i(2k+1)\eta}d\eta \\ &=&-{w}^{n}\frac{2}{\pi}\sum_{k=0}^{n-1}\frac{1}{2k+1}\cdot
\end{eqnarray*}
To compute the second integral we argue as before by using suitable change of variables,
\begin{eqnarray*}
\fint_\mathbb{T}\frac{(\tau-w)^2(\tau^n-w^n)}{\vert w-\tau\vert^3}\frac{d\tau}{\tau}&=&-w^{n+2}\fint_\mathbb{T}\frac{(\zeta-1)^2(1-\zeta^n)}{\vert 1-\zeta\vert^3}\frac{d\zeta}{\zeta}\\ &=&-w^{n+2}\fint_\mathbb{T}\frac{(1-\zeta)(1-\zeta^n)}{(1-\overline{\zeta})\vert 1-\zeta\vert}\frac{d\zeta}{\zeta}\\ &=&-w^{n+2}\frac{1}{2\pi}\int_0^{2\pi}\frac{(1-e^{i\eta})(1-e^{in\eta})}{(1-e^{-i\eta})\vert 1-e^{i\eta}\vert}d\eta.
\end{eqnarray*}
Thus we get
\begin{eqnarray*}
\fint_\mathbb{T}\frac{(\tau-w)^2(\tau^n-w^n)}{\vert w-\tau\vert^3}\frac{d\tau}{\tau} &=&-w^{n+2}\frac{i}{\pi}\int_0^{\pi}\frac{e^{i3\eta}(1-e^{i2n\eta})}{1-e^{i2\eta}}d\eta\\ &=&-w^{n+2}\frac{i}{\pi}\sum_{k=0}^{n-1}\int_0^{\pi}e^{i(2k+3)\eta}d\eta \\ &=&w^{n+2}\frac{2}{\pi}\sum_{k=0}^{n-1}\frac{1}{2k+3}\cdot
\end{eqnarray*}
This concludes the proof of the lemma.
\end{proof}

\subsection{Dispersion relation}
We shall now compute the G\^ateaux derivative of $F$ with respect to $f$  in the direction $h$, denoted as before by $\partial_fF(\Omega,f)h.$ Afterwards, we shall exhibit the dispersion set corresponding to the values of $\Omega$ where the kernel of the linearized operator around $f=0$ is non trivial.
Using \eqref{modelq}
\begin{eqnarray}\label{df0}
\partial_fF(\Omega,f)h(w)&=&\frac{d}{dt}F(\Omega,f+th)(w)_{\vert t=0}\notag\\ &=&
\textnormal{Im}\Bigg\{\Omega\Big(\phi(w)\overline{h'(w)}+\frac{\overline{\phi'(w)}}{w}h(w)\Big)\notag\\ &-&\frac{\overline{h'(w)}}{w}{\fint_{\mathbb{T}}}\frac{\tau\phi'(\tau)-w\phi'(w)}{\vert \phi(w)-\phi(\tau)\vert}\frac{d\tau}{\tau}-\frac{\overline{\phi'(w)}}{w}{\fint_{\mathbb{T}}}\frac{\tau h'(\tau)-wh'(w)}{\vert \phi(w)-\phi(\tau)\vert}\frac{d\tau}{\tau}\notag\\ &+& \frac{\overline{\phi'(w)}}{w}{\fint_{\mathbb{T}}}\frac{\big(\tau\phi'(\tau)-w\phi'(w)\big)\textnormal{Re}\Big\{{\big(\overline{h(\tau)}-\overline{h(w)}}\big)\big({\phi(\tau)-\phi(w)}\big)\Big\}}{\vert \phi(w)-\phi(\tau)\vert^3}\frac{d\tau}{\tau}\Bigg\}
\end{eqnarray} 
with the notation $\phi=\textnormal{Id}+f$.\\
\vspace{0,2cm}

In  the particular case $f=0$ one has
\begin{eqnarray}\label{df1}
\partial_fF(\Omega,0)h(w)&=&\frac{d}{dt}F(\Omega,th)(w)_{\vert t=0}\notag\\ &=&
\textnormal{Im}\bigg\{\Omega\big(\overline{h'(w)}+\overline{w}h(w)\big)-\fint_\mathbb{T}\frac{\tau h'(\tau)-wh'(w)}{w\vert w-\tau\vert}\frac{d\tau}{\tau}-\frac{\overline{h'(w)}}{w}\fint_\mathbb{T}\frac{\tau-w}{\vert w-\tau\vert}\frac{d\tau}{\tau}\notag\\ &+& \frac{1}{2w}\fint_\mathbb{T}\frac{h(\tau)-h(w)}{\vert w-\tau\vert}\frac{d\tau}{\tau}+ \frac{1}{2w}\fint_\mathbb{T}\frac{(\tau-w)^2\big(\overline{h(\tau)}-\overline{h(w)}\big)}{\vert w-\tau\vert^3}\frac{d\tau}{\tau}\bigg\}.
\end{eqnarray}
with $h(w)=\displaystyle{\sum_{n\geq 0}\frac{b_n}{w^n}}$ and $b_n\in\RR$ for all $n\geq 1$.

Our next goal is  to look for the values  $\Omega$ where the linearized operator fails to be injective.  We will be seeing that the function spaces that we shall use  differs from the ones of the case $\alpha\in(0,1).$ We will abandon the use of H\"older spaces which generate more technical difficulties. We introduce the spaces,

$$
B^s(\mathbb{T})=\Big\{f(w)=\sum_{n\geq 0}b_n\overline{w}^n, b_n\in \RR,\,\| f\|_{B^s}<\infty\Big\},\, \|f\|_{B^s}=|b_0|+\sum_{n\geq1} n^s|b_n|
$$
and
$$
B^s_{\textnormal{Log}}(\mathbb{T})=\Big\{ g(w)=i\sum_{n\geq 1}g_n\big(w^n-\overline{w}^n\big), g_n\in \RR,\,\| g\|_{B^s_{\textnormal{Log}}}<\infty\Big\},\, \|g\|_{B^s_{\textnormal{Log}}}=\sum_{n\geq1} \frac{n^s}{1+\ln n}|g_n|.
$$
  First we define the dispersion set
$$
\mathcal{S}_1=\Big\{\Omega= \frac{2}{\pi}\sum_{k=1}^{m-1}\frac{1}{2k+1}\quad m\geq 2\Big\}.
$$
 The main result of this section reads as follows.
\begin{proposition}\label{propgh0}
Let $s\geq1, \, m\geq2$.
\begin{enumerate}
\item For any $\Omega\in \mathbb{R},$ $\partial_fF(\Omega,0): B^s(\mathbb{T})\to B^{s-1}_{\textnormal{Log}}(\mathbb{T})$ is continuous.

\item
 The kernel  of $\partial_fF(\Omega,0)$ is non trivial if and only if \, $\Omega\in \mathcal{S}$
  and, in this case,  it is a one-dimensional vector space generated by
$$
v_m(w)=\overline{w}^{m-1}\quad\hbox{with}\quad \Omega=\Omega_m^1\triangleq \frac{2}{\pi}\sum_{k=1}^{m-1}\frac{1}{2k+1}\cdot
$$  
\item The range of  $\partial_fF(\Omega_m^1,0)$ is closed in $B^{s-1}_{\textnormal{Log}}(\mathbb{T})$ and is  of co-dimension one. It  is given by 
$$
R\big(\partial_fF(\Omega_m^1,0)\Big)=\Big\{g\in B^s_{\textnormal{Log}}(\mathbb{T}), \quad g(w)=i\sum_{n\geq1\atop\\ n\neq m}^\infty g_n(w^{n}-\overline{w}^{n}),\, g_n\in \RR\Big\}.
$$
\item Transversality assumption:   
$$
\partial_\Omega\partial_f F(\Omega_m^1,0)(v_m)\not\in R\big(\partial_fF(\Omega_m^1,0)\big).
$$

\end{enumerate}
\end{proposition}
Before giving the proof of this result, we should  make few comments.
\begin{remark}
\begin{enumerate}
\item The dispersion relation was discovered  in \cite{A-H} by using another analytical approach based on Bessel functions. The  proof that we shall present is different and is somehow elementary.
\item The spaces $B^s(\mathbb{T})$ and $B^{s-1}_{\textnormal{Log}}(\mathbb{T})$ introduced above are well-adapted to the study of the linear operator but it  is not at all clear whether the nonlinear function $F$ sends $B^s(\mathbb{T})$ into $B^{s-1}_{\textnormal{Log}}(\mathbb{T})$ and satisfies the regularity properties required by C-R Theorem. If this is the case then one can show the existence of the V-states for the SQG equation.
 \end{enumerate}
\end{remark}
\begin{proof}
{$\bf(1)-(2)$. } We shall prove in the same time the two points. We start with replacing $h$ and $\overline{h'}$ in the identity \eqref{df1} by their Fourier expansions,
$$
h(w)=\sum_{n\geq 0}b_n\overline{w}^n\quad \textnormal{and}\quad \overline{h'(w)}=-\sum_{n\geq 0}nb_n{w}^{n+1}.
$$
Therefore  we get
\begin{eqnarray*}
\partial_fF(\Omega,0)h(w)&=&\textnormal{Im}\bigg\{\Omega\sum_{n\geq0}b_n\big(\overline{w}^{n+1}-n w^{n+1}\big)+\sum_{n\geq 1}nb_n\overline{w}\fint_\mathbb{T}\frac{\overline{\tau}^{n}-\overline{w}^n}{\vert w-\tau\vert}\frac{d\tau}{\tau}\\ &+&\sum_{n\geq 1}nb_n w^{n}\fint_\mathbb{T}\frac{\tau-w}{\vert w-\tau\vert}\frac{d\tau}{\tau}+\frac{1}{2}\sum_{n\geq 1}b_n\overline{w}\fint_\mathbb{T}\frac{\overline{\tau}^{n}-\overline{w}^n}{\vert w-\tau\vert}\frac{d\tau}{\tau}\\ &+&\frac{1}{2}\sum_{n\geq 1}b_n\overline{w}\fint_{\mathbb{T}}\frac{(w-\tau)^2(\tau^n-w^n)}{\vert w-\tau\vert}\frac{d\tau}{\tau}\bigg\}\cdot
\end{eqnarray*}
Note that
\begin{equation}\label{intbar}
\fint_\mathbb{T}\frac{\overline{\tau}^{n}-\overline{w}^n}{\vert w-\tau\vert}\frac{d\tau}{\tau}=\overline{\fint_\mathbb{T}\frac{\tau^n-w^n}{\vert w-\tau\vert}\frac{d\tau}{\tau}},
\end{equation}
and consequently we can rewrite in view of Lemma \ref{lem3} the linear operator as follows,
\begin{eqnarray}\label{Qs1}
\nonumber\partial_fF(\Omega,0)h(w)&=&\textnormal{Im}\bigg\{\Omega\sum_{n\geq0}b_n\big(\overline{w}^{n+1}-n w^{n+1}\big)-\frac{2}{\pi}\sum_{n\geq 1}nb_n\overline{w}^{n+1}\sum_{k=0}^{n-1}\frac{1}{2k+1}\\
\nonumber &-&\frac{2}{\pi}\sum_{n\geq 1}nb_n w^{n+1}-\frac{1}{\pi}\sum_{n\geq 0}b_n\overline{w}^{n+1}\sum_{k=0}^{n-1}\frac{1}{2k+1}+\frac{1}{\pi}\sum_{n\geq 0}b_nw^{n+1}\sum_{k=1}^{n}\frac{1}{2k+1}\bigg\}\\ 
\nonumber &=&\frac{b_0\Omega}{2} i(w-\overline{w})+
\textnormal{Im}\bigg\{\sum_{n\geq1} b_n\Big(\Omega -\alpha_n\Big)\overline{w}^{n+1}-\sum_{n\geq1} b_n\Big(n\Omega -\beta_n\Big)w^{n+1}\bigg\}
 \\ &=& \frac{b_0\Omega}{2} i(w-\overline{w})+\frac{1}{2i}\sum_{n\geq1}(n+1)\Big(\Omega-\frac{\alpha_n+\beta_n}{n+1}\Big)b_n\big(\overline{w}^{n+1}-w^{n+1}\big),
\end{eqnarray}
where
$$
\alpha_n\triangleq \frac{2n+1}{\pi}\sum_{k=0}^{n-1}\frac{1}{2k+1}
$$
and
$$
\beta_n\triangleq -\frac{2n}{\pi}+ \frac{1}{\pi}\sum_{k=1}^{n}\frac{1}{2k+1}\cdot
$$
It is plain to see that
\begin{eqnarray*}
\alpha_n+\beta_n&=&\frac{2n+1}{\pi}\Big(1-\frac{1}{2n+1}+\sum_{k=1}^{n}\frac{1}{2k+1}\Big)-\frac{2n}{\pi}+ \frac{1}{\pi}\sum_{k=1}^{n}\frac{1}{2k+1}\\ &=& \frac{2(n+1)}{\pi}\sum_{k=1}^{n}\frac{1}{2k+1}\\
&\triangleq&(n+1)\Omega_{n+1}^1. 
\end{eqnarray*}
Inserting this formula into \eqref{Qs1} we obtain
\begin{eqnarray}\label{Lin11}
\partial_fF(\Omega,0)h(w)&=&\frac{b_0\Omega}{2} i(w-\overline{w})+\frac{1}{2} i \sum_{n\geq1}(n+1)\Big(\Omega-\Omega_{n+1}^1\Big)b_n\big({w}^{n+1}-\overline{w}^{n+1}\big).
\end{eqnarray}
To check that $\partial_fF(\Omega,0): B^s(\mathbb{T})\to  B^{s-1}_{\textnormal{Log}}(\mathbb{T})$ is continuous we  write
\begin{eqnarray*}
\big\| \partial_fF(\Omega,0)h\big\|_{B^{s-1}_{\textnormal{Log}}}&=&\frac12|b_0\Omega|+\frac12\sum_{n\geq1}\frac{(1+n)^{s}}{1+\ln(1+n)}\big|\Omega-\Omega_{n+1}^1\big|\,|b_n|\\
&\le&\frac12|b_0\Omega|+C\sum_{n\geq1}\frac{n^{s}}{1+\ln n}\big|\Omega-\Omega_{n+1}^1\big|\,|b_n|\\
&\le& C\Omega\|h\|_{B^s}+C\sum_{n\geq1}\frac{n^{s}}{1+\ln n}\sum_{k=1}^{n}\frac{1}{2k+1}\,|b_n|.
\end{eqnarray*}
To estimate the last term we shall use the asymptotic behavior of the harmonic series
\begin{equation}\label{harmonic}
\sum_{k=1}^n\frac{1}{2k+1}=\frac12\ln n+\frac12\gamma+\ln 2-1+O(\frac1n)
\end{equation}
which yields,
$$
\big\| \partial_fF(\Omega,0)h\big\|_{B^{s-1}_{\textnormal{Log}}}\le C\|h\|_{B^s}.
$$
This concludes the continuity of the linear operator $\partial_fF(\Omega,0): B^s(\mathbb{T})\to B^{s-1}_{\textnormal{Log}}.$

Now we shall study the kernel of this operator. 
From the formulae \eqref{Lin11} we immediately deduce  that the kernel of $\partial_fF(\Omega,0)$ is non trivial if and only if there exists 
 $m\geq 2$ such that
\begin{eqnarray*}
\Omega=\Omega_{m}^1 = \frac{2}{\pi}\sum_{k=1}^{m-1}\frac{1}{2k+1}\cdot
\end{eqnarray*}
 In which case the kernel contains   the eigenfunction $w\mapsto \overline{w}^{m-1}$.  Moreover, it is one-dimensional vector space because 
the sequence $n\mapsto \Omega_n$ is strictly increasing and therefore the Fourier coefficients in \eqref{Lin11} satisfy
$$
(1+n)\Big(\Omega_{m}^1-\Omega_{n+1}^1\Big)\neq0,\quad \forall\, n\neq m-1.
$$
This achieves the proof of a simple kernel.

${\bf{(3)}}$ Denote by
$$
Z_m\triangleq \Big\{g\in B^s_{\textnormal{Log}}(\mathbb{T}), \quad g(w)=i\sum_{n\geq1\atop\\ n\neq m}^\infty g_n(w^{n}-\overline{w}^{n})\Big\}.
$$
Clearly $Z_m$ is a closed subspace of $ B^s_{\textnormal{Log}}(\mathbb{T})$  and   from \eqref{Lin11} we deduce the obvious  embedding $R(\partial_fF(\Omega_m,0))\subset Z_m$. Thus  it remains to check the converse, that is, for any $g\in Z_m$  there exists  $\displaystyle{w\mapsto h(w)=\sum_{n\geq0} b_n\overline{w}^n \in B^s(\mathbb{T})}$ such that  $\partial_fF(\Omega_m,0)h=g.$ In terms of Fourier coefficients this is equivalent to 
$$
b_0\Omega_m^1=2 g_0,\quad n\Big(\Omega_m^1-\Omega_n^1\Big)b_{n-1}=2g_n,\quad \forall n\geq2, 
$$
This defines only one sequence $(b_n)_{n\neq m-1}$ and the coefficient $b_{m-1}$ is free. To check the regularity of $h$ it suffices to prove that   
$$w\mapsto H(w)=\sum_{n\geq m} b_n \overline{w}^n
\in B^s(\mathbb{T}).
$$
According to the definition of the norm of $B^s$ and \eqref{harmonic} one gets
\begin{eqnarray*}
\|H\|_{B^s}&=&\sum_{n\geq m} n^s|b_n|\\
&=&2\sum_{n\geq m} n^s\frac{|g_{n+1}|}{(1+n)(\Omega_{n+1}^1-\Omega_{m}^1)}\\
&\le& C\sum_{n\geq m} \frac{n^{s-1}}{\ln n}|g_{n+1}|\\
&\le& \|g\|_{B^s_{\textnormal{Log}}}.
\end{eqnarray*}
This completes the proof of $Z_m=R(\partial_fF(\Omega_{m}^1,0))$.

${\bf(4)}$  We shall now check the transversality assumption

$$
\partial_\Omega\partial_f F(\Omega_m^1,0)(v_m)\not\in R\big(\partial_fF(\Omega_m^1,0)\Big).
$$
Differentiating \eqref{Lin11} one gets
$$
\partial_\Omega\partial_fF(\Omega,0)(h)(w)=\textnormal{Im}\Big\{\overline{h^\prime(w)}+\overline{w}{h(w)}\Big\}.
$$
Then
$$
\partial_\Omega\partial_fF(\Omega_m^1,0)(v_m)=i\frac{m}{2}\big(w^{m}-\overline{w}^{m}\big),
$$
which is not clearly in the subspace  $Z_m=\partial_fF(\Omega_m,0)$ as it  was claimed. The proof of Proposition \ref{propgh0} is now completed.

\end{proof}
\begin{gracies}
\textnormal{The second  author has been partially funded by the ANR  project Dyficolti ANR-13-BS01-0003-01.}
\end{gracies}

\end{document}